\documentclass{article}
\usepackage[journal=JEMS,lang=british]{ems-journal-jems} 

\DeclareMathOperator{\R}{\mathbb{R}}
\DeclareMathOperator{\Q}{\mathbb{Q}}

\DeclareMathOperator{\N}{\mathbb{N}}

\DeclareMathOperator{\Co}{\mathbb{C}}

\usepackage{mathtools}	 
\usepackage{ mathrsfs}

\newcommand{\pnorm}[2][]{\left\lVert#2\right\rVert \ifx\\#1\\\else_{#1}\fi}
\DeclarePairedDelimiterX{\abs}[1]{\lvert}{\rvert}{#1}

\newcommand*{\lop}[1]{\mathcal{L}(#1)}
\newcommand*{\lopp}[1]{\mathcal{L}^+(#1)}
\newcommand*{\loppt}[1]{\mathcal{L}_1^+(#1)}

\newcommand*{\s}[1]{(#1_n)_{n\in\N}}
\newcommand*{\sk}[1]{(#1_k)_{k\in\N}}

\newcommand*{\lin}[1]{\operatorname{span}\{ #1\}}
\newcommand*{\fbs}[1]{\mathcal{F}C_b^\infty(#1)}
\newcommand*{\fbk}[1]{\mathcal{F}C_b^{#1}(B_W)}
\newcommand*{\fbkn}[1]{\mathcal{F}C_b^{#1}(B_W,n)}

\newcommand*{\partx}[1]{\partial_{d_#1}}
\newcommand*{\party}[1]{\partial_{e_#1}}

\newcommand*{\partxsq}[1]{\partial^2_{d_#1}}

\newcommand*{\defeq}{\mathrel{\vcenter{\baselineskip0.5ex \lineskiplimit0pt
			\hbox{\scriptsize.}\hbox{\scriptsize.}}}%
	=}
\newcommand*{\eqdef}{=\mathrel{\vcenter{\baselineskip0.5ex \lineskiplimit0pt
			\hbox{\scriptsize.}\hbox{\scriptsize.}}}}

\providecommand{\li}[1]{\lim\limits_{#1\to \infty}}
\newcommand*{\id}{\text{Id}}

\newcommand*{\spdx}{(\cdot,\cdot)}
\newcommand*{\normx}{\|\cdot\|}
\DeclarePairedDelimiterX{\norm}[1]{\lVert}{\rVert}{#1}

\newcommand*{\acore}{\mathcal{D}}
\newcommand*{\sccs}{(T_t)_{t\geq 0}}

\newcommand*{\sccsz}{(T_t^0)_{t\geq 0}}
\newcommand*{\sccr}{(R_\alpha)_{\alpha> 0}}

\newcommand*{\dm}{\,\mathrm{d}\mu}

\newcommand*{\Lphic}{(L^{\Phi},\fbk{\infty})}

\DeclareMathOperator{\tr}{tr}

\newcommand*{\p}{\mathbb{P}}

\newcommand{\Prop}[2][]{\mathbb{P} \ifx\\#1\\\else_{#1}\fi\!\left( #2\right)}
\newcommand{\PropPalm}[2][]{\mathbb{P} \ifx\\#1\\\else^{#1}\fi\!\left( #2\right)}
\newcommand{\borel}[1][]{\mathscr{B} \ifx\\#1\\\else_{#1}\fi}

\usepackage{csquotes}
\usepackage[utf8]{inputenc}	
\usepackage[T1]{fontenc}

\usepackage{todonotes}	 
\usepackage{makecell}

\usepackage[automark,headsepline]{scrlayer-scrpage}	
\usepackage{lscape} 
\usepackage[toc,page]{appendix}
\usepackage{mathtools}	 
\usepackage{ mathrsfs}
\usepackage{dsfont} 

\usepackage[final]{microtype}

\usepackage{cleveref}
\usepackage{bbm}
\usepackage[framemethod=tikz]{mdframed}	 
\usepackage{pifont}
\usepackage{xcolor}

\usepackage{tabularx, booktabs, multirow, array} 
\usepackage{rotating} 
\usepackage{diagbox}
\usepackage{nicefrac}
\theoremstyle{plain}
\newtheorem{thm}{Theorem}[section]
\newtheorem{lem}[thm]{Lemma}
\newtheorem{prop}[thm]{Proposition}
\newtheorem{cor}[thm]{Corollary}
\theoremstyle{definition}
\newtheorem{exa}[thm]{Example}
\newtheorem{rem}[thm]{Remark}
\newtheorem{defn}[thm]{Definition}
\newtheorem{assum}[thm]{Assumption}

\newtheorem*{ass}{Assumption}
\newenvironment{cond}[1]{\begin{ass}[#1] \def\@currentlabelname{(\upshape\boldmath\textbf{#1}}}{\end{ass}}
\numberwithin{equation}{section}

\begin{document}

\title{The essential m-dissipativity for degenerate infinite dimensional stochastic Hamiltonian systems and applications}
\titlemark{Essential m-dissipativity for stochastic Hamiltonian systems}



\emsauthor{2}{
	\givenname{Benedikt}
	\surname{Eisenhuth}
	\mrid{}
	\orcid{}}{S.~Contributor}
\emsauthor*{3}{
	\givenname{Martin}
	\surname{Grothaus}
	\mrid{}
	\orcid{}}{S.~Else}

\Emsaffil{2}{
	\pretext{}
\department{Department of Mathematics}
\organisation{University of Kaiserslautern-Landau}
\address{Erwin-Schr{\"o}dinger-Stra{\ss}e 48}
\zip{67663}
\city{Kaiserslautern}
\country{Germany}
\posttext{}
\affemail{eisenhuth@mathematik.uni-kl.de}}

\Emsaffil{3}{
	\pretext{}
\department{Department of Mathematics}
\organisation{University of Kaiserslautern-Landau}
\address{Erwin-Schr{\"o}dinger-Stra{\ss}e 48}
\zip{67663}
\city{Kaiserslautern}
\country{Germany}
\posttext{}
\affemail{grothaus@mathematik.uni-kl.de}}


%
%


\classification[35K57, 35J70]{ 47D07, 37K45}


\begin{abstract}
	We consider a degenerate infinite dimensional stochastic Hamiltonian system with multiplicative noise and establish the essential m-dissipativity on $L^2(\mu^{\Phi})$ of the corresponding Kolmogorov (backwards) operator. Here, $\Phi$ is the potential and $\mu^{\Phi}$ the invariant measure with density $e^{-\Phi}$ with respect to an infinite dimensional non-degenerate Gaussian measure. The main difficulty, besides the non-sectorality of the Kolmogorov operator, is the coverage of a large class of potentials. We include potentials that have neither a bounded nor a Lipschitz continuous gradient. The essential m-dissipativity is the starting point to establish the hypocoercivity of the strongly continuous contraction semigroup $(T_t)_{t\geq 0}$ generated by the Kolmogorov operator. By using the refined abstract Hilbert space hypocoercivity method of Grothaus and Stilgenbauer, originally introduced by Dolbeault, Mouhot and Schmeiser, we construct a $\mu^{\Phi}$-invariant Hunt process with weakly continuous paths and infinite lifetime, whose transition semigroup is associated with $(T_t)_{t\geq 0}$. This process provides a stochastically and analytically weak solution to the degenerate infinite dimensional stochastic Hamiltonian system with multiplicative noise. The hypocoercivity of $(T_t)_{t\geq 0}$ and the identification of $(T_t)_{t\geq 0}$ with the transition semigroup of the process leads to the exponential ergodicity. Finally, we apply our results to degenerate second order in time stochastic reaction-diffusion equations with multiplicative noise. A discussion of the class of applicable potentials and coefficients governing these equations completes our analysis.
\end{abstract}

\maketitle


\section{Introduction}\label{ch:Intro}
The objective of this article is to examine an infinite dimensional version of the degenerate stochastic Hamiltonian systems studied in \cite{dPL06,GS16, BG21-wp}, whereas we  allow multiplicative noise. Let $W\defeq U\times V$ be the Cartesian product of two infinite dimensional real separable Hilbert spaces $(U,\spdx_U)$ and $(V,\spdx_V)$, respectively.
Then, the degenerate infinite dimensional stochastic Hamiltonian system with multiplicative noise is 
described on $W$, by the following infinite dimensional nonlinear degenerate stochastic differential equation with multiplicative noise
\begin{equation}\label{eq:sde}
	\begin{aligned}
		\mathrm{d} X_t &=\,K_{21}Q_2^{-1}Y_t\,\mathrm{d}t \\
		\mathrm{d}Y_t &= \sum_{i=1}^\infty \party{i}K_{22}(Y_t)e_i\,\mathrm{d}t  - K_{22}(Y_t)Q_2^{-1}Y_t\,\mathrm{d}t-K_{12}Q_1^{-1}X_t\,\mathrm{d}t -K_{12}					D\Phi(X_t)\,\mathrm{d}t \\&\qquad+ \sqrt{2K_{22}(Y_t)}\,\mathrm{d}W_t.
	\end{aligned}
\end{equation}
The action of the associated Kolmogorov backwards operator $L^{\Phi}$, on the space of bounded smooth cylinder functions, denoted by $\fbs{B_W}$, is given as
\[
\begin{aligned}
	L^{\Phi}f(u,v)
	&\defeq \tr\left[K_{22}(v)\circ D_2^2f(u,v)\right]
	+ \sum_{i=1}^{\infty} (\party{i}K_{22}(v)D_2f(u,v),e_i)_V \\
	&\quad- (v,Q_2^{-1}K_{22}(v)D_2f(u,v))_V- (u,Q_1^{-1}K_{21}D_2f(u,v))_U\\
	&\quad-(D\Phi(u),K_{21}D_2f(u,v))_U+(v,Q_2^{-1}K_{12}D_1f(u,v))_V.
\end{aligned}
\]
Above, $K_{21}$ is a bounded linear operator from $V$ to $U$ and $K_{12}$ its adjoint. For a sufficiently smooth function $f$ we denote by $D_1f$ and $D_2f$ the gradient in the first and second component, respectively. The diffusion part is determined by the variable coefficient $K_{22}$, where  $K_{22}(v)$ is a bounded symmetric positive linear operator on $V$ for every $v\in V$. The stochastic force is governed by a cylindrical Wiener process $(W_t)_{t\geq 0}$ with values in $V$. Moreover, $D\Phi$ is the gradient of a potential $\Phi:U\rightarrow (-\infty,\infty]$ and $Q_1$, as well as $Q_2$, are the covariance operators of two centered non-degenerate Gaussian measures $\mu_1$ and $\mu_2$ on $U$ and $V$, respectively. The underlying Hilbert space is determined by the product measure $ \mu^{\Phi}\defeq e^{-\Phi}\,\mu_1\otimes\mu_2$ defined on the Borel $\sigma$-algebra $\borel(W)$. The partial derivatives $\partx{i}K_{22}$, $i\in\N$, are taken with respect to the orthonormal Basis $(e_i)_{i\in\N}$, diagonalizing the covariance operator $Q_2$. Regularity assumptions for $\Phi$ and suitable invariance properties for the coefficients ensure that $L^{\Phi}$ is well-defined on $\fbs{B_W}$.

The degeneracy of the Equation \eqref{eq:sde} corresponds to the degeneracy of $L^{\Phi}$ in the sense that the second order differential operator in the definition of $L^{\Phi}$ only acts in the second component.

Our examination of the degenerate infinite dimensional stochastic Hamiltonian system with multiplicative noise is focused on the essential m-dissipativity of the operator $(L^{\Phi},\fbs{B_W})$, which we show by using Assumption \nameref{ass:pot_app} in \Cref{theo:ess_diss_phi_Lfour}. The essential m-dissipativity of $(L^{\Phi},\fbs{B_W})$ is equivalent to the existence of a strongly continuous contraction semigroup $\sccs$ generated by the closure $(L^{\Phi},D(L^{\Phi}))$ of the operator $(L^{\Phi},\fbs{B_W})$. Consequently, for each $u_0\in D(L^{\Phi})$, the function $[0,\infty)\ni t\mapsto u(t)\defeq T_tu_0\in D(L^{\Phi})$ is the unique classical solution to the abstract Cauchy Problem
\begin{align*}
	\frac{\mathrm{d}}{\mathrm{d}t}u(t)=L^{\Phi}u(t),\quad u(0)=u_0,
\end{align*}
compare \cite[Theorem 3.1.12]{cauchy-problems}. We highlight that our approaches to establish essential m-dissipativity of $(L^{\Phi},\fbs{B_W})$ are applicable, if the variable diffusion coefficient $K_{22}$ and the potential $\Phi$ are not $C^{\infty}$-smooth. In this sense, our results complement those of \cite{BDPRS15}, where existence and uniqueness (for $C^{\infty}$-smooth coefficients) of solutions for a large class of highly degenerate Fokker-Planck-Kolmogorov equations for probability measures on infinite dimensional spaces has been established.

Under mild regularity assumption on the potential $\Phi$ and $K_{22}$, as well as reasonable block invariance properties of the coefficients, compare \Cref{sec:inf_langevin_general}, we use an integration by parts formula with respect to measures of type $\mu^{\Phi}$ to obtain dissipativity of $(L^{\Phi},\fbs{B_W})$ on $L^2(W;\mu^{\Phi})$. The essential m-dissipativity of $(L^{\Phi},\fbs{B_W})$ then follows, in view of the famous Lumer-Phillips theorem, if there is some $\lambda\in (0,\infty)$ such that
\begin{equation}\label{denserange}
	(\lambda-L^{\Phi})(\fbs{B_W})\quad\text{is dense in}\quad L^2(W;\mu^{\Phi}).
\end{equation}
The difficulty to establish this so-called dense range condition, is governed by the degeneracy of the operator, the infinite dimensionality of the problem and the regularity properties of the potential and the coefficients. 
In view of Assumption \nameref{ass:pot_app}, we approximate the potential $\Phi$ with a double sequence $(\Phi_n^m)_{n,m\in\N}$ and get  the existence of a constant $\lambda\in(0,\infty)$ independent of $m,n\in\N$, such that for each $g\in \fbs{B_W}$ there exists a function $f_{n,m}\in \mathcal{F}C_b^3(B_W)$ with
\begin{equation}\label{eq:g-Lg=f}
	\lambda f_{n,m}-L^{\Phi_n^m}f_{n,m}=g.
\end{equation}
We then provide an $L^4(W;\mu^{\Phi_n^m})$ first order regularity estimate for $f_{n,m}$, independent of $m,n\in\N$, which allows us to show \eqref{denserange}. The strategy to derive such $L^4(W;\mu^{\Phi_n^m})$ first order regularity estimates is inspired by the considerations
from \cite{DP00},\cite[Chapter 12.3]{dapratozabczyk} and \cite{dPL06}, where non-degenerate infinite dimensional and degenerate finite dimensional operators, both without variable diffusion coefficient, of this type have been studied. In the case of a potential with bounded gradient the essential m-dissipativity of the Kolmogorov operator corresponding to degenerate infinite dimensional Hamiltonian system, by means of an $L^2(W;\mu^{\Phi})$ first order regularity estimate, has been shown for additive noise in \cite{EG21} and for multiplicative noise in \cite{EG23}.
Other results, concerning the essential m-dissipativity of such degenerate Kolmogorov operators, are, to our knowledge, only available in finite dimensional situations, compare e.g.~\cite{dPL06,GS16, BG21-wp}.
However, concerning the essential m-dissipativity and even the essential self-adjointness of perturbed infinite dimensional Ornstein-Uhlenbeck operators, there are strong results, compare e.g.~\cite{DP00,dPL14, Lunardi15,LunardiOrnsteinInfinite,BigF22}. The essential m-dissipativity for generators associated to singular dissipative stochastic equations in Hilbert space was derived in \cite{DaRoe, Big22}. 

In \cite{Villani}, Villani developed hypocoercivity methods to provide and quantify convergence rates to equilibrium of non-coercive and, in this sense, degenerate diffusive equations. Inspired by the ideas of Villani and influenced by the methods from H\'erau in \cite{H05}, Dolbeault, Mouhot and Schmeiser developed an abstract hypocoercivity concept, compare \cite{DMS15}. They studied exponential convergence to equilibrium of non-coercive evolution equations in a general Hilbert space setting, by means of entropy methods. Although their results were fundamental and opened the door to studying a wide range of degenerate evolution equations, the authors failed to address domain issues that commonly arise when dealing with unbounded linear operators. The essential m-dissipativity of the generator governing the evolution equation on an appropriate core is crucial. Grothaus and Stilgenbauer's significant contribution, in \cite{GS14} and \cite{GS16}, was to incorporate these concepts into a Kolmogorov backwards setting while also taking domain issues into account.

It is the rigorous method of Grothaus and Stilgenbauer, in the following called the abstract hypocoercivity method, we use in \Cref{sec:hypo_appl} to establish hypocoercivity, with explicitly computable constants determining the exponential speed of convergence to equilibrium, for the semigroup corresponding to the degenerate  stochastic Hamiltonian system. We contribute by formulating assumptions on the coefficients and the potential, determining $L^{\Phi}$, under which the abstract hypocoercivity method is applicable. To check the sufficiency of these assumptions, we derive a general Poincar\'{e} inequality for measures of type $\mu^{\Phi}$ and a second order regularity estimate for infinite dimensional (perturbed) Ornstein-Uhlenbeck operators with possibly unbounded diffusion coefficient, see \Cref{poin_check} and \Cref{eq:regularity-estimate-two}, respectively. We cover situations in which $\Phi$ is not convex and the gradient of $\Phi$ merely exists in a suitable Sobolev space.
The results we present in this context are based on the already published articles \cite{EG21_Pr} and \cite{EG23}, where hypocoercivity for degenerate infinite dimensional stochastic Hamiltonian systems with multiplicative noise and potentials with bounded gradients has been established. 
Approaches to provide explicit exponential convergence rates for infinite dimensional degenerate dynamics can be found in \cite{Zimmer2017} and \cite{W17}.
Using coupling methods, the author of \cite{Zimmer2017} derived explicit contraction rates for degenerate and infinite dimensional diffusions in an $L^1$ Wasserstein distance. In \cite{W17}, the author established $L^2$-$L^4$ hypercontractivity (stronger notion than hypocoercivity) for stochastic Hamiltonian systems. The result is obtained by using of a dimension free Harnack inequality and coupling methods. 
However, both dynamics considered in \cite{Zimmer2017} and \cite{W17} are less general in terms of the allowed coefficients describing the dynamic and are limited to additive noise. Moreover, the assumptions on the nonlinearity in \cite{Zimmer2017} and \cite{W17} translate to Lipschitz continuity of $D\Phi$, which we do not need for our approach.

We further show that the strongly continuos contraction semigroup $\sccs$, generated by the Kolmogorov operator, is sub-Markovian and conservative. In the appendix \Cref{sec:process}, we explain how the analytic potential theoretic results, described by Beznea, Boboc and R\"{o}ckner in \cite{BBR06_Right_process}, guarantee existence of a right process whose transition semigroup is associated with $\sccs$. The process provides a martingale solution for the Kolmogorov operator $L^{\Phi}$ with respect to the equilibrium measure.

The identification of $\sccs$ and the transition semigroup enables us to derive $L^2$-exponential ergodicity of the process, provided $\sccs$ is hypocoercive. By equipping the state space $W$ with the weak topology and in presence of the assumptions described in \Cref{sec:process}, we are able to apply the abstract resolvent methods from \cite{BBR06} to construct a $\mu^{\Phi}$-invariant Hunt process $\mathbf{M}$ with weakly continuous paths, infinite lifetime and whose transition semigroup is associated with $\sccs$. The existence of a suitable core and a $\mu^{\Phi}$-nest of weakly compact sets is essential for this approach. We follow the lines of \cite{EG23} and \cite{EG21_Pr} and show that $\mathbf{M}$ provides a stochastically and analytically weak solution to \eqref{eq:sde}.

We apply the aforementioned results in the context of stochastic reaction-diffusion equations. Our analysis is based on \cite{dPL14}, where Lunardi and Da Prato studied maximal Sobolev regularity and m-dissipativity for second order elliptic partial differential equations in infinite dimensions to analyze stochastic reaction-diffusion equations. We translate the classic non-degenerate reaction-diffusion equations into our framework of degenerate infinite dimensional stochastic Hamiltonian systems. For both $U$ and $V$, in the context of equation \eqref{eq:sde}, we choose $L^2((0,1);\mathrm{d}\xi)$. In principle, we consider potentials of type
\[
\Phi:U\to (-\infty,\infty],\quad\text{with}\quad\Phi(u)\defeq \int_0^1\phi( u(\xi))\,\mathrm{d}\xi,\quad u\in U,
\]
where we assume that $\phi$ is continuously differentiable with at most polynomial growth and additional technical assumptions. We cover situations where the potential has unbounded and non Lipschitz gradient, compare \Cref{lem:appl_unb_pot_last}. Moreover, the operators $K_{22}$, $K_{21}$, $K_{12}$, $Q_1$ and $Q_2$ are determined by suitable powers of minus the second order derivative with Dirichlet boundary condition. The assumptions to obtain the essential m-dissipativity of the corresponding Kolmogorov operators, hypocoercivity of the semigroups and associated martingale respectively stochastically and analytically weak solutions with weakly continuous paths, are translated into inequalities in terms of the powers determining the coefficient operators.
This example highlights the strength of our results, as they are more general than the degenerate semi-linear infinite dimensional stochastic differential equations discussed in \cite{W17}.

The following enumeration summarizes the major achievements of this article and outlines our strategy. We emphasise that we are able to treat potentials $\Phi$ that have neither a bounded nor a Lipschitz continuous gradient.
\begin{enumerate}
	\item[•] We establish the essential m-dissipativity of $(L^{\Phi},\fbs{B_W})$ on $L^2(W;\mu^{\Phi})$. 
	\item[•] We apply the abstract hypocoercivity method to provide and quantify the exponential convergence rate to equilibrium of the strongly continuos contraction semigroup $\sccs$ generated by the closure of $(L^{\Phi},\fbs{B_W})$.
	\item[•] We give sufficient conditions, ensuring the existence of a $\mu^{\Phi}$-invariant Hunt process with infinite lifetime, weakly continuous paths, providing a stochastically and analytically weak solution to \eqref{eq:sde}. Hypocoercivity of $\sccs$ translates into $L^2$-exponential ergodicity of the processes.
	\item[•] We formulate degenerate second order in time  stochastic reaction-diffusion equations in the context of degenerate infinite dimensional stochastic Hamiltonian systems. They are analyzed based on the previous results.
\end{enumerate}

\section{Preliminaries}

Before we rigorously introduce the Kolmogorov operator associated to the degenerate infinite dimensional stochastic Hamiltonian system we collect some results we need for our analysis.
Let $X$ be an infinite dimensional real separable Hilbert space with inner product $\spdx_X$. Further, assume that $Q$ is an injective linear operator on $X$ in the space of symmetric non-negative operators with finite trace, in the following denoted by$\loppt{X}$. By \cite[{Theorem 1.12}]{dP06} we know that there exists a unique probability measure denoted by $N(0,Q)$ on $(X,\borel(X))$ such that its characteristic function is given via
\[
X\ni x \mapsto e^{-\frac{1}{2}(Qx,x)_X}\in\Co.
\]This measure is called the centred infinite dimensional non-degenerate Gaussian measure with covariance operator $Q\in\loppt{X}$.  For a detailed introduction into the concept of infinite dimensional Gaussian measures, we refer to \cite[{Chapter 1}]{dP06} or \cite[Section 2.1 ]{PR07}. 

For the rest of this section we fix a centred non-degenerate infinite dimensional Gaussian measure $\mu\defeq N(0,Q)$. Since $Q$ is symmetric and of trace class, there is an orthonormal basis $B_X=\s{e}$ of $X$ consisting of eigenvectors of $Q$ with corresponding positive eigenvalues $\s{\lambda}\in\ell^1(\N)$. Without loss of generality, we assume that the eigenvalues are decreasing to zero.
In this situation the following definition is useful.
\begin{defn}\label{def:finitely_based}
	For each $n\in\N$, define $X_n\defeq\lin{e_1,\dots,e_n}$ and denote the orthogonal projection from $X$ to $X_n$ by $P^X_n$, with the	corresponding coordinate map $p^X_n:X\to\R^n$. This means for all $x\in X$
	\[
	P^X_nx\defeq \sum_{k=1}^n (x,e_k)_X e_k\in X_n	\quad\text{ and }\quad
	p^X_nx \defeq \left((x,e_1)_X,\dots,(x,e_n)_X\right)\in \R^n.
	\]	
	By $\overline{p}^X_n:\R^n\to X_n$ we denote the canonical embedding of $\R^n$ into $X_n$, i.e.
	\[
	\overline{p}^X_ny\defeq \sum_{k=1}^n y_k e_k\in X_n,\quad y\in\R^n.
	\]
	If the underlying real separable Hilbert space and the corresponding orthonormal basis is clear from the context, we also write $P_n,p_n$ and $\overline{p}_n$ to avoid an overload of notation.
	Let $k\in\N\cup\lbrace\infty\rbrace$ and $a\in \left\{b,c\right\}$. Then we define for each $n\in\N$ spaces of cylinder functions by
	\[
	\begin{aligned}
		\mathcal{F}C_a^k(B_X,n)
		&	\defeq \{ \varphi\circ p^X_n:X\to\R \mid
		\text{ for some } \varphi\in C_a^k(\R^n) \}\quad\text{and} \\
		\mathcal{F}C(B_X)&\defeq \bigcup_{n\in\N} \mathcal{F}C(B_X,n). 
	\end{aligned}
	\]
\end{defn}

Some of the results below are closely related to \cite[Chapter 9 and 10]{dapratozabczyk}, \cite[Chapter 10 and 11]{PR07}, as well as \cite[Section 2]{dPL14} and \cite[Section 2]{Lunardi15}. Note that the last mentioned reference also deals with (Gaussian) Sobolev spaces on convex subsets of real separable Hilbert spaces.

First we state Fernique's theorem, which tells us that infinite dimensional Gaussian measures have exponential tails. We also formulate an important result about moments of Gaussian measures.

\begin{prop}{\cite[Proposition 1.13]{dP06}}\label{theo:fernique} Given $s\in\R$. It holds 
	\[
	\int_X e^{s \norm{x}_X^2}\,\,\mu(\mathrm{d}x)=\begin{cases}
		\left(\prod_{k=1}^{\infty}(1-2s\lambda_k)\right)^{-\frac{1}{2}}&\quad\text{for}\quad s<\frac{1}{2\lambda_1}\\
		\infty&\quad\text{else}.
	\end{cases}
	\]
\end{prop}

\begin{lem}{\cite[Lemma 2.2]{EG21}}\label{lem:Gaussianmoments} For $l_1,l_2,l_3,l_4\in X$ set $q_{ij}=(Ql_i,l_j)_X$, $i,j\in \lbrace1,2,3,4\rbrace$. Then it holds
	\begin{align*}
		&\int_X (x,l_1)_X\,\mu(\mathrm{d}x)=0,\quad \int_X (x,l_1)_X(x,l_2)_X\,\mu(\mathrm{d}x)=q_{12},\\
		&\int_X (x,l_1)_X(x,l_2)_X(x,l_3)_X\,\mu(\mathrm{d}x)=0\quad \text{and} \\&\int_X (x,l_1)_X(x,l_2)_X(x,l_3)_X(x,l_4)_X\,\mu(\mathrm{d}x)=q_{12}q_{34}+q_{13}q_{24}+q_{14}q_{23}.
	\end{align*}
	Moreover, for all $s\in [0,\infty)$, $X\ni x\mapsto\|x\|_X^s\in\R$ is $\mu$-integrable and particularly 
	\[
	\int_X \|x\|_X^2\,\mathrm{d}\mu
	= \int_X \sum_{n\in\N} (x,e_n)_X^2 \,\mathrm{d}\mu
	= \sum_{n\in\N} (Qe_n,e_n)_X
	= \sum_{n\in\N} \lambda_n <\infty
	\]
	due to monotone convergence and the fact that $Q$ is trace class.
\end{lem}

Before we introduce several Sobolev spaces with respect to the infinite dimensional non-degenerate Gaussian measure we quickly recall some concepts of G\^{a}teaux and  Fr\'echet differentiability as discussed in \cite[Chapter 1]{ambrosetti1995primer}. In order to do that we fix a real normed vector space $(Y,\normx_Y)$ and an open set $U\subseteq X$ and a function $f:U\to Y$.

Below, we introduce $n$-times G\^{a}teaux and Fr\'echet differentiable functions.
\begin{defn}\label{def:der_higher_order}
	Let $f :U\to Y$ be G\^{a}teaux (Fr\'echet) differentiable.  If $Df \;(df):U\to \lop{X;Y}$ is G\^{a}teaux (Fr\'echet) differentiable, we call $f$ two times G\^{a}teaux (Fr\'echet) differentiable and denote the second order G\^{a}teaux (Fr\'echet) derivative by
	\[
	D^2f\; (d^2f):U\to \lop{X;\lop{X;Y}}.
	\]
	Inductively, this construction generalizes to higher order G\^{a}teaux (Fr\'echet) derivatives. The space of (bounded) $n$-times Fr\'echet differentiable functions, $n\in\N$, from $U$ to $Y$ with continuous (and bounded) derivatives up to order $n$, is denoted by $C^n(U;Y)$ ($C_b^n(U;Y)$). If $Y=\R$ we just write $C^n(U)$ ($C_b^n(U)$).  Since for $f\in C^1(U;Y)$ the G\^{a}teaux and Fr\'echet derivative coincide, we sometimes just call $f$ continuously differentiable and $Df$ its derivative. 
\end{defn}
\begin{rem}\label{rem_der_cyl_X}
	Assume that $f:U\rightarrow \R$ is G\^{a}teaux differentiable. For $u\in U$
	the Riesz representation theorem allows us to identify $Df(u)\in \lop{X;\R}$ with the gradient of $\nabla f(u)\in X$, i.e. with the unique element such that
	\[
	Df(u)(v)=(\nabla f(u),v)_X\quad\text{for all}\quad v\in X. 
	\]
	Analogously, for a two times G\^{a}teaux differentiable function $f:U\rightarrow \R$, we identify $D^2f(u)\in \lop{X;\lop{X;\R}}$ with the unique element $\nabla^2 f(u)\in \lop{X}$ such that
	\[
	D^2f(u)(v)(w)=(\nabla^2 f(u)(v),w)_X\quad\text{for all}\quad v,w\in X. 
	\]
	It holds that $Df(x)=\sum_{i\in\N}\party{i}f(x) e_i$ for all G\^{a}teaux differentiable functions $f:X\to\R$ and $x\in X$. Let $n\in\N$ be given. If $f=\varphi\circ p_n$ for some $\varphi\in C_b^1(\R^n)$, then the chain rule implies $Df(x)=\sum_{i=1}^n \partial_i\varphi(p_n(x))e_i\in 	X_n$ for all $x\in X$.
\end{rem}

\begin{thm}{\cite[Theorem 2.9]{EG21}}\label{theorem:Gaussian_Sobolev_construction} Let $p\in (1,\infty)$ be given. Moreover, let $(A,D(A))$ with $\lin{e_1,e_2,...}\subseteq D(A)$ and $\lin{e_1,e_2,...}\subseteq D(A^*)$ be a linear operators on $X$. Assume that for each $n\in\N$ there is some $m\in\N$ with 
	\[
	A^*(\lin{e_1,...,e_n})\subseteq \lin{e_1,...,e_m}.
	\]Then the operator
	\[
	\begin{aligned}
		AD:\mathcal{F}C_b^1(B_X)&\rightarrow L^p(X;\mu;X)\quad\text{and}\\
	\end{aligned}
	\]is closable in $L^p(X;\mu)$. We denote by $W_{A}^{1,p}(X;\mu)$ the domain of the closure. If $A=\text{Id}$, we simply write $W^{1,p}(X;\mu)$. We equip this space with the corresponding Graph norm. Here and in the following, $L^p(X;\mu;X)$ is defined in the Bochner-Lebesgue sense.
\end{thm}

\Cref{theorem:Gaussian_Sobolev_construction} is an extension of \cite[Lemma 2.3]{dPL14}, where only the case $A=Q^{\theta}$ for $\theta\in \{-\frac12,0,\frac12\}$ is considered. Moreover, note that for $A=Q^{\frac 12}$ the corresponding Sobolev space coincides with the usual Sobolev space of Malliavin calculus.

By \Cref{theorem:Gaussian_Sobolev_construction}, it is reasonable to consider the infinite dimensional Gaussian Sobolev space $W_{Q^{\theta}}^{1,p}(X;\mu)$ for all $\theta\in\R$.
Moreover, for $f\in W^{1,p}_{Q^{\theta}}(X;\mu)$ and $n\in\N$ we set
\[
\party{i} f\defeq(Q^{\theta}Df,e_i)_X\frac{1}{\lambda_i^{\theta}}\in L^p(X;\mu)\quad \text{and}\quad P_nDf\defeq \sum_{i=1}^n\party{i}f e_i\in L^p(X;\mu).
\]
Note that $\party{i} f_m$ converges to $\party{i} f$ in $L^p(X;\mu)$ if $(f_m)_{m\in\N}\subseteq \mathcal{F}C_b^1(B_X)$ is a sequence converging to $f$ in $ W^{1,p}_{Q^{\theta}}(X;\mu)$.

Next, we state the integration by parts formula for infinite dimensional Gaussian measures. For $\theta\in  \{-\frac12,0,\frac12\}$ the statement can be found in \cite{dPL14}. The proof for general $\theta\in \R$ works similarly and is therefore omitted below.
\begin{prop}\label{lem:inf-dim-ibp-fbs}
	Let $\theta\in \R$, $p,q\in (1,\infty)$ with $\frac{1}{p}+\frac{1}{q}\leq 1$ and $f\in W^{1,p}_{Q^{\theta}}(X;\mu)$, $g\in W^{1,q}_{Q^{\theta}}(X;\mu)$.
	Then
	\[
	\int_X \party{i}f g\,\mathrm{d}\mu= -\int_X f\party{i}g\,\mathrm{d}\mu+ \int_X (x,Q^{-1}e_i)_X fg\,\mathrm{d}\mu.
	\]
\end{prop}

The remark below, helps to to determine if the pointwise limit of a $\mu$-a.e.~convergent sequence in $W_{Q^{\theta}}^{1,p}(X;\mu)$,  $\theta\in [0,\infty)$ and $p\in (1,\infty)$, is again in $W_{Q^{\theta}}^{1,p}(X;\mu)$.
\begin{rem}\label{rem:Banach-Saks}
	%
	Let $\theta\in [0,\infty)$ and $p\in (1,\infty)$ be given. Similar to \cite[Lemma 5.4.4]{Bogachev}, we can establish that $W_{Q^{\theta}}^{1,p}(X;\mu)$ has the Banach-Saks property and for each bounded sequence $(f_n)_{n\in\N}$ in $W_{Q^{\theta}}^{1,p}(X;\mu)$ with $\lim_{n\to\infty}f_n=f$, $\mu$-a.e., it holds $f\in W_{Q^{\theta}}^{1,p}(X;\mu)$.
\end{rem}
Next, we recall the Moreau-Yosida approximation, which provides a useful approximation scheme. Before, we quickly state some results about subdifferential functions. Suppose $\Phi:X\to\R\cup\left\{\infty\right\}$ is convex, bounded from below, lower semicontinuous and not identically to $\infty$, then for each $x\in X$, the set
\[\partial \Phi(x)\defeq \{y\in X\mid \text{for all}\; z\in X\;\text{it holds}\; (z-x,y)_X+\Phi(x)\leq \Phi(z)\},\]
denotes the subdifferential of $\Phi$ in $x$.
For each $x\in X$, where $\Phi(x)=\infty$ we set $D_0\Phi(x)\defeq \infty$. For all $x\in X$ with $\Phi(x)\neq \infty$ the set $\partial \Phi(x)$ is closed and convex, compare \cite[Proposition 16.4]{yoshida}. In particular, for such $x$ it is reasonable to define $D_0\Phi(x)$ as the element in $\partial \Phi(x)$ with minimal norm if $\partial \Phi(x)\neq \emptyset$ and $\infty$ otherwise. 

\begin{exa}\label{ex:yoshida_one}Suppose $\Phi:X\to\R\cup\left\{\infty\right\}$ is not identical to $\infty$, convex, bounded from below and lower semicontinuous.
	For such functions, the so-called Moreau-Yosida approximation $\Phi_{t}$, $t>0$, is defined by
	\begin{align*}
		\Phi_{t}:X\rightarrow \R,\quad	\Phi_{t}(y)=\inf_{x\in X}\left\{ \Phi(x)+\frac{\norm{y-x}_X^2}{2t} \right\}.
	\end{align*}One can show that for all ${t>0}$, $\Phi_{t}$ is convex and Fr\'{e}chet differentiable with
	\begin{enumerate}
		\item[(i)] $-\infty<\inf_{y\in X}\Phi(y)\leq\Phi_{t}(x)\leq \Phi(x)$ for all $x\in X$.
		\item[(ii)] $\lim_{t\rightarrow 0}\Phi_{t}(x)=\Phi(x)$ for all $x\in X$.
		\item[(iii)] $D\Phi_{t}$ is Lipschitz continuous and for all $x\in X$ with $\partial\Phi(x)\neq\emptyset$, $\norm{D\Phi_t(x)}_X$ converges monotonically to $\norm{D_0\Phi(x)}_X$ with
		\[\norm{D\Phi_t(x)-D_0\Phi(x)}_X^2\leq \norm{D_0\Phi(x)}_X^2-\norm{D\Phi_t(x)}_X^2. \]
	\end{enumerate}
	A proof of these statements, except the convergence result of in Item (iii), is given in \cite{yoshida}. The statement in Item (iii), is shown in \cite[Chapter 2]{brzis1973oprateurs}.
\end{exa}

We continue the analysis of Sobolev spaces with respect to infinite dimensional Gaussian measures with a very useful approximation result. 
\begin{lem}{\cite[Lemma 2.2]{Lunardi15}}\label{ex:yoshida_two} Suppose $\Phi:X\to(-\infty,\infty]$ is as in \Cref{ex:yoshida_one}, i.e.~convex, bounded from below, lower-semicontinuous and not identically to $\infty$. If $ x\mapsto \norm{D_0\Phi}_X\in L^{p_1}(X;\mu)$ for some $p_1\in (1,\infty)$, then for each $1\leq p_0<p_1$
	\begin{enumerate}
		\item[(i)] $D\Phi=D_0\Phi$, $\mu$-a.e..
		\item[(ii)]$\Phi\in W_{Q^{\theta}}^{1,p_0}(X;\mu)$  and  $\lim_{t\rightarrow 0}\Phi_{t}=\Phi$ in $W_{Q^{\theta}}^{1,p_0}(X;\mu)$ for all $\theta\in [0,\infty)$.
	\end{enumerate}
\end{lem}
%

In the last part of this section, we generalize some of the above constructions and results to the case where the infinite dimensional Gaussian measure is additionally equipped with a density.
Without further mentioning we consider potential functions $\Phi: X\rightarrow (-\infty,\infty]$, as described in the subsequent definition.
\begin{defn}\label{defn:Phi_gen_def}
	Suppose $\Phi: X\rightarrow (-\infty,\infty]$ is measurable, bounded from below and such that $\int_Xe^{-\Phi}\,\mathrm{d}\mu>0$. For such 		$\Phi$, we consider the measure $\mu^{\Phi}\defeq \frac{1}{\int_Xe^{-\Phi}\,\mathrm{d}\mu}e^{-\Phi}\mu$ and set $\mu^0\defeq \mu$, as 		well as
	\[
	\mu^{\Phi}(f)\defeq \int_Xf\,\mathrm{d}\mu^{\Phi}\quad\text{for},\quad f\in L^1(X;\mu^{\Phi}).
	\]
\end{defn}

\begin{lem}{\cite[Lemma 2.2]{dPL14}}\label{lem:dense_mu_one_with_pot} Suppose $p\in[1,\infty)$. Then, $L^p(X;\mu)\subseteq L^p(X;\mu^{\Phi})$ and the space of smooth cylinder functions $\mathcal{F}C_b^{\infty}(B_X)$ is dense in $L^p(X;\mu^{\Phi})$.
\end{lem}

As in \cite[Chapter~2.2]{dPL14}, we can extend the integration by parts formula for measures of type $\mu^{\Phi}$ and draw some consequences summarized in the proposition below.
\begin{lem}\label{lem:intbp_classic_pot}
	Let $\theta\in \R$ and $\Phi\in  W^{1,2}_{Q^{\theta}}(X;\mu)$. Then, for $f,g\in \mathcal{F}C_b^1(B_X)$ and $i\in\N$, it holds the integration by parts formula
	\begin{equation}\label{eq:intBP_general}
		\int_X \party{i}fg\,\mu^{\Phi}=-\int_X f\party{i}g\,\mathrm{d}\mu^{\Phi}+\int_X (x,Q_1^{-1}e_i)_X fg\,\mathrm{d}\mu^{\Phi}+\int_X\party{i}\Phi fg\,\mathrm{d}\mu^{\Phi}.
	\end{equation}
\end{lem}

\begin{prop}{\cite[Lemma 2.3 and Remark 2.6]{dPL14}}\label{prop:IBP_PHI_most_gen} Let $\theta\in \R$ and $p\in[2,\infty)$ be given. Further, assume that $\Phi\in W^{1,2}_{Q^{\theta}}(X;\mu)$.  Then the following statements hold true.
	\begin{enumerate}
		\item[(i)]  The operator\[
		\begin{aligned}
			Q^{\theta}D:\mathcal{F}C_b^1(B_X)\rightarrow L^p(X;\mu^{\Phi};X)
		\end{aligned}
		\]is closable in $L^p(X;\mu^{\Phi})$. Therefore, it is reasonable to consider the Sobolev space $W^{1,p}_{Q^{\theta}}(X;\mu^{\Phi})$. Again we use the abbreviation $W^{1,p}(X;\mu^{\Phi})\defeq W^{1,p}_{\text{Id}}(X;\mu^{\Phi})$.
		
		\item[(ii)]It holds $W^{1,p}_{Q^{\theta}}(X;\mu)\subseteq W^{1,p}_{Q^{\theta}}(X;\mu^{\Phi})$.
		\item[(iii)]For each $q\in [2,\infty)$ with $\frac{1}{p}+\frac{1}{q}\leq \frac{1}{2}$ the integration by parts formula \eqref{eq:intBP_general} is valid for $f\in W^{1,p}_{Q^{\theta}}(X;\mu^{\Phi})$ and $g\in W^{1,q}_{Q^{\theta}}(X;\mu^{\Phi})$.
	\end{enumerate}
\end{prop}

Next, we introduce an important class of Potential functions $\Phi$. This class plays an important role in \Cref{ch:example_Lfour_ess_diss}. For that, assume \[(X,\spdx_X)=(L^2((0,1);\mathrm{d}\xi),\spdx_{L^2(\mathrm{d}\xi)}),\]where $\mathrm{d}\xi$ denotes the classical Lebesgue measure on $((0,1),\borel(0,1))$. In addition, we fix a continuous differentiable function $\phi:\R\rightarrow \R$, which is bounded from below and such that its derivative grows at most of order $b\in [0,\infty)$, i.e.~there exists $a\in (0,\infty)$ such that
\[\abs{\phi'(x)}\leq  a(1+\abs{x}^{b})\quad \text{for all}\quad x\in\R.\]
Using the mean value theorem, it is easy to check that there exists some $\tilde{a}\in (0,\infty)$ such that
\[\abs{\phi(x)}\leq  \tilde{a}(1+\abs{x}^{b+1})\quad \text{for all}\quad x\in\R.\]Therefore, $\phi$ grows at most of order $b+1$. For such $\phi$ it is reasonable to define \[\Phi:X\rightarrow (-\infty,\infty],\; x\mapsto \Phi(x)\defeq \begin{cases}
	\int_0^1\phi( x(\xi))\,\mathrm{d}\xi & \;\text{if}\; x\in L^{b+1}((0,1);\mathrm{d}\xi)\\
	\infty & \;\text{else}.
\end{cases}\] 
\begin{rem}\label{rem:pot_lower_semi_cont}
	Suppose $(x_n)_{n\in\N}$ is a sequence in $X$ converging to some element $x\in X$. Since $\phi$ is bounded from below, the same applies to $\Phi$. Hence, $\inf_{n\in\N}\Phi(x_n)\in (-\infty,\infty]$ and there exists a subsequence $(x_{n_k})_{k\in\N}$, such that $\lim_{k\to\infty}\Phi(x_{n_k})=\inf_{n\in\N}\Phi(x_n)$. We also find a subsubsequence $(x_{n_{k_i}})_{i\in\N}$ converging to $x$ pointwisely $\mathrm{d}\xi$-a.e.. Suppose $\Phi(x)\neq \infty$. Then using Fatous lemma and the continuity of $\phi$, we can conclude
	\[\begin{aligned}
		\liminf_{n\to\infty}\Phi(x_n)\geq \inf_{n\in\N}\Phi(x_n)= \liminf_{i\to\infty}\Phi(x_{n_{k_i}})\geq \int_0^1 \liminf_{i\to\infty}\phi(x_{n_{k_i}}(\xi))\;\mathrm{d}\xi=\Phi(x).
	\end{aligned}\]
	If $\Phi(x)=\infty$ also $\liminf_{n\to\infty}\Phi(x_n)=\infty$. In summary, $\Phi$ is lower semicontinuous.
\end{rem}
It is well known that $B_X=(e_k)_{k\in\N}=(\sqrt{2}\sin(k\pi\cdot))_{k\in\N}$ is an orthonormal basis of $X$. Recall the corresponding orthogonal projection $P_n$ and define
\[\Phi_n\defeq \Phi\circ P_n:X\rightarrow (-\infty,\infty).\]

\begin{prop}\label{prop:pot_approx_wit_conv}
	For each $p\in [1,\infty)$, we have $\lim_{n\to\infty}\Phi_n=\Phi$ in $L^{p}(X;\mu)$. If $p>1$ it holds, 
	\[\begin{aligned}
		\Phi\in W^{1,p}(X;\mu)\quad\text{with}\quad D\Phi(x)=\phi'\circ x\quad\text{for $\mu$-a.e.}\quad x\in X.
	\end{aligned}\]
\end{prop}
\begin{proof}
	By \cite[Lemma 5.1]{dPL14} and the growth condition on $\phi$, it holds $\Phi_n\in L^p(X;\mu)$ . Moreover, $\Phi_n\in C^1(X;\R)$ with $D\Phi_n(x)=\phi'( P_nx)\in  L^p((0,1);\mathrm{d}\xi)$, as it is the composition of the smooth function $X\ni x\mapsto P_nx\in C^0([0,1];\R)$ and the $C^1(C^0([0,1];\R);\R)$ function $ C^0([0,1];\R)\ni y\mapsto \int_0^1\phi(y(\xi))\;\mathrm{d}\xi$. An application of \cite[Section 10.1.1]{PR07} shows that $\Phi_n\in W^{1,2}(X;\mu)$.
	
	Using the Hölder inequality, the mean value theorem and \cite[Lemma 5.1]{dPL14}, we find ${A}\in (0,\infty)$ such that
	\[
	\begin{aligned}
		&\quad\int_X \abs{\Phi_n-\Phi}^{p}\;\mathrm{d}\mu\\
		&\leq \int_X \int_0^1\left(\phi(P_nx(\xi))-\phi(x(\xi)\right)^p\,\mathrm{d}\xi \;\mu(\mathrm{d}x)\\
		&\leq {a}^p\int_X\int_0^1 \left(1+\left(\abs{P_nx(\xi)}+\abs{x(\xi)}\right)^{b}\right)^{p}\abs{P_nx(\xi)-x(\xi)}^p\,\mathrm{d}\xi\;\mu(\mathrm{d}x)\\
		&\leq {a}^p\int_X\norm[\big]{\left(1+\left(\abs{P_nx}+\abs{x}\right)^{b}\right)^{p}}_{X}\norm[\big]{\left(P_nx-x\right)^{p}}_{X}\;\mu(\mathrm{d}x)\\
		&\leq  {A}\left(\int_X\norm[\big]{\left(P_nx-x\right)^{p}}^2_{X} \;\mu(\mathrm{d}x)\right)^{\frac{1}{2}}\\
		&={A}\left(\int_X\int_0^1\abs[\big]{\left(P_nx(\xi)-x(\xi)\right)}^{2p} \;\mathrm{d}\xi\;\mu(\mathrm{d}x)\right)^{\frac{1}{2}}.
	\end{aligned}
	\]Therefore, by \cite[Lemma 5.1]{dPL14} we have $\lim_{n\to\infty}\Phi_n=\Phi$ in $L^{p}(X;\mu)$. 
	
	Observe that for each $j\in\N$ with $j\leq n$, we can estimate 
	\[\begin{aligned}
		&\quad\int_X\abs{\party{j}\Phi_n(x)-(\phi'(x),e_j)_X}^p\;\mu(\mathrm{d}x)=\int_X\abs{(\phi'(P_nx)-\phi'(x),e_j)_X}^p\;\mu(\mathrm{d}x)\\
		&\leq \int_X\int_0^1\abs{\phi'(P_nx(\xi))-\phi'(x(\xi))e_j(\xi)}^p\;\mathrm{d}\xi\;\mu(\mathrm{d}x).
	\end{aligned}\]
	By \cite[Lemma 5.1]{dPL14}, we know that 
	\[
	(x,\xi)\mapsto \left(\sqrt{2}a \left(2+\abs{P_nx(\xi)}^b+\abs{x(\xi)}^b\right)\right)^p
	\] converges in $L^1(X\times (0,1);\mu\otimes\mathrm{d}\xi)$ as $n\to\infty$. Therefore, \cite[Theorem IV.9]{br} provides a function $g\in L^1(X\times (0,1);\mu\otimes\mathrm{d}\xi)$ such that for some subsequence, for $\mu\otimes \mathrm{d}\xi$-a.e.~$(x,\xi)\in X\times(0,1)$ and for all $k\in\N$
	
	\[
	\begin{aligned}
		\abs{\phi'(P_{n_k}x(\xi))-\phi'(x(\xi))e_j(\xi)}^p
		\leq \left(\sqrt{2}a \left(2+\abs{P_{n_k}x(\xi)}^b+\abs{x(\xi)}^b\right)\right)^p
		\leq g(x,\xi).
	\end{aligned}
	\]
	Since for a subsequence $\lim_{i\to\infty}P_{n_{k_i}}x(\xi)=x(\xi)$ for $\mu\otimes\mathrm{d}\xi$-a.e.~$(x,\xi)\in X\times (0,1)$ and $\phi'$ is continuous,
	we can apply the theorem of dominated convergence to show that there exists a subsequence $(\Phi_{n(j)_k})_{k\in \N}$ of $(\Phi_{n})_{n\in \N}$ such that $(\party{j}\Phi_{n(j)_k})_{k\in \N}$ converges pointwisely $\mu$-a.e.~and in $L^p(X;\mu)$ to $(\phi'(\cdot),e_j)_X$. 
	
	To continue, we first establish that $(D\Phi_{n(j)_k})_{k\in \N}$ is bounded in $L^p(X;\mu;X)$. This follows since 
	\[\begin{aligned}
		\int_X\norm{D\Phi_n(x)}_X^p\;\mu(\mathrm{d}x)&=\int_X\left(\int_0^1\abs{\phi'(P_nx)}^2\;\mathrm{d}\xi\right)^{\frac{p}{2}}\;\mu(\mathrm{d}x)\\
		&\leq \int_X\int_0^1\left({a \left(1+\abs{P_{n}x(\xi)}^b\right)}\right)^p\;\mathrm{d}\xi\;\mu(\mathrm{d}x)
	\end{aligned}\]
	and the right-hand side is bounded independent of $n\in\N$ by \cite[Lemma 5.1]{dPL14}. Since the sequence $(\Phi_n)_{n\in\N}\subseteq W^{1,2}(X;\mu)$ is bounded in $L^p(X;\mu),$ we get the boundedness of $(\Phi_{{n(j)}_k})_{k\in \N}$ in $W^{1,p}(X;\mu)$. As $W^{1,p}(X;\mu)$ has the Banach-Saks property for every $p\in(1,\infty)$, see \Cref{rem:Banach-Saks}, we know that there exits a subsequence  $(\Phi_{{n(j)}_{k_i}})_{i\in \N}$ of $(\Phi_{{n(j)}_k})_{k\in \N}$ and $\Psi\in W^{1,p}(X;\mu)$ such that 
	\[
	\lim_{N\to\infty}\frac{1}{N}\sum_{i=1}^N\Phi_{{n(j)}_{k_i}}=\Psi\quad\text{in}\quad W^{1,p}(X;\mu).
	\]Using $\lim_{n\to\infty}\Phi_n=\Phi$ in $L^{p}(X;\mu)$, we see\[ \Psi=\lim_{N\to\infty}\frac{1}{N}\sum_{k=1}^N\Phi_{{n(j)}_{k_i}}=\Phi\quad\text{in}\quad L^{p}(X;\mu)\]and therefore $\Psi$ does not depend on the subsequence we were starting with. In particular, the above argumentation shows that $\Phi\in W^{1,p}(X;\mu)$ with $D\Phi=D\Psi$ in $L^p(X;\mu;X)$.

	Moreover, we know that there is a subsequence $(N_m)_{m\in\N}$ such that for $\mu$-a.e. $x\in X$
	\[\lim_{m\to\infty}\frac{1}{N_m}\sum_{i=1}^{N_m}D\Phi_{{n(j)}_{k_i}}(x)=D\Phi(x).\]
	We finally conclude that for $\mu$-a.e.~$x\in X$ and for all $j\in\N$
	\[
	\begin{aligned}
		(D\Phi(x),e_j)_X=\lim_{m\to\infty}\frac{1}{N_m}\sum_{i=1}^{N_m}\party{j}\Phi_{{n(j)}_{k_i}}(x)=(\phi'(x),e_j)_X
	\end{aligned}
	\]and therefore
	\[D\Phi(x)=\sum_{j=1}^{\infty}(\phi'(x),e_j)_Xe_j=\phi'(x)\quad \text{for $\mu$-a.e.~$x\in X$},\]
	as $(e_j)_{j\in\N}$ is an orthonormal basis of $X$.
\end{proof}


\begin{rem}\label{rem:pot_approx_alter}Suppose $\phi$ and $\Phi$ are as described above. There are two more natural situations in which we derive similar results as in \Cref{prop:pot_approx_wit_conv}. 
	\begin{enumerate}[(i)]
		\item Assume that $\phi$ and therefore also $\Phi$ is convex. Then, the Moreau-Yosida approximation $(\Phi_t)_{t> 0}$ from \Cref{ex:yoshida_two} converges to $\Phi$ in $W^{1,p}(X;\mu)$ for all $p\in [1,\infty)$ and $D\Phi(x)=\phi'(x)$ for all $x\in L^{2b}(X;\mu)$, compare \cite[Proposition 5.1]{dPL14}. We do not give the proof here, but it relies on \Cref{ex:yoshida_two} and the facts that for each $x\in L^{2b}(X;\mu)$ we have $\partial\Phi(x)=\{\phi'(x)\}$, as well as $x\mapsto \norm{\phi'(x)}_X\in L^p(X;\mu)$ for all $p\in [1,\infty)$. 
		\item Assume that $\phi''$ exists, is continuous and grows at most of order $\tilde{b}\in [0,\infty)$. Then, using similar arguments as in the beginning of the proof of \Cref{prop:pot_approx_wit_conv}, we find $\tilde{A}\in (0,\infty)$, such that for all $p\in [1,\infty)$
		\[
		\begin{aligned}
			\int_X \norm[\big]{D\Phi_n(x)-\phi'(x)}_X^{p}\;\mu(\mathrm{d}x)&\leq \int_X \int_0^1\left(\phi'(P_nx(\xi))-\phi'(x(\xi)\right)^p\,\mathrm{d}\xi \;\mu(\mathrm{d}x)\\
			&\leq  \tilde{A} \left(\int_X\norm[\big]{\left(P_nx-x\right)^{p}}^2_{X} \;\mu(\mathrm{d}x)\right)^{\frac{1}{2}}.
		\end{aligned}
		\]
		Therefore, by \cite[Lemma 5.1]{dPL14}, we know that the sequence $(D\Phi_n)_{n\in\N}$ converges to $\phi'(\cdot)$ in $L^p(X;\mu;X)$. As $\lim_{n\to\infty}\Phi_n=\Phi$ in $L^p((0,1);\mathrm{d}\xi)$, we get $\lim_{n\to\infty}\Phi_n=\Phi$ in $W^{1,p}(X;\mu)$ with $D\Phi(x)=\phi'(x)$ for $\mu$-a.e.~$x\in X$.
	\end{enumerate}
	Both, the approximation from \Cref{prop:pot_approx_wit_conv} and the one from \Cref{rem:pot_approx_alter}, play an important role for our applications. Note that the first one does not demand the convexity of $\phi$, while the second yields one with Lipschitz continuous derivatives.
\end{rem}

Below, we derive a generalized Poincar\'e inequality designed for the application in \Cref{sec:hypo_appl}.

\begin{lem}\label{poin_check}Suppose $\Phi: X\rightarrow  (-\infty,\infty]$ is convex, bounded from below, lower semicontinuous and not identically to $\infty$. Then, for all $f\in \fbs{B_X}$ it holds
	\begin{align*}
		\lambda_1	\int_X(QD f,D f)_X\mathrm{d}\mu^{\Phi}\geq \int_X (f-\mu^{\Phi}(f))^2 \mathrm{d}\mu^{\Phi}.
	\end{align*}
	\begin{proof}
		The idea of the proof is to approximate $\Phi$. Afterwards, we apply the Poincar\'e inequality from \cite[Proposition 4.5]{AFP21}. 
		
		Denote by $(\Phi_{\alpha})_{\alpha >0}$ the Moreau-Yosida approximation of $\Phi$.
		By \Cref{ex:yoshida_one}, we know that $\Phi_{\alpha}$ is convex and differentiable with Lipschitz continuous derivative. Furthermore, for all $x\in X$, $\lim_{\alpha\rightarrow 0}\Phi_{\alpha}(x)=\Phi(x)$.
		To apply the Poincar\'e inequality from \cite[Proposition 4.5]{AFP21} $\Phi_{\alpha}$ is not regular enough. Therefore, let $\beta>0$ and define the function $\Phi_{\alpha,\beta}\defeq S_{\beta}\Phi_{\alpha}$, where $(S_{\beta})_{\beta\geq 0}$ is the Ornstein-Uhlenbeck semigroup considered in \cite[Chapter 11.6]{dPL06}. In formulas
		\[
		\Phi_{\alpha,\beta}(x) =\int_X \Phi_{\alpha}(e^{\beta B}x+\sqrt{\id-e^{2\beta B}}y)\,N(0,-B^{-1})(\mathrm{d}y).\]
		As discussed in \cite[Chapter 11.6]{dPL06}, it holds for every $\alpha,\beta\in (0,\infty)$
		\begin{enumerate}
			\item[(i)] $\Phi_{\alpha,\beta}$ is convex and has derivatives of all orders.
			\item[(ii)] $D\Phi_{\alpha,\beta}$ is Lipschitz continuous (with Lipschitz constant independent of $\beta$) and has bounded derivatives of all orders.
		\end{enumerate}
		In particular, \cite[Proposition 4.5]{AFP21} is applicable. We get for all $f\in \fbs{B_X}$
		\begin{align}\label{alp_bet}
			\lambda_1\int_X(QD f,D f)_X\mathrm{d}\mu^{\Phi_{\alpha,\beta}}\geq \int_X (f-\mu^{\Phi_{\alpha,\beta}}(f))^2 \mathrm{d}\mu^{\Phi_{\alpha,\beta}}.
		\end{align}
		Since the derivative of $\Phi_{\alpha}$ is Lipschitz continuous, one can show that $\Phi_{\alpha}$ has at most quadratic growth. Hence, there exits a constant $c\in (0,\infty)$ such that for all $x,y\in X$
		\[
		\begin{aligned}
			\abs{\Phi_{\alpha}(e^{\beta B}x+\sqrt{\id-e^{2\beta B}}y)}&\leq c\left(1+\norm{e^{\beta B}x+\sqrt{\id-e^{2\beta B}}y)}^2_X\right)\\
			&\leq 2c\left(1+\norm{e^{\beta B}x}_X^2+\norm{\id-e^{2\beta B}}_{\mathcal{L}(X)}\norm{y}_X^2\right)\\
			&\leq 2c\left(1+\norm{x}_X^2+2\norm{y}_X^2\right).
		\end{aligned}
		\]
		Above, we also used that $\norm{e^{\beta B}}_{\mathcal{L(X)}}\leq 1$ for all $\beta\in [0,\infty)$.
		Since, for all $x,y\in X$\[\lim_{\beta\to 0}\Phi_{\alpha}(e^{\beta B}x+\sqrt{\id-e^{2\beta B}}y)=\Phi_{\alpha}(x),\]  we obtain $\lim_{\beta\rightarrow 0}\Phi_{\alpha,\beta}(x)=\Phi_{\alpha}(x)$ by the theorem of dominated convergence. This yields $\lim_{\alpha\rightarrow 0}\lim_{\beta\rightarrow 0}\Phi_{\alpha}(x)=\Phi(x)$ for all $x\in X$. 
		As $-\infty<\inf_{\bar{x}\in X}\Phi(\bar{x})\leq\Phi_{\alpha}(x)\leq \Phi(x)$ for all $x\in X$, it is easy to see that $-\infty<\inf_{\bar{x}\in X}\Phi(\bar{x})\leq \Phi_{\alpha,\beta}(x)$ for all $x\in X$. In particular $e^{-\Phi_{\alpha}}$ and $e^{-\Phi_{\alpha,\beta}}$ are bounded independent of $\alpha,\beta$. 
		An iterative application of the theorem of dominated convergence  shows that 
		\begin{align*}
			\lim_{\alpha\rightarrow 0}\lim_{\beta\rightarrow 0}\mu^{\Phi_{\alpha,\beta}}(g)=\mu^{\Phi}(g)\quad\text{for all}\quad g\in L^1(X;\mu).
		\end{align*}Consequently, taking the limits $\beta\rightarrow 0$ and $\alpha\rightarrow 0$ in Inequality \eqref{alp_bet} yields the claim.
	\end{proof}
\end{lem}
The subsequent lemma shows that the Poincar\'e inequality is stable under additive perturbations with bounded oscillation. Consequently, we can also consider potentials which are not necessary convex.
\begin{lem}{\cite[Lemma 2.11]{EG23}}\label{lem:inf-dim-poincare}
	Suppose $\Phi=\Phi_1+\Phi_2$, where $\Phi_1:X\rightarrow (-\infty,\infty]$ is as in \Cref{poin_check} and the function $\Phi_2:X\rightarrow \R$ is measurable with $\norm{\Phi_2}_{osc}\defeq\sup_{x\in X}\Phi_2(x)-\inf_{x\in X}\Phi_2(x)<\infty$. Then 
	\[
	\lambda_1e^{\norm{\Phi_2}_{osc}}\int_X (QDf,Df)_X\,\mathrm{d}\mu^{\Phi}
	\geq \int_X \left(f-\mu^{\Phi}(f) \right)^2\,\mathrm{d}\mu^{\Phi}\quad\text{for all}\quad f\in\fbs{B_X}.
	\]
\end{lem}

\begin{rem}
	Note that the Poincar\'e inequality from \Cref{poin_check} above is valid without assuming that $\Phi\in W^{1,2}(X;\mu)$. 
\end{rem}

\section{The Kolmogorov operator}\label{sec:inf_langevin_general}
Let $(U,\spdx_U)$ and $(V,\spdx_V)$ be two real separable Hilbert spaces. Moreover, we fix two centered non-degenerate Gaussian measures $\mu_1$ and $\mu_2$ on $(U,\borel(U))$ and $(V,\borel(V))$, respectively.
Let $Q_i$ denote the covariance operator of $\mu_i$, $i=1,2$ with corresponding basis of eigenvectors $B_U=(d_k)_{k\in\N}$ and $B_V=(e_k)_{k\in\N}$ and positive eigenvalues $(\lambda_{1,k})_{k\in\N}$ and $(\lambda_{2,k})_{k\in\N}$, respectively. Without loss of generality, we assume that $(\lambda_{1,k})_{k\in\N}$ and $(\lambda_{2,k})_{k\in\N}$ are decreasing to zero. The corresponding projections to the induced subspaces, coordinate maps and embeddings are denoted by $P_n^U$, $p_n^U$, $\overline{p}_n^U$ and $P_n^V$, $p_n^V$, $\overline{p}_n^V$, respectively.

In addition, we fix a potential $\Phi: U\rightarrow  (-\infty,\infty]$ and assume for the rest of this section the following assumption.

\begin{assum}\label{ass:Phi_general}
	$\Phi:U\rightarrow (-\infty,\infty]$ is bounded from below by zero and there is $\theta\in [0,\infty)$ such that $\Phi\in W_{Q_1^{\theta}}^{1,2}(U;\mu_1)$. $\Phi$ is normalized, i.e.~$\int_U e^{-\Phi}\,\mathrm{d}\mu_1=1$.
\end{assum}

All results below are also valid if we replace bounded from below by zero with bounded from below. Through the application of a suitable scaling, $\int_U e^{-\Phi}\,\mathrm{d}\mu_1=1$ holds without loss of generality. 

\begin{defn}\label{def:core}
	Set $W\defeq U\times V$ and denote by $\spdx_W$ the canonical inner product on $W$ defined by
	\[
	\left((u_1,v_1),(u_2,v_2)\right)_W\defeq (u_1,u_2)_U+(v_1,v_2)_V,\quad \text{for all} \quad(u_1,v_1),(u_2,v_2)\in W.
	\] Then, $(W,\spdx_W)$ is a real separable Hilbert space. Furthermore, we define the measure $\mu_1^{\Phi}\defeq e^{-\Phi}\mu_1$ on $\borel(U)$ and recall the measure \[\mu^{\Phi}\defeq \mu_1^{\Phi}\otimes\mu_2\]on the Borel $\sigma$-algebra $\borel(W)=\borel(U)\otimes\borel(V)$.
	We set $\mu\defeq\mu^0=\mu_1\otimes\mu_2$. Due to \cite[Theorem~1.12]{dP06}, $\mu$ is a centered Gaussian measure with covariance operator $Q$ defined by 					\[Q:W\to W, \quad(u,v)\mapsto(Q_1u,Q_2v).\]
	Let $B_W$ be an ordered enumeration of the set
	\[
	\{ (d_n,0)\mid n\in\N \} \cup \{ (0,e_n)\mid n\in\N \} \subseteq W.
	\]
	Then, $B_W$ is an orthonormal basis of eigenvectors of $Q$. For each $n\in\N$, $k\in\N \cup \lbrace\infty\rbrace$ and $a\in \{c,b\}$ we define the spaces of finitely based cylinder functions with respect to $B_W$ by
	\[
	\begin{aligned}
		\mathcal{F}C_a^k(B_W,n)&\defeq \left\{f= \varphi\circ (p_n^U,p_n^V) \;\text{for some}\;\varphi\in C_a^k(\R^n\times\R^n)\right\}\quad\text{and}\\
		\mathcal{F}C_a^k(B_W)&\defeq \bigcup_{n\in\N}	\mathcal{F}C(B_W,n).
	\end{aligned}
	\]
	Further, we introduce $\mu^n\defeq \mu_1^n\otimes\mu_2^n$ on $\borel(\R^n\times\R^n)$, with $\mu_i^n$ being a centered Gaussian measure on $\borel(\R^n)$ with diagonal covariance matrix $Q_{i,n}\defeq \text{diag}(\lambda_{i,1}\dots,\lambda_{i,n})$.
\end{defn}

\begin{defn}\label{def:grad_comp}
	Let $n\in\N$ and $\varphi\in C^1(\R^n\times\R^n)$ be given. By $\partial_{i,1}\varphi$ and $\partial_{i,2}\varphi$, $1\leq i\leq n$, we denote the $i$-th partial derivative of $\varphi$ in the first and second component, respectively. We generalize this notation to gradients, e.g.~$D_1\varphi$ denotes the gradient of $\varphi$ with respect to the first component. 
	
	For G\^{a}teaux differentiable $f:W\to\R$ and all $w=(u,v)\in W$ we define
	\begin{align*}
		D_1f(w)&\defeq \sum_{n\in\N} (Df(w),(d_n,0))_W d_n\in U,\quad  \partx{i} f(w) \defeq (D_1f(w),d_i)_U\quad\text{and}\\
		D_2f(w)&\defeq \sum_{n\in\N} (Df(w),(0,e_n))_W e_n\in V,\quad\party{i} f(w) \defeq (D_2f(w),e_i)_V.
	\end{align*}
	In this manner we also define higher order (partial) derivatives.
\end{defn}

\begin{rem}\label{rem:fcb_dense_prod_space}
		Let $n\in\N$ and $f=\varphi\circ (p^U_n,p_n^V)\in \mathcal{F}C_b^1(B_W)$. Similar to \Cref{rem_der_cyl_X}, we compute for all $(u,v)\in W$ \[D_1f(u,v)=\sum_{n\in\N}\partial_{i,1}\varphi(p^U_nu,p^V_nv) d_i\quad\text{and}\quad D_2f(u,v)=\sum_{n\in\N}\partial_{i,2}\varphi(p^U_nu,p^V_nv)e_i.\] 
	
\end{rem}

In the next definition, we fix the coefficient operators determining the Kolmogorov operator. We directly include the invariance and growth conditions needed for our further considerations.
\begin{defn}\label{def:inf-dim-operators}
	We fix $K_{12}\in\lop{U;V}$ and set $K_{21}\defeq K_{12}^*\in\lop{V;U}$.
	Moreover, suppose $K_{22}:V\to\mathcal{L}_{>0}^{+}(V)$ and $ v\mapsto K_{22}(v)e_i\in C^1(V;V)$ for all $i\in\N$.
	Further, assume that there is a strictly increasing sequence $\sk{m}$ in $\N$ such that for each $k\in\N$ and $v\in V$, it holds that	
	\[
	\begin{aligned}
		K_{12}(U_{m_k})&\subseteq V_{m_k},\; K_{21}(V_{m_k})\subseteq U_{m_k},\; K_{22}(v)(V_{m_k})\subseteq V_{m_k}\;\text{and}\\				K_{22}(v)|_{V_{m_k}}&= K_{22}(P_{m_k}^Vv)|_{V_{m_k}}.
	\end{aligned}
	\]
	Moreover, suppose that for each $k\in\N$, there is a constant $M_k\in(0,\infty)$ such that
	\[
	\begin{aligned}
		\sup_{v\in V_{m_k}} \norm{K_{22}(v)}_{\lop{V_{m_k}}} &\leq M_k 
		\quad\text{ and }\\
		\norm{\party{i}K_{22}(v)}_{\lop{V_{m_k}}} &\leq M_k(1+\norm{v}_{V_{m_k}})
		\quad\text{ for all }v\in V_{m_k}, 1\leq i\leq m_k.
	\end{aligned}
	\]	
	Above, for each $v\in V_{m_k}$ the bounded linear operator $\party{i}K_{22}(v):V_{m_k}\to V_{m_k}$ is defined by $\party{i}K_{22}(v)\tilde{v}\defeq \sum_{j=1}^{m_k}\party{i}K_{22}(v)e_j (\tilde{v},e_j)_V$, $\tilde{v}\in V_{m_k}$.
	
	In the following, we set $m^K(n)\defeq \min_{k\in\N} \{m_k: m_k\geq n \}$.
\end{defn}
Roughly speaking, the invariance properties $K_{12},K_{21}$ and $K_{22}$ imply that they have a block invariance structure, where the size of the blocks is described by the increasing sequence $\sk{m}$.

\begin{rem}\label{rem:too-diagonal}
	Assume that $K_{22}(v)$ leaves $V_n$ invariant for all $n\in\N$ and $v\in V$. Using the strengthened invariance properties of $K_{22}$,
	it follows quickly that $K_{22}(v)$ is diagonal, i.e.~$K_{22}(v)e_i=\lambda_{22,i}(v)e_i$ for some positive continuous differentiable $\lambda_{22,i}:V\to\R$. In this situation it is particularly easy to check if $K_{22}$ fullfils the properties stated in \Cref{def:inf-dim-operators}.
\end{rem}

\begin{rem}\label{rem:dphi_finite}
	Suppose $f=\varphi\circ (p_n^U,p_n^V)\in\fbs{B_W,n}$ and by trivially extending $\varphi$ if necessary that $m^K(n)=n$.
	Then, by invariance properties of the coefficients, we compute
	\[
	\begin{aligned}
		(Q_1^{\theta}D\Phi,Q_1^{-\theta}K_{21}D_2f)_U&=\sum_{i=1}^n \lambda_{1,i}^{-\theta}(Q_1^{\theta}D\Phi,d_i)_U(d_i,K_{21}D_2f)_U\\&=\sum_{i=1}^n \partx{i}\Phi(d_i,K_{21}D_2f)_U.
	\end{aligned}
	\]Therefore, the interpretation of $(D\Phi,K_{21}D_2f)_U$ as $(Q_1^{\theta}D\Phi,Q_1^{-\theta}K_{21}D_2f)_U$ is reasonable even though we do not know if $\Phi\in W^{1,2}(U;\mu_1)$. For the following consideration we define for all $f\in \fbs{B_W}$ \[(D\Phi,K_{21}D_2f)_U\defeq(Q_1^{\theta}D\Phi,Q_1^{-\theta}K_{21}D_2f)_U\quad\text{ for all}\quad f\in \fbs{B_W}.\]  
	
\end{rem}
We are now able to define the Kolmogorov operator, associated to the degenerate infinite dimensional stochastic Hamiltonian system, on $\fbs{B_W}$.

\begin{defn}\label{def:inf_Lang_op}
	The differential operators $(S,\fbs{B_W})$ and $(A^{\Phi},\fbs{B_W})$ are defined on $L^2(W;\mu^{\Phi})$ by
	\[
	\begin{aligned}
		Sf(u,v)
		\defeq \tr\left[K_{22}(v)\circ D_2^2f(u,v)\right]
		&+ \sum_{j=1}^{\infty} (\party{j}K_{22}(v)D_2f(u,v),e_j)_V \\
		&- (v,Q_2^{-1}K_{22}(v)D_2f(u,v))_V
	\end{aligned}
	\]
	and
	\[
	\begin{aligned}
		A^{\Phi}f(u,v)
		\defeq &(u,Q_1^{-1}K_{21}D_2f(u,v))_U+(D\Phi(u),K_{21}D_2f(u,v))_U\\
		&-(v,Q_2^{-1}K_{12}D_1f(u,v))_V,
	\end{aligned}
	\]
	respectively, for all $(u,v)\in W$.
	The Kolmogorov operator denoted by $(L^{\Phi},\fbs{B_W})$ is defined via \[L^{\Phi}\defeq S-A^{\Phi}.\] For notational convenience we set $A\defeq A^0$ and $L\defeq L^0$.
\end{defn}
\begin{rem}\label{rem:inf-dim-op-well-defined}
	The invariance assumptions made on $K_{12}$, $K_{21}$ and $K_{22}$ ensure that $S$ and $A^{\Phi}$ and therefore also $L^{\Phi}$ are well-defined on $\fbs{B_W}$. Indeed, let $n\in\N$ and suppose $f=\varphi\circ (p_n^U,p_n^V)\in\fbs{B_W,n}$. By trivially extending $\varphi$ if necessary, we can assume $m^K(n)=n$. Then, for all $(u,v)\in W$ we get by \Cref{rem:fcb_dense_prod_space} Item (i)
	\[
	\begin{aligned}
		Q_1^{-1}K_{21}D_2f(u,v)\in U_n,\quad\text{and}\quad Q_2^{-1}K_{12}D_1f(u,v),\;Q_2^{-1}K_{22}(v)D_2f(u,v)\in V_n.
	\end{aligned}
	\]
	Moreover, these maps are uniformly bounded in $(u,v)\in W$ due to uniform boundedness of 	$K_{22}:V_{n}\rightarrow \mathcal{L}(V_{n})$ and the fact that all derivatives of $f$ are bounded. 
	By the observation that all sums appearing in the definition of $Sf$ and $A^{\Phi}f$ are finite, the fact that $\partx{i}\Phi\in L^2(U;\mu_1^{\Phi})$, as well as $\normx_U,\normx_V\in L^2(W;\mu^{\Phi})$, it follows that 
	\[
	\begin{aligned}
		Sf(u,v)&=Sf(P_n^U u, P_n^V v),\\ A^{\Phi}f(u,v)&=Af(P_n^U u, P_n^V) +(D\Phi(u),K_{21}D_2f(P_n^U u, P_n^V))_U
	\end{aligned}
	\]
	and $Sf,\,A^{\Phi}f\in L^2(W;\mu^{\Phi})$. Therefore, $Lf\in L^2(W;\mu^{\Phi})$ is finitely based and we have
	\[
	\begin{aligned}
		L^{\Phi}f(u,v)&=Sf(P_n^U u, P_n^V v)-Af(P_n^U u, P_n^V)-(D\Phi(u),K_{21}D_2f(P_n^U u, P_n^V))_U\\
		&=Lf(P_n^U u, P_n^V)-(D\Phi(u),K_{21}D_2f(P_n^U u, P_n^V))_U.
	\end{aligned}
	\]It is also possible to consider $(L^{\Phi},\fbs{B_W})$ on $L^p(W;\mu^{\Phi})$ for $p\in [1,2]$.
\end{rem}

As the abbreviation should suggest, we show below, among other things, that $S$ is symmetric and $A^{\Phi}$ is antisymmetric.
\begin{lem}{\cite[Lemma 3.5]{EG23}}\label{lem:inf-dim-op-decomp} The linear operator 
	$(S,\fbs{B_W})$ is symmetric and negative semi-definite, whereas $(A^{\Phi},\fbs{B_W})$ is antisymmetric on $L^2(W;\mu^{\Phi})$. Therefore, $(L^{\Phi},\fbs{B_W})$ is dissipative on $L^2(W;\mu^{\Phi})$.
	
	Denote by $(S,D(S))$, $(A^{\Phi},D(A^{\Phi}))$ and $(L^{\Phi},D(L^{\Phi}))$ the closures of the respective operators.
	Then, for all $f,g\in\fbs{B_W}$ it holds
	\[
	\begin{aligned}\label{lem:intbp}
		-\int_W L^{\Phi}fg\;\mathrm{d}\mu^{\Phi}=  \int_W (D_2 f, K_{22} D_2g)_V - (D_1 f, K_{21}D_2 g)_U + (D_2 f, K_{12}D_1 g)_V \,\mathrm{d}			\mu^{\Phi}
	\end{aligned}
	\]and the measure $\mu^{\Phi}$ is invariant for $(L^{\Phi},\fbs{B_W})$ in the sense that $\int_W L^{\Phi}f\;\mathrm{d}\mu^{\Phi}=0$ for all $f\in \fbs{B_W}$.
\end{lem}

We explicitly calculate the Carr\'e du champ operator $\Gamma$ of $(L^{\Phi},\fbs{B_W})$ and verify that the Kolmogorov operator is an abstract diffusion operator, below. Hence, if the operator $(L^{\Phi},\fbs{B_W})$ generates a strongly continuous semigroup, we immediately know by \cite[Lemma 1.8 and Lemma 1.9]{Eberle1999} that it is a Dirichlet operator in the sense of \cite[Capter I.4]{MaRockner} and consequently generated by a sub-Markovian semigroup.

\begin{lem}\label{cor:invariance_ASP} 
	The measure $\mu^{\Phi}$ is invariant for the symmetric operator $(S,\fbs{B_W})$ and the antisymmetric operator $(A^{\Phi},\fbs{B_W})$, therefore also for $(L^{\Phi},\fbs{B_W})$. Moreover, the Kolmogorov operator $(L^{\Phi},\fbs{B_W})$ is an abstract diffusion operator on $L^p(W;\mu^{\Phi})$ for all $p\in [1,2]$ and the corresponding Carr\'e du champ operator $\Gamma$ is given by
	\begin{equation}\label{eq:CDC}
		\Gamma(f,g)=\left(K_{22}D_2f,D_2g\right)_V\quad\text{for all}\quad f,g\in \fbs{B_W}.
	\end{equation}
\end{lem}
\begin{proof}
	As mentioned at the end of \Cref{rem:inf-dim-op-well-defined}, it is possible to consider the operator $(L^{\Phi},\fbs{B_W})$ on $L^p(W;\mu^{\Phi})$, for all $p\in [1,2]$.
	The first part of the statement directly follows by \Cref{lem:inf-dim-op-decomp}. 
	
	To calculate the Carr\'e du champ operator, let $f,g\in \fbs{B_W}$ be given. Obviously, their product $fg$ is in $\fbs{B_W}$ and by the classical product rule for differentiable functions, we obtain (note that all appearing infinite sums below are finite)
	\[
	\begin{aligned}
		L^{\Phi}(fg)&=\sum_{i,j=1}^{\infty} (K_{22}e_i,e_j)_V\party{i}\party{j}(fg)
		+ \sum_{j=1}^{\infty} (\party{j}K_{22}D_2(fg),e_j)_V \\
		&\quad- (v,Q_2^{-1}K_{22}D_2(fg))_V -(u,Q_1^{-1}K_{21}D_2(fg))_U\\
		&\quad-(D\Phi,K_{21}D_2(fg))_U+(v,Q_2^{-1}K_{12}D_1(fg))_V\\
		&=\sum_{i,j=1}^{\infty} (K_{22}e_i,e_j)_V\party{i}f\party{j}g+\sum_{i,j=1}^{\infty} (K_{22}e_i,e_j)_V\party{i}g\party{j}f+fL^{\Phi}g+gL^{\Phi}f\\
		&=2\left(K_{22}D_2f,D_2g\right)_V+fL^{\Phi}g+gL^{\Phi}f.
	\end{aligned}
	\]Therefore, \eqref{eq:CDC} holds and since $K_{22}(v)\in \lopp{V}$ for all $v\in V$, we conclude $\Gamma(f,f)\geq 0$. Let $m,n\in\N$, $f_1,\dots,f_m\in \fbs{B_W}$ and $\varphi\in C^{\infty}(\R^m)$ with $\varphi(0)=0$ be given. Since the composition of a $C^{\infty}(\R^m)$ function with a vector $(\psi_1,\dots,\psi_m)\in (C_b^{\infty}(\R^n))^m$ is in $C_b^{\infty}(\R^n)$, we obtain $\varphi(f_1,\dots,f_m)\in \fbs{B_W}$. Finally, 
	\[
	\begin{aligned}
		&\quad	L^{\Phi}\varphi(f_1,\dots,f_m)\\
		&=\sum_{k=1}^m\partial_k\varphi(f_1,\dots,f_m)L(f_k)+\sum_{k,l=1}^m \partial_l\partial_k\varphi(f_1,\dots,f_m)\left(K_{22}D_2f_k,D_2f_l\right)_V,
	\end{aligned}
	\]follows similar as above,  by the classical chain rule. All together we have verified that $(L^{\Phi},\fbs{B_W})$ is an abstract diffusion operator on $L^p(W;\mu^{\Phi})$ in the sense of \cite{Eberle1999}.
\end{proof}

We conclude this section by an first order $L^2(W;\mu^{\Phi})$ regularity result.

\begin{lem}{\cite[Lemma 4.7]{EG23}}\label{lem:reg_K_22}
	For $f\in\fbs{B_W}$ and $\lambda \in (0,\infty)$, set \[g\defeq\lambda f-L^{\Phi}f.\] Then, the following inequalities are valid
	\begin{equation}\label{varepsilonone}
		\int_W\norm{K_{22}^{\frac{1}{2}}D_2f}^2_V\dm^{\Phi}\leq \frac{1}{2} \int_Wf^2+(Lf)^2\dm^{\Phi}\quad\text{and}
	\end{equation}
	\begin{equation}\label{varepsilonn2}
		\int_W\norm{K_{22}^{\frac{1}{2}}D_2f}^2_V\dm^{\Phi}\leq \frac{1}{4\lambda} \int_Wg^2\dm^{\Phi}.
	\end{equation}
	
\end{lem}

\section{Essential m-dissipativity}\label{subsec:unb_grad}
In this section to overall goal is to establish essential m-dissipativity of the Kolmogorov operator, associated to the degenerate infinite dimensional stochastic Hamiltonian system, using the strategy described in the introduction. We start by stating the main assumption.

\begin{cond}{App($\Phi$)}\label{ass:pot_app} There exists a sequence $(\Phi_n^m)_{n,m\in \N}$ such that,
	\begin{enumerate}[label=Class \Alph*: , wide=0pt, leftmargin=\widthof{\textbf{Class B:\hspace*{\labelsep}}}%
		]
		\item[(App($\Phi 1$))]for each fixed $m\in\N$, $\Phi_n^m(u)=\Phi_n^m\circ P_n^Uu$ for all $n\in\N$ and $u\in U$.
		\item[(App($\Phi 2$))]$ v\mapsto K_{22}(v)e_i\in C^2_b(V;{V})$ for all $i\in\N$. Further, for all $m,n\in\N$, $\Phi_n^m\in C^3(U;\R)$ and $D\Phi_n^m$ has bounded derivatives up to the second order. 
		\item[(App($\Phi 3$))] there exists $\lambda\in (0,\infty)$ such that for every $n,m\in\N$ and $g\in \fbs{B_W}$, there is a solution $f_{n,m}\in \fbk{3}$ of Equation \eqref{eq:g-Lg=f} with $\norm{f_{n,m}}_{\infty}\leq \frac{1}{\lambda}\norm{g}_{\infty}$
		\item[(App($\Phi 4$))]there are $\alpha,\beta,\gamma\in [0,\infty)$ and  
		$\kappa\in (1,\infty)$, all independent of $m,n\in\N$ such that

		\begin{equation}
			\begin{aligned}\label{ineq:lfourone}
				&\int_W \norm[\big]{\sum_{i=1}^{\infty}\lambda_{1,i}^{\alpha}K_{22}^{-\frac{1}{2}} K_{12}D \partxsq{i}\Phi_n^m}_V^2+
				\norm[\big]{\sum_{i=1}^{\infty} K_{22}^{-\frac{1}{2}} K_{12}D\partx{i}\Phi_n^m\, (u,Q_1^{\alpha-1}e_i)_U}_V^2 \\
				&\quad + \norm[\big]{\sum_{i=1}^{\infty}K_{22}^{-\frac{1}{2}} K_{12}D\partx{i}\Phi_n^m\,\lambda_{1,i}^{\alpha}\partx{i}\Phi_n^m}_V^2\;\mathrm{d}\mu^{\Phi_n^m}\leq \kappa^2
			\end{aligned}	
		\end{equation}
		
		\begin{equation}\label{ineq:lfourtwo}
			\sum_{i=1}^{\infty}\int_W\norm[\big]{\lambda_{1,i}^{\frac{\alpha}{2}}K_{22}^{-\frac{1}{2}} K_{12}D \partx{i}\Phi_n^m}_U^2\;\mathrm{d}\mu^{\Phi_n^m}\leq \kappa.
		\end{equation}
		Additionally, it holds for all $f \in \fbs{B_W}$
		\begin{equation}
			\label{ineq:lfourthree}
			\norm[\big]{Q_1^{\frac{\alpha}{2}-1}K_{21}D_2f}_V\leq \kappa\norm[\big]{K_{22}^{\frac{1}{2}} D_2f}_V
		\end{equation}
		\begin{equation}
			\label{ineq:lfourfour}
			\norm[\big]{Q_2^{\frac{\beta}{2}-1}K_{22}D_2f}_V+\left(\sum_{i=1}^{\infty}\norm[\big]{\lambda_{2,i}^{\frac{\beta}{2}}K_{22}^{-\frac{1}{2}}\party{i}K_{22}D_2f}_V^2\right)^{\frac{1}{2}}\leq \kappa\norm[\big]{K_{22}^{\frac{1}{2}} D_2f}_V			
		\end{equation}
		\begin{equation}
			\label{ineq:lfourfive}
			\norm[\big]{Q_2^{\frac{\beta}{2}-1}K_{12}D_1f}_V\leq \kappa\norm[\big]{Q_1^{\frac{\alpha}{2}} D_1f}_U
		\end{equation}		
		\begin{equation}
			\label{ineq:lfoursix}
			\left(\sum_{i=1}^{\infty}\left(\int_W\lambda_{2,i}^{4\gamma}(\party{i}^2f)^2\;\mathrm{d}\mu^{\Phi_n^m}\right)^{\frac{1}{2}}\right)^2\leq \kappa\int_W \sum_{i=1}^{\infty}\norm[\big]{\lambda_{2,i}^{\frac{\beta}{2}}K_{22}^{\frac{1}{2}}D_2\party{i}f}^2_V \;\mathrm{d}\mu^{\Phi_n^m}
		\end{equation}
		\begin{equation}
			\label{ineq:lfourseven}
			\sum_{i=1}^{\infty}  \lambda_{2,i}^{2\gamma-1} \leq 	\kappa^{\frac{1}{2}}.
		\end{equation}
		
		\item[(App($\Phi 5$))]there is some constant $p^*\in (4,\infty)$ such that	\begin{equation}\label{eq:conv_phi_m_n_p}
			\li{n}\li{m} \int_U \norm[\big]{Q_2^{-\gamma}K_{12}(D\Phi_n^m-P_nD\Phi)}_V^{p^*}\;\mathrm{d}\mu^{\Phi_n^m}_1=0.
		\end{equation}
		
		For $q^*\defeq \frac{2p^*}{p^*-4}$ there are constants $c_1,c_2\in \R$ and $c_3<\frac{1}{2\lambda_{1,1}}$ such that for all $m,n\in\N$
		\begin{align}\label{eq:phig4_one}
			c_1&\leq \Phi_n^m(u)\quad \text{and}\\\label{eq:phig4_two} (q^*-1)\Phi_n^m(u)&\leq  c_2+c_3\norm{u}_U^2  +q^*\Phi(u)\quad\text{for all}\quad u\in U.
		\end{align}
	\end{enumerate}
\end{cond}

For the next three lemmas, we consider Assumption \nameref{ass:pot_app} as valid. In particular, there exists $ \lambda\in (0,\infty)$ such that for $g\in \fbs{B_W}$ there is a function $f_{n,m}\in \fbk{3}$ with
\begin{equation}\label{eq:resolvent_eq_est}
	\lambda f_{n,m}-L^{{\Phi_n^m}}f_{n,m}=g\quad\text{and}\quad\norm{f_{n,m}}_{\infty}\leq \frac{1}{\lambda}\norm{g}_{\infty}.
\end{equation}
We next establish the existence of a constant $c$, independent of $n,m\in\N$ such that
\[
\int_W \norm[\big]{Q_2^{{\gamma}}D_2f_{n,m}}^4_V\;\mathrm{d}\mu^{\Phi_n^m}\leq c,
\]
where $\gamma$ is the parameter from App($\Phi 4$).

\begin{lem}\label{lemma:reg_estimate_D_one}
	There is a constant $a\defeq a(\lambda,g,\kappa)\in (0,\infty)$, independent of $m,n\in\N$ with
	\[
	\int_W \norm[\big]{Q_1^{\frac{\alpha}{2}}D_1f_{n,m}}_U^2\;\mathrm{d}\mu^{\Phi_n^m}\leq a.
	\]
\end{lem}
\begin{proof}
	To avoid an overload of notation, we fix $n,m$ and substitute $\Phi_n^m$ with $\Phi$ and $f_{n,m}$ by $f$ in the following proof.
	Differentiating \eqref{eq:resolvent_eq_est} with respect to $d_i$ yields
	\begin{equation}
		\lambda\partx{i}f-L^{\Phi}\partx{i}f+(d_i,Q_1^{-1}K_{21}D_2f)_U+(D\partx{i}\Phi,K_{21}D_2f)_U=\partx{i}g.
	\end{equation}
	Multiply the equation above with $\lambda_{1,i}^{\alpha}\partx{i}f$, integrate over $W$ with respect to $\mu^{\Phi}$ and sum over all $i\in\N$ to obtain

	\begin{equation}
		\begin{aligned}\label{eq_reg_est_a}
			&\int_W \lambda \norm[\big]{Q_1^{\frac{\alpha}{2}}D_1f}^2_U+\sum_{i=1}^{\infty}\norm[\big]{\lambda_{1,i}^{\frac{\alpha}{2}}K_{22}^{\frac{1}{2}}D_2\partx{i}f}_V^2 +\left( Q_1^{\alpha}D_1f,Q_1^{-1}K_{21}D_2f\right)_U\;\mathrm{d}\mu^{\Phi}\\
			&\quad+\int_W \sum_{i=1}^{\infty}\lambda_{1,i}^{\alpha}\partx{i}f \left( D\partx{i}\Phi,K_{21}D_2f\right)_U\;\mathrm{d}\mu^{\Phi}=\int_W(Q_1^{\frac{\alpha}{2}}D_1f,Q_1^{\frac{\alpha}{2}}D_1g)_U\;\mathrm{d}\mu^{\Phi},
		\end{aligned}
	\end{equation}
	where we also used \Cref{lem:intbp} and in particular 
	\[
	\begin{aligned}
		-\left( L^{\Phi}\partx{i}f,\partx{i}f\right)_{L^2(\mu^{\Phi})}
		=\int_W\left(K_{22}D_2\partx{i}f,D_2\partx{i}f\right)_V\;\mathrm{d}\mu^{\Phi}.
	\end{aligned}
	\]
	Note that 
	\[
	\sum_{i=1}^{\infty}\lambda_{1,i}^{\alpha}\partx{i}f \left( D\partx{i}\Phi,K_{21}D_2f\right)_U=\sum_{i,j=1}^{\infty}\lambda_{1,i}^{\alpha}\partx{i}f \partx{i}\partx{j}\Phi \left( d_j,K_{21}D_2f\right)_U.
	\] Using the integration by parts formula from \Cref{lem:intbp_classic_pot}, we compute
	\[
	\begin{aligned}
		&\quad\int_W \lambda_{1,i}^{\alpha}\partx{i}f \partx{i}\partx{j}\Phi \left( d_j,K_{21}D_2f\right)_U\;\mathrm{d}\mu^{\Phi}\\
		&=-\int_W   \lambda_{1,i}^{\alpha}f \partx{i}\partx{i}\partx{j}\Phi \left( d_j,K_{21}D_2f\right)_U\;\mathrm{d}\mu^{\Phi}\\
		&\quad-\int_W   \lambda_{1,i}^{\alpha}f \partx{i}\partx{j}\Phi \left( d_j,K_{21}D_2\partx{i}f\right)_U\;\mathrm{d}\mu^{\Phi}\\
		&\quad+\int_W   \lambda_{1,i}^{\alpha-1}(u,d_i)_Uf \partx{i}\partx{j}\Phi \left( d_j,K_{21}D_2f\right)_U\;\mathrm{d}\mu^{\Phi}\\
		&\quad+\int_W   \lambda_{1,i}^{\alpha}\partx{i}\Phi f \partx{i}\partx{j}\Phi \left( d_j,K_{21}D_2f\right)_U\;\mathrm{d}\mu^{\Phi}\\
		&\eqdef -I^1_{ij}-I^2_{ij}+I^3_{ij}+I^4_{ij}
	\end{aligned}
	\]
	The next step on our agenda is to estimate $I^1_{i,j},I^2_{i,j},I^3_{i,j}$ and $I^4_{i,j}$ separately.
	Indeed, using Inequality \eqref{ineq:lfourone} from Assumption \nameref{ass:pot_app} and the Inequalities \eqref{eq:resolvent_eq_est} and \eqref{varepsilonn2}, we estimate
	\[
	\begin{aligned}
		\abs[\Big]{\sum_{i,j=1}^{\infty}I^1_{ij}}&= \abs[\Big]{\int_W \bigl(\sum_{i=1}^{\infty}\lambda_{1,i}^{\alpha}D \partxsq{i}\Phi,K_{21}D_2f\bigr)_Uf\;\mathrm{d}\mu^{\Phi}}\\
		&\leq \norm{f}_{\infty} \left(\int_W \norm[\Big]{\sum_{i=1}^{\infty}\lambda_{1,i}^{\alpha}K_{22}^{-\frac{1}{2}} K_{12}D \partxsq{i}\Phi}_V^2\;\mathrm{d}\mu^{\Phi}\right)^{\frac{1}{2}}\left(\int_W\norm[\big]{K_{22}^{\frac{1}{2}} D_2f}^2_V\;\mathrm{d}\mu^{\Phi}\right)^{\frac{1}{2}}\\
		&\leq \frac{\norm{g}_{\infty}}{\lambda}\kappa \frac{1}{2\sqrt{\lambda}} \norm{g}_{L^2(\mu^{\Phi})}\eqdef a_1.
	\end{aligned}
	\]
	Using Inequality \eqref{ineq:lfourtwo} and Youngs inequality for $\delta\in(0,\infty)$, yields
	\[
	\begin{aligned}
		\abs[\Big]{\sum_{i,j=1}^{\infty}I^2_{ij}}&= \abs[\Big]{\int_W \sum_{i=1}^{\infty}\bigl(\lambda_{1,i}^{\alpha}D \partx{i}\Phi,K_{21}D_2\partx{i}f\bigr)_Uf\;\mathrm{d}\mu^{\Phi}}\\
		&\leq \norm{f}_{\infty} \sum_{i=1}^{\infty} \int_W \norm[\big]{\lambda_{1,i}^{\frac{\alpha}{2}}K_{22}^{-\frac{1}{2}} K_{12}D \partx{i}\Phi}_V\norm[\big]{\lambda_{1,i}^{\frac{\alpha}{2}}K_{22}^{\frac{1}{2}}D_2\partx{i}f}_V\;\mathrm{d}\mu^{\Phi}\\
		&\leq \norm{f}_{\infty}\sum_{i=1}^{\infty}\int_W\frac{1}{2\delta}\norm[\big]{\lambda_{1,i}^{\frac{\alpha}{2}}K_{22}^{-\frac{1}{2}} K_{12}D \partx{i}\Phi}_V^2+\frac{\delta}{2}\norm[\big]{\lambda_{1,i}^{\frac{\alpha}{2}}K_{22}^{\frac{1}{2}}D_2\partx{i}g}_V^2\;\mathrm{d}\mu^{\Phi}\\
		&\leq \frac{\norm{g}_{\infty}}{\lambda}\frac{1}{2\delta}\kappa+\frac{\norm{g}_{\infty}}{\lambda}\frac{\delta}{2}\sum_{i=1}^{\infty}\int_W \norm[\big]{\lambda_{1,i}^{\frac{\alpha}{2}}K_{22}^{\frac{1}{2}}D_2\partx{i}f}_V^2\;\mathrm{d}\mu^{\Phi}.
	\end{aligned}
	\]
	Taking Inequality \eqref{ineq:lfourone} into account, we estimate
	\[
	\begin{aligned}
		&\quad	\abs[\Big]{\sum_{i,j=1}^{\infty}I^3_{ij}}\\&= \abs[\Big]{\int_W \sum_{i=1}^{\infty} \left(D \partx{i}\Phi,K_{21}D_2f\right)_U(u,Q_1^{\alpha-1}e_i)_Uf\;\mathrm{d}\mu^{\Phi}}\\
		&\leq \norm{f}_{\infty} \int_W  \norm[\Big]{\sum_{i=1}^{\infty} K_{22}^{-\frac{1}{2}} K_{12}D\partx{i}\Phi\, (u,Q_1^{\alpha-1}e_i)_U}_V \norm[\big]{K_{22}^{\frac{1}{2}} D_2f}_V\;\mathrm{d}\mu^{\Phi}\\
		&\leq \norm{f}_{\infty} \left(\int_W  \norm[\Big]{\sum_{i=1}^{\infty} K_{22}^{-\frac{1}{2}} K_{12}D\partx{i}\Phi\, (u,Q_1^{\alpha-1}e_i)_U}_V^2 \;\mathrm{d}\mu^{\Phi}\int_W \norm[\big]{K_{22}^{\frac{1}{2}} D_2f}^2_V\;\mathrm{d}\mu^{\Phi}\right)^{\frac{1}{2}}\\
		&\leq  \frac{\norm{g}_{\infty}}{\lambda} \kappa  \frac{1}{2\sqrt{\lambda}} \norm{g}_{L^2(\mu^{\Phi})}=a_1.
	\end{aligned}
	\]
	Again, by \eqref{ineq:lfourone}, we have
	\[
	\begin{aligned}
		\abs[\Big]{\sum_{i,j=1}^{\infty}I^4_{ij}}&= \abs[\Big]{\int_W \sum_{i=1}^{\infty} \left(D \partx{i}\Phi ,K_{21}D_2f\right)_U \lambda_{1,i}^{\alpha}\partx{i}\Phi f\;\mathrm{d}\mu^{\Phi}}\\
		&\leq \norm{f}_{\infty} \int_W  \norm[\big]{\sum_{i=1}^{\infty}K_{22}^{-\frac{1}{2}} K_{12}D\partx{i}\Phi\,\lambda_{1,i}^{\alpha}\partx{i}\Phi}_V \norm[\big]{K_{22}^{\frac{1}{2}} D_2f}_V\;\mathrm{d}\mu^{\Phi}\\
		&\leq \frac{\norm{g}_{\infty}}{\lambda} \kappa  \frac{1}{2\sqrt{\lambda}} \norm{g}_{L^2(\mu^{\Phi})}= a_1.
	\end{aligned}
	\]
	A combination of the above (in)equalities implies
	\begin{equation}\label{eq:sum_ij}
		\sum_{i,j=1}^{\infty}I^1_{ij}+I^2_{ij}-I^3_{ij}-I^4_{ij} \leq \frac{\norm{g}_{\infty}}{\lambda}\frac{1}{2\delta}\kappa+\frac{\norm{g}_{\infty}}{\lambda}\frac{\delta}{2}\sum_{i=1}^{\infty}\int_W \norm[\big]{\lambda_{1,i}^{\frac{\alpha}{2}}K_{22}^{\frac{1}{2}}D_2\partx{i}f}_V^2\;\mathrm{d}\mu^{\Phi}+3a_1.	
	\end{equation}
	
	If we plug Inequality \eqref{eq:sum_ij} into Equation \eqref{eq_reg_est_a} and apply Youngs inequality two times for $\varepsilon\in (0,\infty)$, we derive the following inequality
	\[
	\begin{aligned}
		&\quad\int_W \lambda \norm[\big]{Q_1^{\frac{\alpha}{2}}D_1f}^2_U+\sum_{i=1}^{\infty}\norm[\big]{\lambda_{1,i}^{\frac{\alpha}{2}}K_{22}^{\frac{1}{2}}D_2\partx{i}f}_V^2\;\mathrm{d}\mu^{\Phi}\\
		&=-\int_W \big( Q_1^{\alpha}D_1f,Q_1^{-1}K_{21}D_2f\big)_U+\big(Q_1^{\frac{\alpha}{2}}D_1f,Q_1^{\frac{\alpha}{2}}D_1g\big)_U\,\mathrm{d}\mu^{\Phi}\\&\quad+\sum_{i,j=1}^{\infty}I^1_{ij}+I^2_{ij}-I^3_{ij}-I^4_{ij}\\
		&\leq  \frac{\varepsilon}{2} \int_W \norm[\big]{Q_1^{\frac{\alpha}{2}}D_1f}^2_U\;\mathrm{d}\mu^{\Phi} + \frac{1}{2\varepsilon}\int_W\norm[\big]{Q_1^{\frac{\alpha}{2}-1}K_{21}D_2f}^2_U\;\mathrm{d}\mu^{\Phi}\\
		&\quad+ \frac{\varepsilon}{2} \int_W \norm[\big]{Q_1^{\frac{\alpha}{2}}D_1f}^2_U\;\mathrm{d}\mu^{\Phi} + \frac{1}{2\varepsilon}\int_W\norm[\big]{Q_1^{\frac{\alpha}{2}}D_1g}^2_U\;\mathrm{d}\mu^{\Phi}\\
		&\quad+\frac{\norm{g}_{\infty}}{\lambda}\frac{1}{2\delta}\kappa+\frac{\norm{g}_{\infty}}{\lambda}\frac{\delta}{2}\sum_{i=1}^{\infty}\int_W \norm[\big]{\lambda_{1,i}^{\frac{\alpha}{2}}K_{22}^{\frac{1}{2}}D_2\partx{i}f}_V^2\;\mathrm{d}\mu^{\Phi}+3a_1.
	\end{aligned}
	\] 
	Using that $\norm[\big]{Q_1^{\frac{\alpha}{2}-1}K_{21}D_2f}_V\leq \kappa\norm[\big]{K_{22}^{\frac{1}{2}} D_2f}_V$ by Inequality \eqref{ineq:lfourthree} from the Assumption \nameref{ass:pot_app}, we obtain after rearranging the terms
	\[
	\begin{aligned}
		&\quad(\lambda-\varepsilon)\int_W  \norm[\big]{Q_1^{\frac{\alpha}{2}}D_1f}^2_U\;\mathrm{d}\mu^{\Phi}+\left(1-\frac{\norm{g}_{\infty}}{\lambda}\frac{\delta}{2}\right)\int_W \sum_{i=1}^{\infty}\norm[\big]{\lambda_{1,i}^{\frac{\alpha}{2}}K_{22}^{\frac{1}{2}}D_2\partx{i}f}_V^2\;\mathrm{d}\mu^{\Phi}\\
		&\leq \frac{1}{2\varepsilon}\kappa\int_W\norm[\big]{K_{22}^{\frac{1}{2}} D_2f}_V^2\;\mathrm{d}\mu^{\Phi}+ \frac{1}{2\varepsilon}\int_W\norm[\big]{Q_1^{\frac{\alpha}{2}}D_1g}^2_U\;\mathrm{d}\mu^{\Phi}+\frac{\norm{g}_{\infty}}{\lambda}\frac{1}{2\varepsilon}\kappa+3a_1\\
		&\leq  \frac{1}{2\varepsilon}\kappa\frac{1}{4\lambda} \norm{g}^2_{L^2(\mu^{\Phi})}+ \frac{1}{2\varepsilon}\int_W\norm[\big]{Q_1^{\frac{\alpha}{2}}D_1g}^2_U\;\mathrm{d}\mu^{\Phi}+\frac{\norm{g}_{\infty}}{\lambda}\frac{1}{2\delta}\kappa+3a_1.
	\end{aligned}
	\]For $\delta\defeq\frac{\lambda}{\norm{g}_{\infty}}$ and $\varepsilon\defeq \frac{\lambda}{2}$ this implies
	\[
	\begin{aligned}
		&\quad\frac{\lambda}{2}\int_W  \norm[\big]{Q_1^{\frac{\alpha}{2}}D_1f}^2_U\;\mathrm{d}\mu^{\Phi}+\frac{1}{2}\int_W \sum_{i=1}^{\infty}\norm[\big]{\lambda_{1,i}^{\frac{\alpha}{2}}K_{22}^{\frac{1}{2}}D_2\partx{i}f}_V^2\;\mathrm{d}\mu^{\Phi}\\
		&\leq\frac{\kappa}{4\lambda^2} \norm{g}^2_{L^2(\mu^{\Phi})}+ \frac{1}{\lambda}\int_W\norm[\big]{Q_1^{\frac{\alpha}{2}}D_1g}^2_U\;\mathrm{d}\mu^{\Phi}+\frac{\norm{g}_{\infty}^2}{2\lambda^2}\kappa+3a_1\eqdef a_2.
	\end{aligned}
	\]Finally, setting $a\defeq a(\lambda,g,\kappa)\defeq\frac{2a_2}{\lambda}$ concludes the proof.
\end{proof}
\begin{lem}\label{lemma:reg_estimate_D_two}
	There is a constant $b\defeq b(\lambda,g,\kappa)\in (0,\infty)$, independent of $m,n\in\N$ such that
	\[
	\int_W \lambda\norm[\big]{Q_2^{\frac{\beta}{2}}D_2f}_V^2+\sum_{i=1}^{\infty}\norm[\big]{\lambda_{2,i}^{\frac{\beta}{2}}K_{22}^{\frac{1}{2}}D_2\party{i}f}^2_V\;\mathrm{d}\mu^{\Phi}\leq b.
	\]
\end{lem}
\begin{proof}
	Let $n,m$ be fixed and substitute $\Phi_n^m$ and $f_{n,m}$ with $\Phi$ and $f$, respectively.
	Differentiating \eqref{eq:resolvent_eq_est} with respect to $e_i$ yields
	\begin{equation}
		\begin{aligned}\label{eq:reg_est_b}
			&\quad \lambda\party{i}f-L^{\Phi}\party{i}f\\& -\Big(\tr[\party{i}K_{22}D_2^2f]+\sum_{j=1}^{\infty}\left(\party{j}\party{i}K_{22}D_2f,e_j\right)_V-(v,Q_2^{-1}\party{i}K_{22}D_2f)_V\Big)\\&+(e_i,Q_2^{-1}K_{22}D_2f)_V-(e_i,Q_2^{-1}K_{12}D_1f)=\party{i}g.
		\end{aligned}
	\end{equation}
	We define the operator $S_i$ on $\fbk{3}$ by
	\[
	S_{i}h\defeq \tr[\party{i}K_{22}D_2^2h]+\sum_{j=1}^{\infty}\left(\party{j}\party{i}K_{22}D_2f,e_j\right)_V-(v,Q_2^{-1}\party{i}K_{22}D_2h)_V.
	\]
	$S_i$ is an operator of Ornstein-Uhlenbeck type with variable diffusion operator $\party{i}K_{22}$. Hence, an integration by parts formula, similar to the one from \Cref{lem:intbp}, is applicable. We then calculate
	\[
	\begin{aligned}
		-\left( L^{\Phi}\party{i}f,\party{i}f\right)_{L^2(\mu^{\Phi})}&=\int_W\left(K_{22}D_2\party{i}f,D_2\party{i}f\right)_V\;\mathrm{d}\mu^{\Phi}\quad \text{and}\\
		-\left( S_if,\party{i}f\right)_{L^2(\mu^{\Phi})}&=\int_W \left(\party{i}K_{22}D_2f,D_2\party{i}f\right)_V\;\mathrm{d}\mu^{\Phi}.
	\end{aligned}
	\]	
	
	Using the equations above, we continue by multiplying Equation \eqref{eq:reg_est_b} with $\lambda_{2,i}^{\beta}\party{i}f$, by integrating over $W$ with respect to $\mu^{\Phi}$ and by summing over all $i\in\N$, to obtain
	\[
	\begin{aligned}
		&\quad\int_W \lambda \norm[\big]{Q_2^{\frac{\beta}{2}}D_2f}^2_V+\sum_{i=1}^{\infty}\norm[\big]{\lambda_{2,i}^{\frac{\beta}{2}}K_{22}^{\frac{1}{2}}D_2\party{i}f}^2_V+\sum_{i=1}^{\infty}\lambda_{2,i}^{\beta}\left(\party{i}K_{22}D_2f,D_2\party{i}f\right)_V \\
		&\quad+( Q_2^{\beta}D_2f,Q_2^{-1}K_{22}D_2f)_V-(Q_2^{\beta}D_2f,Q_2^{-1}K_{12}D_1f)_V\;\mathrm{d}\mu^{\Phi}\\&=\int_W\big(Q_2^{\frac{\beta}{2}}D_2f,Q_2^{\frac{\beta}{2}}D_2g\big)_V\mathrm{d}\mu^{\Phi}.
	\end{aligned}
	\]
	In view of Inequality \eqref{ineq:lfourfour} and \eqref{ineq:lfourfive} from assumption \nameref{ass:pot_app} and Youngs inequality for $\delta \in (0,\infty)$, we estimate
	\[
	\begin{aligned}
		&\quad\int_W \lambda \norm[\big]{Q_2^{\frac{\beta}{2}}D_2f}^2_V+ \sum_{i=1}^{\infty}\norm[\big]{\lambda_{2,i}^{\frac{\beta}{2}}K_{22}^{\frac{1}{2}}D_2\party{i}f}^2_V\;\mathrm{d}\mu^{\Phi}\\
		&\leq\int_W \norm[\big]{Q_2^{\frac{\beta}{2}}D_2f}_V\norm[\big]{Q_2^{\frac{\beta}{2}}D_2g}_V+\sum_{i=1}^{\infty}\norm[\big]{\lambda_{2,i}^{\frac{\beta}{2}}K_{22}^{-\frac{1}{2}}\party{i}K_{22}D_2f}_V\norm[\big]{\lambda_{2,i}^{\frac{\beta}{2}}K_{22}^{\frac{1}{2}}D_2\party{i}f}_V\\
		&\quad+\norm[\big]{Q_2^{\frac{\beta}{2}}D_2f}_V\norm[\big]{Q_2^{\frac{\beta}{2}-1}K_{22}D_2f}_V+\norm[\big]{Q_2^{\frac{\beta}{2}}D_2f}_V\norm[\big]{Q_2^{\frac{\beta}{2}-1}K_{12}D_1f}_V\;\mathrm{d}\mu^{\Phi}\\
		&\leq \int_W \norm[\big]{Q_2^{\frac{\beta}{2}}D_2f}_V\kappa\left(  \norm[\big]{Q_2^{\frac{\beta}{2}}D_2g}_V+\norm[\big]{K_{22}^{\frac{1}{2}}D_2f}_V +\norm[\big]{Q_1^{\frac{\alpha}{2}}D_1f}_U \right)\;\mathrm{d}\mu^{\Phi}\\
		&\quad+\frac{1}{2}\int_W \sum_{i=1}^{\infty}\norm[\big]{\lambda_{2,i}^{\frac{\beta}{2}}K_{22}^{\frac{1}{2}}D_2\party{i}f}^2_V\;\mathrm{d}\mu^{\Phi}+\frac{1}{2}\int_W\sum_{i=1}^{\infty}\norm[\big]{\lambda_{2,i}^{\frac{\beta}{2}}K_{22}^{-\frac{1}{2}}\party{i}K_{22}D_2f}_V^2\;\mathrm{d}\mu^{\Phi}\\
		&\leq  \frac{\delta}{2}\int_W \norm[\big]{Q_2^{\frac{\beta}{2}}D_2f}_V^2\;\mathrm{d}\mu^{\Phi}\\&\quad+\frac{\kappa^2}{2\delta}\int_W\left(  \norm[\big]{Q_2^{\frac{\beta}{2}}D_2g}_V+\norm[\big]{K_{22}^{\frac{1}{2}}D_2f}_V +\norm[\big]{Q_1^{\frac{\alpha}{2}}D_1f}_U \right)^2\;\mathrm{d}\mu^{\Phi}\\
		&\quad+\frac{1}{2}\int_W \sum_{i=1}^{\infty}\norm[\big]{\lambda_{2,i}^{\frac{\beta}{2}}K_{22}^{\frac{1}{2}}D_2\party{i}f}^2_V\;\mathrm{d}\mu^{\Phi}+\frac{1}{2}\int_W\kappa^2\norm[\big]{K_{22}^{\frac{1}{2}}D_2f}_V^2\;\mathrm{d}\mu^{\Phi}.
	\end{aligned}
	\]
	The choice of $\delta={\lambda}$ leads to
	\[
	\begin{aligned}
		&\quad\int_W {\lambda} \norm[\big]{Q_2^{\frac{\beta}{2}}D_2f}^2_V+\sum_{i=1}^{\infty}\norm[\big]{\lambda_{2,i}^{\frac{\beta}{2}}K_{22}^{\frac{1}{2}}D_2\party{i}f}^2_V\;\mathrm{d}\mu^{\Phi}\\
		&\leq\kappa^2 \int_W\frac{1}{\lambda}\Big(  \norm[\big]{Q_2^{\frac{\beta}{2}}D_2g}_V+\norm[\big]{K_{22}^{\frac{1}{2}}D_2f}_V +\norm[\big]{Q_1^{\frac{\alpha}{2}}D_1f}_U \Big)^2+\norm[\big]{K_{22}^{\frac{1}{2}}D_2f}_V^2\;\mathrm{d}\mu^{\Phi}\\
		&\leq \kappa^2 \int_W\frac{4}{\lambda}\Big(  \norm[\big]{Q_2^{\frac{\beta}{2}}D_2g}^2_V+\norm[\big]{K_{22}^{\frac{1}{2}}D_2f}^2_V +\norm[\big]{Q_1^{\frac{\alpha}{2}}D_1f}^2_U \Big)+\norm[\big]{K_{22}^{\frac{1}{2}}D_2f}_V^2\;\mathrm{d}\mu^{\Phi}.
	\end{aligned}
	\]By \Cref{lemma:reg_estimate_D_one} and Inequality \eqref{varepsilonn2}, we conclude 
	\[
	\begin{aligned}
		&\quad\int_W {\lambda} \norm[\big]{Q_2^{\frac{\beta}{2}}D_2f}^2_V+\sum_{i=1}^{\infty}\norm[\big]{\lambda_{2,i}^{\frac{\beta}{2}}K_{22}^{\frac{1}{2}}D_2\party{i}f}^2_V\;\mathrm{d}\mu^{\Phi}\\
		&\leq \kappa^2 \int_W\frac{4}{\lambda}\Big(  \norm[\big]{Q_2^{\frac{\beta}{2}}D_2g}^2_V+\frac{1}{4\lambda}g ^2 +a(\lambda,g,\kappa) \Big)+\frac{1}{4\lambda}g ^2\;\mathrm{d}\mu^{\Phi}\eqdef b.
	\end{aligned}
	\]
\end{proof}
\begin{lem}\label{lem:regLfour}
	There is a constant $c\defeq c(\lambda,g,\kappa)\in (0,\infty)$, independent of $m,n\in\N$ such that
	\[
	\int_W \norm[\big]{Q_2^{{\gamma}}D_2f_{n,m}}^4_V\;\mathrm{d}\mu^{\Phi_n^m}\leq c.
	\]
\end{lem}
\begin{proof}
	Again, we fix $n,m$ and replace $\Phi_n^m$ with $\Phi$ and $f_{n,m}$ with $g$.
	For $i\in\N$ set 
	\[
	\begin{aligned}	
		p_i\defeq \int_W\lambda_{2,i}^{4\gamma}(\party{i}f)^4\;\mathrm{d}\mu^{\Phi},\quad h_i\defeq \left(\int_W\lambda_{2,i}^{4\gamma}(\party{i}^2f)^2\;\mathrm{d}\mu^{\Phi}\right)^{\frac{1}{2}}\quad \text{and}\quad
		\tilde{h}_i\defeq \sqrt[4]{3}\lambda_{2,i}^{\gamma-\frac{1}{2}} .
	\end{aligned}
	\]
	Using the integration by parts formula from \Cref{lem:intbp_classic_pot}, we obtain, by an application of the Cauchy-Schwarz and Hölder inequality ($p=\frac{4}{3}$ and $q=4$),
	\[
	\begin{aligned}
		p_i&=\int_W \lambda_{2,i}^{4\gamma}(\party{i}f)^3\party{i}f\;\mathrm{d}\mu^{\Phi}\\&=\int_W -\lambda_{2,i}^{4\gamma}3(\party{i}f)^2\party{i}^2ff+\lambda_{2,i}^{4\gamma}(\party{i}f)^3f (v,Q_2^{-1}e_i)_V\;\mathrm{d}\mu^{\Phi}\\
		&\leq 3\frac{\norm{g}_{\infty}}{\lambda} \left(\int_W\lambda_{2,i}^{4\gamma}(\party{i}f)^4\;\mathrm{d}\mu^{\Phi}\right)^{\frac{1}{2}}\left(\int_W\lambda_{2,i}^{4\gamma}(\party{i}^2f)^2\;\mathrm{d}\mu^{\Phi}\right)^{\frac{1}{2}}\\
		&\quad+\frac{\norm{g}_{\infty}}{\lambda} \left(\int_W\lambda_{2,i}^{\frac{16}{3}\gamma}(\party{i}f)^4\;\mathrm{d}\mu^{\Phi}\right)^{\frac{3}{4}}\left(\int_W(v,Q_2^{-1}e_i)_V^4\;\mathrm{d}\mu^{\Phi}\right)^{\frac{1}{4}}\\
		&=3\frac{\norm{g}_{\infty}}{\lambda} p_i^{\frac{1}{2}}\left(\int_W\lambda_{2,i}^{4\gamma}(\party{i}^2f)^2\;\mathrm{d}\mu^{\Phi}\right)^{\frac{1}{2}}+\frac{\norm{g}_{\infty}}{\lambda} (\lambda_{2,i}^{\frac{4}{3}\gamma})^{\frac{3}{4}} p_i^{\frac{3}{4}}\lambda_{2,i}^{-1}\left(\int_W(v,e_i)_V^4\;\mathrm{d}\mu^{\Phi}\right)^{\frac{1}{4}}\\
		&=3\frac{\norm{g}_{\infty}}{\lambda} p_i^{\frac{1}{2}}h_i+\frac{\norm{g}_{\infty}}{\lambda} p_i^{\frac{3}{4}}\tilde{h}_i.
	\end{aligned}
	\]
	In the last equality above, we used that $\int_W(v,e_i)_V^4\;\mathrm{d}\mu^{\Phi}=\int_V(v,e_i)_V^
	4\;\mathrm{d}\mu_2=3\lambda_{2,i}^{2}$, which follows by \Cref{lem:Gaussianmoments}.
	Dividing by $p_i^{\frac{1}{2}}$ and using that $xy\leq \frac{1}{2}(x^2+y^2)$ for all $x,y\in\R$, we get
	\[
	p_i^{\frac{1}{2}}\leq 3\frac{\norm{g}_{\infty}}{\lambda}h_i+\frac{\norm{g}_{\infty}}{\lambda} p_i^{\frac{1}{4}}\tilde{h}_i\leq3\frac{\norm{g}_{\infty}}{\lambda}h_i+\frac{1}{2} p_i^{\frac{1}{2}}+\frac{1}{2}\frac{\norm{g}^2_{\infty}}{\lambda^2}\tilde{h}^2_i.
	\]
	Set $A\defeq \max\left\{6\frac{\norm{g}_{\infty}}{\lambda},\frac{\norm{g}_{\infty}^2}{\lambda^2}\right\}$. Then it holds
	\[
	p_i^{\frac{1}{2}}\leq A\left( h_i+\tilde{h}^2_i\right).
	\]
	Lastly, we conclude with the Inequalities \eqref{ineq:lfoursix} and \eqref{ineq:lfourseven} from Assumption \nameref{ass:pot_app} and \Cref{lemma:reg_estimate_D_two}
	\[
	\begin{aligned}
		&\quad\int_W \norm[\big]{Q_2^{\gamma}D_2f}^4_V\;\mathrm{d}\mu^{\Phi}=\int_W \left(\sum_{i=1}^{\infty} \lambda_{2,i}^{2\gamma}(\party{i}f)^2\right)^2\;\mathrm{d}\mu^{\Phi}\\
		&=\sum_{i,j=1}^{\infty}\int_W  \lambda_{2,i}^{2\gamma}(\party{i}f)^2 \lambda_{2,j}^{2\gamma}(\party{j}f)^2\;\mathrm{d}\mu^{\Phi}
		\leq \left(\sum_{i=1}^{\infty} p_i^{\frac{1}{2}}\right)^2\leq A^2\left(\sum_{i=1}^{\infty} h_i+\tilde{h}^2_i\right)^2\\
		&\leq  2{A^2} \left(\sum_{i=1}^{\infty} h_i\right)^2+ 2{A^2} \left(\sum_{i=1}^{\infty} \tilde{h}^2_i\right)^2
		\leq  2{A^2} \kappa(b+{3}) \eqdef c.
	\end{aligned}
	\]
\end{proof}
As we have our desired $L^4(W;\mu^{\Phi_m^n})$ regularity estimate at hand, we are able to show that the Kolmogorov operator $\Lphic$ is essentially m-dissipative on $L^2(W;\mu^{\Phi})$. In \Cref{ch:example_Lfour_ess_diss} we apply this result in a degenerate stochastic reaction-diffusion setting.
%
\begin{thm}\label{theo:ess_diss_phi_Lfour}
	Let Assumption \nameref{ass:pot_app} be valid. Then, $(L^{\Phi},\fbs{B_W})$ is essentially m-dissipative on $L^2(W;\mu^{\Phi})$. The corresponding strongly continuous contraction semigroup $\sccs$ is sub-Markovian and conservative. 
\end{thm}
\begin{proof}
	By \Cref{lem:intbp}, we already know that $(L,\fbs{B_W})$ is dissipative on $L^2(W;\mu^{\Phi})$ and therefore closable in $L^2(W;\mu^{\Phi})$, with closure denoted by $(L^{\Phi},D(L^{\Phi}))$.
	To apply the Lumer-Phillips theorem, it is left to show that $(\lambda-L^{\Phi})(D(L^{\Phi}))$ is dense in $L^2(W;\mu^{\Phi})$ for some $\lambda\in (0,\infty)$. Since $\fbs{B_W}$ is dense in $L^2(W;\mu^{\Phi})$, it is enough to show that there exits $\lambda\in (0,\infty)$ such that $(\lambda-L^{\Phi})(D(L^{\Phi}))$ contains $\fbs{B_W}$. So take $\lambda\in (0,\infty)$ from Item App($\Phi3$)  of Assumption \nameref{ass:pot_app}. Further, let $g\in\fbs{B_W}$ and $m,n\in\N$, where we assume without loss of generality that $m^K(n)=n$ and $m^K(m)=m$.
	In view of Item App($\Phi 3$) from Assumption \nameref{ass:pot_app}, there exists $f_{n,m}\in \fbkn{3}$ with 
	\[
	\begin{aligned}
		g&=\lambda f_{n,m}-L^{\Phi^m_n}f_{n,m}=\lambda f_{n,m}-Lf_{n,m}+(K_{12}D\Phi_n^m,D_2f_{n,m})_V\\	 
		&=\lambda f_{n,m}-L^{\Phi}f_{n,m}+(K_{12}(D\Phi_n^m-P_nD\Phi),D_2f_{n,m})_V.
	\end{aligned}
	\]If $(K_{12}(D\Phi_n^m-P_nD\Phi),D_2f_{n,m})_V$ gets arbitrary small in $L^2(W;\mu^{\Phi})$ for big $n,m$, the dense range condition is shown.
	
	It holds, by an application of the generalized Hölder inequality ($\frac{2}{p^*}+\frac{1}{q^*}+\frac{1}{2}=1$),
	\[
	\begin{aligned}
		&\quad\int_W (K_{12}(D\Phi_n^m-P_nD\Phi),D_2f_{n,m})^2_V\;\mathrm{d}\mu^{\Phi}\\
		&\leq {\mu_1(e^{-\Phi_n^m})}\int_W\norm[\big]{Q_2^{-\gamma}K_{12}(D\Phi_n^m-P_nD\Phi)}^2_V\norm[\big]{Q_2^{\gamma}D_2f_{n,m}}^2_Ve^{\Phi_n^m-\Phi}\;\mathrm{d}\mu^{\Phi_n^m}\\
		&\leq \mu_1(e^{-\Phi_n^m})\left(\int_U\norm[\big]{Q_2^{-\gamma}K_{12}(D\Phi_n^m-P_nD\Phi)}^{p^*}_V\;\mathrm{d}\mu_1^{\Phi_n^m}\right)^{\frac{2}{p^*}}\\&\quad\times \left(\int_W\norm[\big]{Q_2^{\gamma}D_2f_{n,m}}^4_V\;\mathrm{d}\mu^{\Phi_n^m}\right)^{\frac{1}{2}}\left( \int_W e^{q^*(\Phi_n^m-\Phi)}\;\mathrm{d}\mu^{\Phi_n^m}\right)^{\frac{1}{q^*}}.
	\end{aligned}
	\]By Item App($\Phi 5$) from Assumption \nameref{ass:pot_app}, we know that \[\li{n}\li{m} \int_U \norm[\big]{Q_2^{-\gamma}K_{12}(D\Phi_n^m-P_nD\Phi)}_V^{p^*}\;\mathrm{d}\mu^{\Phi_n^m}_1=0.\] 
	To conclude the dense range condition, it is enough to bound
	\[
	\mu_1(e^{-\Phi_n^m})\left(\int_W\norm[\big]{Q_2^{\gamma}D_2f_{n,m}}^4_V\;\mathrm{d}\mu^{\Phi_n^m}\right)^{\frac{1}{2}}\left( \int_W e^{q^*(\Phi_n^m-\Phi)}\;\mathrm{d}\mu^{\Phi_n^m}\right)^{\frac{1}{q^*}}
	\]independent of $m,n\in\N$.
	To verify this, we argue as follows. \Cref{lem:regLfour} implies
	\[\int_W\norm[\big]{Q_2^{\gamma}D_2f_{n,m}}^4_V\;\mathrm{d}\mu^{\Phi_n^m}\leq c(\lambda,g,\kappa)\]
	for all $n,m$. Moreover, by Inequality \eqref{eq:phig4_one}, we get $\mu_1(e^{-\Phi_n^m})\leq e^{-c_1}$ independent of $m,n$. Finally, using Inequality \eqref{eq:phig4_one} from Assumption \nameref{ass:pot_app}, we get
	\[
	\begin{aligned}
		\mu_1(e^{-\Phi_n^m})\left(\int_W e^{q^*(\Phi_n^m-\Phi)}\;\mathrm{d}\mu^{\Phi_n^m}\right)^{\frac{1}{q^*}}&= \mu_1(e^{-\Phi_n^m})^{1-\frac{1}{q^*}} \left(\int_U e^{(q^*-1)\Phi_n^m-q^*\Phi}\;\mathrm{d}\mu_1\right)^{\frac{1}{q^*}}\\
		&\leq (e^{-c_1})^{1-\frac{1}{q^*}}\left( \int_U e^{c_2+c_3\norm{u}^2_U}\;\mathrm{d}\mu_1\right)^{\frac{1}{q^*}}.
	\end{aligned}
	\]
	The estimate in the last inequality above is valid as $q^*>1$ and consequently $1-\frac{1}{q^*}>0$.
	By means of \Cref{theo:fernique}, we know that the right-hand side of the inequality above is bounded. This concludes the proof.

	For the second part of the proof, recall that
	\Cref{lem:inf-dim-op-decomp} tells us that $L^{\Phi}1=0$ and $\mu(L^{\Phi}f)=0$ for all $f\in\fbs{B_W}$. The former implies $T_t1=1$ in $L^2(W;\mu^{\Phi})$ for all $t\geq 0$, while the latter shows that $\mu^{\Phi}$ is invariant for $L^{\Phi}$ and consequently for $\sccs$.
	By \Cref{cor:invariance_ASP}, we also know that $(L^{\Phi},\fbs{B_W})$ is an abstract diffusion operator, which implies as explained in \Cref{sec:inf_langevin_general} that $\sccs$ is sub-Markovian and conservative.
\end{proof}

Item App($\Phi3$) from Assumption \nameref{ass:pot_app}, which is necessary to apply \Cref{theo:ess_diss_phi_Lfour} is verified in an finite dimensional setting in \cite[Prop.~2.2]{dPL06}. Indeed the authors assume  $U=V=\R^d$ for some $d\in\N$, $Q_2=K_{12}=K_{22}= \id$, $Q_1$ is a symmetric positive matrix and the potential can be approximated in $L^4$ by a sequence of $C^4$ functions whose gradients have bounded derivatives up to order three, compare \cite[Hypothesis~2.1]{dPL06}. 


Below, we present a strategy inspired by \cite[Prop.~2.2]{dPL06} to generalize the results from \cite[Prop.~2.2]{dPL06} for our setting. To do that the following definition is useful.
\begin{defn}\label{def:inf-dim-fin-dim}
	Fix $n\in\N$ such that $n=m^K(n)$. We define for all $y\in\R^n$
	\[
	\begin{aligned}
		K_{12,n}\defeq ((K_{12}d_i,e_j)_V)_{ij},
		\quad
		K_{21,n}\defeq (K_{12,n})^*
		,\quad 	K_{22,n}(y)\defeq \left(\left(K_{22}(\overline{p}^V_ny)e_i,e_j\right)_V\right)_{ij}
	\end{aligned}
	\]
	and denote the entry of $K_{22,n}$ at position $i,j$ by $k_{ij,n}$.
	
	Moreover, we define $S_n$, $A_n$ and $L_n$ on the Hilbert space $L^2(\R^n\times\R^n;\mu^n)$ with domain $C_b^\infty(\R^n\times\R^n)$ by
	\[
	\begin{aligned}
		S_nf(x,y) &\defeq \tr[K_{22,n} D_2^2 f](x,y) + \sum_{i,j=1}^{n} \partial_j k_{ij,n}(y)\party{i}f(x,y)
		\\&\quad- \langle K_{22,n}(y)Q_{2,n}^{-1}y,D_2 f(x,y) \rangle, \\
		A_nf(x,y) &\defeq \langle K_{12,n}Q_{1,n}^{-1}x,D_2 f(x,y)\rangle - \langle K_{21,n}Q_{2,n}^{-1}y,D_1 f(x,y)\rangle \quad\text{ and }\\
		L_nf &\defeq (S_n-A_n)f.
	\end{aligned}
	\]
	Here $\langle \cdot,\cdot\rangle$ and $\abs{\cdot}$ denotes the Euclidean inner product and norm on $\R^n$, respectively.
	
\end{defn}


\begin{cond}{\textbf{K0}}\label{ass:inf-dim-elliptic}
	Assume that there is some positive operator $K_{22}^0\in \lopp{V}$, which leaves each $V_{m_k}$ for all $k\in\N$ invariant and such that
	\[
	(v,K_{22}(\tilde{v})v)_V \geq (v,K_{22}^0 v)_V
	\quad\text{ for all }v,\tilde{v}\in V.
	\]
\end{cond}

\begin{rem}\label{rem:ktwotwo_last_minute}
	For each fixed $i,j\in\N$ the boundedness and invariance properties of \Cref{def:inf-dim-operators} imply that $(K_{22}(v)e_i,e_j)_V$ is uniformly bounded in $v\in V$. In view of \cite[Theorem 1.1]{DELMORAL2018259}, one can show that $(K_{22}^{\frac12}e_i,e_j)_V$ is Fr\`echet differentiable with bounded derivative, if Assumption \nameref{ass:inf-dim-elliptic} holds true.
	More general, suppose $v\mapsto K_{22}(v)e_k\in C_b^l(V;{V})$ for all $k\in\N$ and some $l\in\N$. Then for each $i,j\in\N$ it holds $(K_{22}^{\frac12}e_i,e_j)_V\in C^l(V;\R)$. Moreover, if  Assumption \nameref{ass:inf-dim-elliptic} holds true, then $(K_{22}^{\frac12}e_i,e_j)_V\in C^l_b(V;\R)$.
\end{rem}

\begin{rem}\label{lem:sol_finite}
	Suppose $ v\mapsto K_{22}(v)e_i\in C^4_b(V;{V})$ for all $i\in\N$ and assume that the sequence $(\Phi_n^m)_{m,n\in\N}$ from Assumption \nameref{ass:pot_app} is in $C^4(U;\R)$. Moreover, suppose $D\Phi_n^m$ and $V\ni v\mapsto K_{22}(P_{n}(v))Q_2^{-1}P_{n}(v)\in V$ have bounded derivatives up to order three for all $m,n\in\N$. Lastly, let Assumption \nameref{ass:inf-dim-elliptic} be valid.
	
	Given $\lambda\in (0,\infty)$ and $g\in \fbs{B_W}$.
	By a trivial extension procedure, there is some $n\in\left\lbrace m_k,m_{k+1}...\right\rbrace$ and $\psi\in C_b^{\infty}(\R^n\times \R^n)$ such that $g=\psi(p_n^U,p_n^V)\in \fbs{B_W}$.  Set $\overline{\Phi_n^m}\defeq \Phi_n^m\circ\overline{p}^U_n$ and define for $\varphi\in C_b^{2}(\R^n\times \R^n)$ and $(x,y)\in \R^n\times \R^n$
	
	\[
	\begin{aligned}
		L^{\overline{\Phi_n^m}}\varphi(x,y)\defeq L_n\varphi(x,y)-\langle D\overline{\Phi_n^m}(x),K_{21,n}D_2\varphi(x,y)\rangle,
	\end{aligned}
	\]where $L_n$ and  $K_{21,n}$ are given as in \Cref{def:inf-dim-fin-dim}.
	%

	Recall the matrix valued map $K_{22,n}$ considered in \Cref{def:inf-dim-fin-dim} and define the maps  $a:\R^n\times \R^n\to\lop{\R^n\times \R^n}$ and $b:\R^n\times \R^n\to \R^n\times \R^n$ via
	\[
	\begin{aligned}
		a\begin{pmatrix} x\\ y\end{pmatrix}&\defeq \begin{pmatrix}
			0 &0 \\ 0&\sqrt{2K_{22,n}(y)}
		\end{pmatrix}\quad\text{and}\\
		b\begin{pmatrix} x\\ y\end{pmatrix}&\defeq\begin{pmatrix}
			K_{21,n}Q_{2,n}^{-1}y\\
			-K_{12,n}Q_{1,n}^{-1}x+\sum_{j=1}^n \partial_jK_{22,n}(y)p_n^Ve_j- K_{22,n}(y)Q_{2,n}^{-1}y-K_{12,n}D\overline{\Phi_n^m}(x)
		\end{pmatrix}.
	\end{aligned}
	\]
	Using the It\^o formula, we see that the operator $L^{\overline{\Phi_m^n}}$ corresponds to the finite dimensional stochastic differential equation given by
	\begin{equation}\label{eq:sde_finite_rewritten}
		\begin{aligned}
			\mathrm{d} &\begin{pmatrix} X_t\\ Y_t\end{pmatrix} =b\begin{pmatrix} X_t\\ Y_t\end{pmatrix}+a\begin{pmatrix} X_t\\ Y_t\end{pmatrix}\,\mathrm{d}W_t,\quad \begin{pmatrix} X_0\\ Y_0\end{pmatrix}=\begin{pmatrix} x\\ y\end{pmatrix},
		\end{aligned}
	\end{equation} with $(W_t)_{t\geq 0}$ being a finite dimensional standard Wiener process on some probability space $(\Omega,\mathcal{A},\p)$ with values in $\R^n\times \R^n$.
	By means of Assumption \nameref{ass:inf-dim-elliptic} and \Cref{rem:ktwotwo_last_minute} we obtain $a\in C_b^4(\R^n\times \R^n; \lop{\R^n\times \R^n})$.
	Taking the assumptions on the potential and the coefficients into account, it follows that both $a$ and $b$ are Lipschitz continuous. Therefore, the finite dimensional stochastic differential equation \eqref{eq:sde_finite_rewritten} has a unique global solution in terms of a time-homogeneous Markov process $(X_t(x,y),Y_t(x,y))_{t\geq 0}$ with $(X_0(x,y),Y_0(x,y))=(x,y)$, compare~\cite[Chapter 3 Paragraph 15]{GihSko}.
	
	As $v\mapsto K_{22}(v)e_i\in C^4_b(V;{V})$ and $D\Phi_m^n$, as well as $V\ni v\mapsto K_{22}(P_{n}(v))Q_2^{-1}P_{n}(v)\in V$ have bounded derivatives up to order three, we know that $b\in C^3(\R^n\times \R^n;\R^n\times \R^n)$ and $b$ has bounded derivatives up to order three.  By a dependence upon initial data result, compare e.g.~\cite[Theorem 1, p. 61]{GihSko},
	we conclude that for all $t\in (0,\infty)$, the map  
	\[
	\R^n\times \R^n\ni(x,y)\mapsto ( X_t(x,y),Y_t(x,y))\in\R^n\times \R^n
	\]
	is in $C^3(\R^n\times \R^n;\R^n\times \R^n)$ with bounded derivatives up to the third order. 
	Consequently, the associated transition semigroup $(p_t)_{t\geq 0}$  leaves $C_b^3(\R^n\times \R^n;\R)$ invariant. 
	
	Let $(x,y)\in \R^n\times \R^n$ and define
	\begin{align}
		\label{eq:laplace_finite}
		\varphi(x,y)\defeq\int_0^{\infty}e^{-\lambda t}	\mathbb{E}[\psi( X_t(x,y),Y_t(x,y))]\,\mathrm{d}t=\int_0^{\infty}e^{-\lambda t}	p_t\psi(x,y)\,\mathrm{d}t,
	\end{align}i.e.~$\varphi$ is the transition resolvent corresponding to $(p_t)_{t\geq 0}$.
	For $\lambda >0 $ large enough one has $\varphi\in C_b^3(\R^n\times \R^n;\R)$ and $\lambda \varphi-L^{\overline{\Phi_n^m}}\varphi=\psi$. In this case, the function $f\in \fbk{3}$ defined via $f(u,v)\defeq \varphi(P_n^Uu,P_nVv)$, $(u,v)\in W$ fulfills 
	\[
	\lambda f(u,v)-L^{\Phi_n^m}f(u,v)=\lambda \varphi(p_n^Uu,p_n^Vv)-L^{\overline{\Phi_n^m}}\varphi(p_n^Uu,p_n^Vv)=\psi(p_n^Uu,p_n^Vv)=g(u,v).
	\]
	Using the representation from Equation \eqref{eq:laplace_finite}, also the inequality $\norm{f}_{\infty}\leq \frac{1}{\lambda}\norm{g}_{\infty}$ is directly derived.

\end{rem}

\section{Hypocoercivity }\label{sec:hypo_appl}
The presentation in this chapter is based on the already published articles \cite{EG21_Pr} and \cite{EG23} and follows the general abstract hypocoercivity strategy described in \cite{GS14} and \cite{GS16}. To give the full picture we nevertheless present the results and give some proofs to emphasize how the general abstract hypocoercivity method can be applied in our infinite dimensional setting and potentials with unbounded non Lipschitz continuos gradients. The strength of our results is highlighted in \Cref{ch:example_Lfour_ess_diss}, where we consider examples of degenerate semi-linear infinite dimensional stochastic differential equations beyond framework discussed in \cite{W17}.

We consider the setting as described in \Cref{sec:inf_langevin_general}, where we require that
$\Phi:U\rightarrow (-\infty,\infty]$ is bounded from below and there is $\theta\in [0,\infty)$ such that $\Phi\in W_{Q_1^{\theta}}^{1,2}(U;\mu_1)$. For simplicity, we again assume without loss of generality that $\Phi$ is bounded from below by zero and normalized.

We begin with a situation where the Kolmogorov operator $L^{\Phi}$ is essentially m-dissipative and has a nice core with corresponding decomposition into a symmetric part $S$ and antisymmetric part $A^{\Phi}$, respectively. Therefore, we assume Assumptions \nameref{ass:pot_app} to apply \Cref{theo:ess_diss_phi_Lfour}. 

\bigskip

We additionally assume the following assumption throughout this chapter.
\begin{cond}{\textbf{K1}}\label{ass:qc-negative-type}
	The operator $K_{21}K_{12}=K_{12}^*K_{12}$ is positive on $U$.
\end{cond}

Starting from here, our goal is to show that the strongly continuous sub-Markovian semigroup $\sccs$, generated by $(L^{\Phi},D(L^{\Phi}))$ is hypocoercive, where the exponential speed of convergence to the equilibrium is determined by explicitly computable constants. 
\begin{defn}\label{def:OU_with_pot}
	The operators $(C,D(C))$ and $(Q_1^{-1}C,D(Q_1^{-1}C))$ on $U$ are defined by
	\[
	\begin{aligned}
		C&\defeq K_{21}Q_2^{-1}K_{12}\quad\text{with}\quad &D(C)&\defeq \{u\in U\mid K_{12}u\in D(Q_2^{-1}) \}\quad\text{and}\\
		Q_1^{-1}C&\defeq Q_1^{-1}K_{21}Q_2^{-1}K_{12}\quad\text{with}\quad &D(Q_1^{-1}C)&\defeq \{u\in D(C)\mid Cu\in D(Q_1^{-1}) \},
	\end{aligned}
	\]
	respectively. 
	Moreover, we define the infinite dimensional Ornstein-Uhlenbeck operator $(N^{\Phi},\fbs{B_U})$ (perturbed by the gradient of $\Phi$)  by
	\begin{align*}
		&N^{\Phi}:\fbs{B_U}\to L^2(U;\mu_1^{\Phi}),\\
		&f\mapsto N^{\Phi}f\defeq\mathrm{tr}[CD^2f]-(u,Q_1^{-1}CD f)_U-(D\Phi,CDf)_U.
	\end{align*}We use the abbreviation $N\defeq N^0$.
\end{defn}

We continue with the collection of a few properties of the newly defined operators.
\begin{rem}\label{rem:prop_C}
	\begin{enumerate}[(i)]
		\item $(C,D(C))$ is symmetric and positive on $U$.
		\item For all $n\in\N$, there is $m_k$ with $n\leq m_k$ such that $(C,D(C))$ maps $U_n$ into $U_{m_k}$. Recall that $(m_k)_{k\in\N}$ is the sequence from \Cref{def:inf-dim-operators}.
	\end{enumerate}
\end{rem}
\begin{prop}{\cite[Theorem 7.5 and Lemma 7.7]{EG23}}\label{prop:ess_self_N}
	The operator $(N^{\Phi},\fbs{B_U})$ is symmetric on $L^2(U;\mu_1^{\Phi})$ and therefore closable with closure denoted by $(N^{\Phi},D(N^{\Phi}))$. Moreover $D(N^{\Phi})\subseteq W_{C^{\frac12}}^{1,2}(U,\mu_1)$ and for every $\lambda\in (0,\infty)$, $f\in D(N^{\Phi})$ and $g\defeq \lambda f-N^{\Phi}f$ we have
	\begin{align}
		\label{eq:regularity-estimate-one-two}		\int_U\norm{C^{\frac12}Df}^2_U\,\mathrm{d}\mu_1^{\Phi}\leq \frac{1}{4\lambda}\int_U g^2\,\mathrm{d}\mu_1^{\Phi}.
	\end{align}
\end{prop}

\begin{cond}{SA($\Phi$)}\label{ass:ess-N_Phi}
$(N^{\Phi},\fbs{B_U})$ is essentially self-adjoint on $L^2(U;\mu_1^{\Phi})$. The resolvent in $\lambda\in (0,\infty)$ of the corresponding closure $(N^{\Phi},D(N^{\Phi}))$ is denoted by $R_{\lambda}^{N^{\Phi}}$.

\end{cond}

The next assumption enables us to generalize the regularity estimates from \Cref{prop:ess_self_N} to the case where $N$ is perturbed by the gradient of the potential $\Phi$. 
\begin{cond}{{Reg}(${{\Phi}}$)}\label{ass:L_Phi-convex-reg-est}$\Phi=\Phi_1+\Phi_2:U\to (-\infty,\infty]$.
\begin{enumerate}[label=Class \Alph*:, wide=0pt, leftmargin=\widthof{\textbf{Class B:\hspace*{\labelsep}}}%
]
\item[Reg(${{\Phi_1}}$)]  
There exists a sequence $(\Phi_{1,n,m})_{m,n\in\N}$ of convex functions from $U$ to $\R$ such that for $\mu_1$-almost all $u\in U$ and for all $m,n\in\N$
\begin{enumerate}[label=Class \Alph*:, wide=0pt, leftmargin=\widthof{\textbf{Class B:\hspace*{\labelsep}}}%
	]
	\item[(i)] $-\infty<\inf_{\tilde{u}\in U}\Phi_1(\tilde{u})\leq\Phi_{1,n,m}(u)$,
	\item[(ii)]  $ \lim_{n\rightarrow \infty}\lim_{m\rightarrow \infty}\Phi_{1,n,m}(u)=\Phi_1(u)$ , 
	\item[(iii)] $\Phi_{1,n,m}$ is differentiable with Lipschitz continuous derivative and 
	\[
	\lim_{n\rightarrow \infty}\lim_{m\rightarrow \infty}\norm{(Q_1^{\theta}D\Phi_{1,n,m}-Q_1^{\theta}D\Phi_{1},d_k)_U}_{L^2(\mu_1)}=0 \;\text{for all}\; k\in\N.
	\] 
\end{enumerate}
\item[Reg(${{\Phi_2}}$)] $\Phi_2$ is bounded and two times continuously Fr\'{e}chet differentiable with bounded first order derivative and second order derivative in $L^1(U;\mu_1^{\Phi})$. Moreover, there is a constant $c_{\Phi_2}\in [0,\infty)$ such that for all $n\in\N$ 
\[
(D^2\Phi_2(\tilde{u})Cu,Cu)_U\geq -c_{\Phi_2}(Cu,u)_U\quad\text{for\;all}\quad \tilde{u}\in U\quad\text{and}\quad u\in 						U_n. 
\]

%

\end{enumerate}
\end{cond}
\begin{rem}\label{rem:weak_conv_measure}
\begin{enumerate}[(i)]
\item 		Suppose Item (i) and (ii) of Item Reg(${{\Phi_1}}$) hold true. Then, $e^{-\Phi_{1,n,m}}$ and $e^{-\lim_{m\to \infty}\Phi_{1,n,m}}$ are bounded by $e^{-\inf_{\tilde{u}\in U}\Phi_1(\tilde{u})}$. Since we also have pointwise convergence, we obtain by the theorem of dominated convergence  	\[
\lim_{n\rightarrow \infty}\lim_{m\rightarrow \infty}\mu_1^{\Phi_{1,n,m}}(f)=\mu_1^{\Phi_1}(f)\quad \text{for all}\quad f\in L^1(U;\mu_1^{\Phi_1}).
\]
\item 

Note that Item Reg(${{\Phi_2}}$) above is satisfied, if $D^2\Phi_2$ is bounded and $C\in \mathcal{L}(U)$.
\end{enumerate}
\end{rem}

The proof of the second order regularity estimate below is similar to \cite[Theorem~2]{EG21}, where only convex potentials were considered. 
\begin{lem}\label{lem:inf-dim-regularity-estimates}

Assume that Assumption \nameref{ass:L_Phi-convex-reg-est} holds true. Then for every function $f\in \fbs{B_U}$ and $g\defeq f-N^{\Phi}f$ we have the following second order regularity estimate
\[
\begin{aligned}\label{eq:regularity-estimate-two}
\int_U \tr[(CD^2f)^2]+\norm{Q_1^{-\frac{1}{2}}CDf}^2_U\,\mathrm{d}\mu_1^{\Phi}
&\leq \left(4+\frac{c_{\Phi_2}}{4}\right) \int_U g^2\,\mathrm{d}\mu_1^{\Phi}.
\end{aligned}	
\]
\end{lem}
\begin{proof}
For each $m,n\in\N$ define $\Phi_{n,m}\defeq \Phi_{1,n,m}+\Phi_2$ where $(\Phi_{1,n,m})_{n,m\in \N}$ is the approximation of $\Phi_1$, provided by Assumption \nameref{ass:L_Phi-convex-reg-est}. Further, let $f\in \fbs{B_U}$ and set \[g_{n,m}\defeq f-N^{\Phi_{n,m}}f.\] Taking derivatives of the equation above, with respect to the $d_k$, gives
\begin{equation}\label{eq:lambda_n_phi}
\partx{k} f-N^{\Phi_{n,m}}\partx{k}f+(d_k,Q_1^{-1}CD f)_U+\sum_{i=1}^{\infty} (d_i,CDf)_U\partx{k}\partx{i}\Phi_{n,m}=\partx{k} g_{n,m}.
\end{equation}The infinite sum in Equation \eqref{eq:lambda_n_phi} above is a finite one. Moreover, $\partx{k}\partx{i}\Phi_{n,m}$ exists $\mu_1$-a.e., since Lipschitz continuous function $\partx{i}\Phi_{n,m} :U\rightarrow \R$ are Gateaux differentiable $\mu_1$-a.e..
We multiply Equation \eqref{eq:lambda_n_phi} with $\partx{l} f(d_k,Cd_l)_U$. Summing over all $k,l\in\N$ shows that the first and third term, as well as the right-hand side in Equation \eqref{eq:lambda_n_phi}, is equal to $(CD f,D f)_U$, $\norm{Q_1^{-\frac{1}{2}}CDf}^2_U$ and $(D g_{n,m},CDf)_U$, respectively. For the second term we calculate
\begin{align*}
&\quad\sum_{k,l=1}^{\infty}(d_k,Cd_l)_U\int_U -N^{\Phi_{n,m}}\partx{k} f \partx{l} f\dm_1^{\Phi_{n,m}}\\&=\sum_{k,l=1}^{\infty}(d_k,Cd_l)_U\int_U (CD \partx{k} f, D \partx{l} f)_U\dm_1^{\Phi_{n,m}}\\
&=\int_U \mathrm{tr}[(CD^2f)^2]\dm_1^{\Phi_{n,m}}.
\end{align*}Additionally, we have
\begin{align*}
\sum_{k,l,i=1}^{\infty}(d_i,CDf)_U\partx{l}f(d_k,Cd_l)_U\partial_{k}\partial_{i}\Phi_{n,m} 
=(D^2\Phi_{n,m} CDf,CDf)_U.
\end{align*}Putting the results we just derived into Equation \eqref{eq:lambda_n_phi} and rearranging the terms, establishes the following equation
\begin{align*}
&\quad\int_U\norm{Q_1^{-\frac{1}{2}}CDf}^2_U+\mathrm{tr}[(CD^2f)^2]+(D^2\Phi_{n,m} CDf,CDf)_U\dm_1^{\Phi_{n,m}}\\
&=\int_U (D(-N^{\Phi_{n,m}}f),CDf)_U\dm_1^{\Phi_{n,m}}.
\end{align*}
One can show that $N^{\Phi_{n,m}}f\in W^{1,2}(U;\mu_1)$, as $\Phi_{n,m}$ is differentiable with Lipschitz continuos derivative and \cite[Proposition 10.11]{dP06}. Together with the fact that $\partx{j}f{e^{-\Phi_{n,m}}}\in W^{1,2}(U;\mu_1)$ for all $j\in\N$ and the integration by parts formula from \Cref{lem:inf-dim-ibp-fbs}, we obtain
\begin{align*}
&\quad \int_U \left(D(-N^{\Phi_{n,m}}f),CDf\right)_U\;\dm_1^{\Phi_{n,m}}\\
&=\frac{1}{\mu_1(e^{-\Phi_{n,m}})}\sum_{i,j=1}^{\infty}(Cd_i,d_j)_U\int_U -\partx{i}N^{\Phi_{n,m}}f\partx{j}f{e^{-\Phi_{n,m}}}\;\dm_1\\
&=\sum_{i,j=1}^{\infty}(Cd_i,d_j)_U\int_U N^{\Phi_{n,m}}f\left(\partx{i}\left(\partx{j}f{e^{-\Phi_{n,m}}}\right)-(u,Q_1^{-1}d_i)_U\partx{j}f\right)\;\dm_1^{\Phi_{n,m}}\\
&=\int_U (N^{\Phi_{n,m}}f)^2\dm_1^{\Phi_{n,m}}.
\end{align*}
Since the resolvent $R_1^{N^{\Phi_{n,m}}}$ is a contraction, we estimate \[\int_U f^2\dm_1^{\Phi_{n,m}}=\int_U (R_1^{N^{\Phi_{n,m}}}g_{n,m})^2\dm_1^{\Phi_{n,m}}\leq \int_U g_{n,m}^2\dm_1^{\Phi_{n,m}}.\]This implies
\[
\begin{aligned}
&\quad\int_U \tr[(CD^2f)^2]+\norm{Q_1^{-\frac{1}{2}}CDf}^2_U+(D^2\Phi_{n,m}CDf,CDf)_U\,\mathrm{d}\mu_1^{\Phi_{n,m}}\\
&= \int_U (N^{\Phi_{n,m}}f)^2\dm_1^{\Phi_{n,m}}=\int_U (f-g_{n,m})^2\dm_1^{\Phi_{n,m}} \leq4 \int_U g_{n,m}^2\,\mathrm{d}\mu_1^{\Phi_{n,m}}.
\end{aligned}
\]
Note that $\partx{i}\partx{j}\Phi_{1,n,m}$ and $\partx{i}\partx{j}\Phi_{2}$ exist in $L^1(\mu_1^{\Phi}),$ since $D\Phi_{1,n,m}$ is Lipschitz continuous and $D^2\Phi_2$ is $\mu^{\Phi}_1$ integrable by Assumption \nameref{ass:L_Phi-convex-reg-est}. 	
Using that $\Phi_{1,n,m}$ is convex, Item Reg(${{\Phi_2}}$) from Assumption \nameref{ass:L_Phi-convex-reg-est} and Inequality \eqref{eq:regularity-estimate-one-two}, we get
\begin{equation}\label{eq:ineq:second}
\int_U \tr[(CD^2f)^2]+\norm{Q_1^{-\frac{1}{2}}CDf}^2_U\,\mathrm{d}\mu_1^{\Phi_{n,m}}\leq \left(4+\frac{c_{\Phi_2}}{4}\right) \int_U g_{n,m}^2\,\mathrm{d}\mu_1^{\Phi_{n,m}}.
\end{equation}
By \Cref{rem:weak_conv_measure}, the left-hand side of the inequality above converges to 
\[
\int_U \tr[(CD^2f)^2]+\norm{Q_1^{-\frac{1}{2}}CDf}^2_U\,\mathrm{d}\mu_1^{\Phi}.
\]
Now we observe

\begin{align*}\label{eqg_N_M}
&\quad\abs[\big]{\mu_1^{\Phi_{n,m}}\left(g_{n,m}^2\right) -\mu_1^{\Phi}\left(g^2\right)}	\\
&\leq	\abs[\big]{\mu_1^{\Phi_{n,m}}\left((g_{n,m}-g)^2\right)}+2	\abs[\big]{\mu_1^{\Phi_{n,m}}\left((g_{n,m}-g)g\right)} +	\abs[\big]{\mu_1^{\Phi_{n,m}}\left(g^2\right)-\mu_1^{\Phi}\left(g^2\right)}\\
&\leq \abs[\big]{\mu_1^{\Phi_{n,m}}\left((g_{n,m}-g)^2\right)}+2	\abs[\big]{\mu_1^{\Phi_{n,m}}\left((g_{n,m}-g)^2\right)\mu_1^{\Phi_{n,m}}\left(g^2\right)}^{\frac12} \\& \quad+	\abs[\big]{\mu_1^{\Phi_{n,m}}\left(g^2\right)-\mu_1^{\Phi}\left(g^2\right)}\\
&\leq\mu_1(e^{-\Phi_{n,m}})^{-1}e^{-inf_{\tilde{u}\in U}\Phi(\tilde{u})}\norm{g_{n,m}-g}_{L^2(\mu_1)}^2\\
&\quad+2\left(\mu_1(e^{-\Phi_{n,m}})^{-1}e^{-inf_{\tilde{u}\in U}\Phi(\tilde{u})}\mu_1^{\Phi_{n,m}}\left(g^2\right)\right)^{\frac12}\norm{g_{n,m}-g}_{L^2(\mu_1)}\\
&\quad+\abs[\big]{(\mu_1^{\Phi_{n,m}}(g^2)-\mu_1^{\Phi}(g^2)}.
\end{align*}
By \Cref{rem:weak_conv_measure}, we follow that
\[
\lim_{n\rightarrow \infty}\lim_{m\rightarrow \infty}{\mu_1(e^{-\Phi_{n,m}})^{-1}}={\mu_1(e^{-\Phi})^{-1}}\quad\text{and}\quad\lim_{n\rightarrow \infty}\lim_{m\rightarrow \infty}\abs[\big]{(\mu_1^{\Phi_{n,m}}(g^2)-\mu_1^{\Phi}(g^2)}=0.
\]In view of Item Reg(${{\Phi_1}}$) from Assumption \nameref{ass:L_Phi-convex-reg-est}, we know that 
\begin{align*}
&\quad \limsup_{n\rightarrow \infty}\limsup_{m\rightarrow \infty}\norm{g_{n,m}-g}_{L^2(\mu_1)}\\
&=\limsup_{n\rightarrow \infty}\limsup_{m\rightarrow \infty}\norm{(Q_1^{\theta}D\Phi_{1,n,m}-Q_1^{\theta}D\Phi_{1},Q^{-\theta}CDf)_U}_{L^2(\mu_1)}=0,
\end{align*}
using that $Q_1^{-\theta}CDf=\sum_{i=1}^N(Q_1^{-\theta}CDf,d_i)_Ud_i\in U_N$ for some $N\in\N$ and the boundedness of $(Q_1^{-\theta}CDf,d_i)_U$. We conclude that
\[\limsup_{n\rightarrow \infty}\limsup_{m\rightarrow \infty}	\abs[\big]{\mu_1^{\Phi_{n,m}}\left(g_{n,m}^2\right) -\mu_1^{\Phi}\left(g^2\right)}=0.\]
Hence, a successive application of the limes superior, first for $m$ and then for $n$, in Inequality \eqref{eq:ineq:second} finishes the proof.
\end{proof}

We continue by restricting the setting to the Hilbert space \[H\defeq \left\{f\in L^2(W;\mu^{\Phi})\mid \mu^{\Phi}(f)=0  \right\}\] and operator domain $\acore\defeq\fbs{B_W}\cap H$. As the essential m-dissipativity of $L^{\Phi}$ holds on $\fbs{B_W}\subseteq L^2(W;\mu^{\Phi})$, we first need to justify the corresponding result in the restricted setting.


\begin{rem}\label{rem:welldefine_on_acore}
In \Cref{cor:invariance_ASP}, we established that the measure $\mu^{\Phi}$ is invariant for the strongly continuous contraction semigroup $\sccs$ generated by $(L^{\Phi},D(L^{\Phi}))$. The measure $\mu^{\Phi}$ is also invariant for the operators $(A^{\Phi},\fbs{B_W})$, $(S,\fbs{B_W})$ $(L^{\Phi},\fbs{B_W})$. Hence, $T_t(H)$, $A^{\Phi}(\acore), S^{\Phi}(\acore)$ and $L^{\Phi}(\acore)$ are contained in $H$ and it is therefore possible to restrict $\sccs$ to a s.c.c.s.~on $H$, which we denote in the following by $\sccsz$. Moreover, we consider $(A^{\Phi},\acore)$, $(S,\acore)$ and $(L^{\Phi},\acore)$ as operators on $H$.
\end{rem}

\begin{prop}{\cite[Proposition 7.1]{EG23}}\label{prop:inf-dim-sub-markov}
$\acore$ is dense in $H$ and the operator $(L^{\Phi},\acore)$ is essentially m-dissipative on $H$. Therefore, its closure, denoted by $(L^{\Phi}_0,D(L^{\Phi}_0))$, generates a s.c.c.s., which is equal to $\sccsz$.
\end{prop}
%

\begin{defn}\label{def:ortho_proj}
Let $H=H_1\oplus H_2$, where $H_1$ is provided by the orthogonal projection
\[
P:H\to H_1, \quad
f\mapsto Pf\defeq \int_{V}f(\cdot,v)\,\mu_2(\mathrm{d}v).
\]
For any $f\in \acore$, we can interpret $Pf$ as an element of $\fbs{B_U}$, which is then denoted by $f_P$.
Further, let $(S_0,D(S_0))$ and $(A^{\Phi}_0,D(A^{\Phi}_0))$ be the closures in $H$ of $(S,\acore)$ and $(A^{\Phi},\acore)$, respectively.
\end{defn}

To formulate the next assumptions it is useful to introduce $V_{\infty}\defeq \lin{e_1,e_2,\dots}$
\begin{cond}{\textbf{K2}}\label{ass:inf-dim-matrix-bounded-ev}
Assume that $K_{22}(v)=K_{1} + K_{2}(v)$, where
$K_{1}\in \lop{V}$ and $K_{2}:V\to \lop{V}$. In addition, assume that $K_1$ and $K_2$ share the same invariance properties as $K_{22}$.
Further, let the following hold
\begin{enumerate}[(i)]
\item There is some $C_{1}\in(0,\infty)$ such that
\[\norm{Q_2^{-\frac12}K_{1}^*Q_2^{-1}K_{1} Q_2^{-\frac12}}_{\mathcal{L}(V_{\infty})} \leq C_1.\]
\item 
There exists a measurable function $\overline{C}_2:V\to [0,\infty)$ such that for $\mu_2$-a.e.~$v\in V$
\begin{align*}
	&\norm{(Q_2^{-\frac12}K_2(v)^*Q_2^{-2}K_2(v)Q_2^{-\frac12})}_{\mathcal{L}(V_{\infty})}\leq \overline{C}_2(v)\quad \text{and}\\&C_2\defeq \int_V \overline{C}_2(v)\norm{v}_V^2\dm_2<\infty.
\end{align*}
\item 
For all $v\in V$, the sequence $(\alpha_n^{22}(v))_{n\in\N}$ defined by
\[\alpha_n^{22}(v)\defeq \sum_{k=1}^{\infty}(Q_2^{-\frac{1}{2}}\party{k}K_{22}(v)e_k,e_n)_V\] is 			in $\ell^2(\N)$
and \[M_{22}\defeq \int_V\| ({\alpha_n^{22}}(v))_{n\in\N}\|^2_{\ell^2}\,\mu_2(\mathrm{d}v)<\infty.\]
\end{enumerate}
\end{cond}

\begin{rem}\label{remark:alter_K_seven}
Assumption \nameref{ass:inf-dim-matrix-bounded-ev} is designed to use $(CD_1f_P,D_1f_P)_U$ as a bounding term, in the proof of \Cref{thm:inf-dim-hypoc-applied}.
By involving $Q_1$ and therefore modify assumption \nameref{ass:inf-dim-matrix-bounded-ev},
we can use $(Q_1^{-1}CDf_P,CDf_P)_U$ instead as a bound. This modification is formulated in the assumption below.
\end{rem}
\begin{cond}{\textbf{K2}*}\label{ass:inf-dim-matrix-bounded-ev_star}
Assume that $K_{22}(v)=K_{1} + K_{2}$, where
$K_{1}\in \lop{V}$ and $K_{2}:V\to \lop{V}$. Moreover, assume that $K_1$ and $K_2$ share the same invariance properties as $K_{22}$.
Further, let the following hold
\begin{enumerate}[(i)]
	\item There is some $C_{1}\in(0,\infty)$ such that
	\[\norm{Q_1^{\frac{1}{2}}K_{21}^{-1}K_1^*Q_2^{-1}K_1K_{12}^{-1} Q_1^{\frac{1}{2}}}_{\mathcal{L}(V_{\infty})}\leq C_1.\]
	\item There exists a measurable function $\overline{C}_2:V\rightarrow [0,\infty)$ such that for all $k\in\N$ and $\mu_2$-a.e.~$v\in V$
	\[
	\begin{aligned}
		\norm{ Q_2^{-1}K_2(v)^*K_{21}^{-1}Q_1K_{12}^{-1} K_2(v)Q_2^{-1}}_{\mathcal{L}(V_{\infty})}&\leq \overline{C}_2(v)\quad\text{and}\\
		C_2\defeq \int_V \overline{C}_2(v)\norm{v}_V^2\dm_2&<\infty.
	\end{aligned}
	\]
	\item 
	Assume that the sequence $\s{\alpha^{22}(v)}$ defined by
	\[\alpha_n^{22}(v)\defeq \sum_{k=1}^{\infty}(Q_1^{\frac{1}{2}}K_{12}^{-1}\party{k}K_{22}(v)e_k,d_n)_V\] is 				an element of $\ell^2(\N)$
	and that $M_{22}\defeq \int_V\| ({\alpha_n^{22}}(v))_{n\in\N}\|^2_{\ell^2}\,\mu_2(\mathrm{d}v)<\infty$.
\end{enumerate}
\end{cond}

To verify the microscopic coercivity assumptions formulated in \cite{GS14} and \cite{GS16} we introduce the following  Poincar\'{e} type assumption.

\begin{cond}{\textbf{K3}}\label{ass:inf-dim-v-poincare}
Assume that there is some $c_S\in(0,\infty)$ such that
\[
\int_V (K_{22}D_2f,D_2f)_V\,\mathrm{d}\mu_2 \geq c_S \int_V \left(f-\mu_2(f)\right)^2\,\mathrm{d}\mu_2\quad \text{for all}\quad f\in\fbs{B_V}.
\]

\end{cond}
\begin{cond}{\textbf{K4}}\label{ass:inf-dim-u-poincare}
Assume that there is some $c_A\in(0,\infty)$ such that
\[
\int_U (CD_1f,D_1f)_V\,\mathrm{d}\mu^{\Phi}_1 \geq c_A \int_U \left(f-\mu^{\Phi}_1(f)\right)^2\,\mathrm{d}\mu^{\Phi}_1\quad \text{for all}\quad f\in\fbs{B_U}.
\]
\end{cond}

\begin{rem}\label{rem:inf-dim-eigenvalue-estimates}
\begin{enumerate}[(i)]
	\item Recall $K_{22}^0$ from Assumption \nameref{ass:inf-dim-elliptic} and assume that there is a sequence of eigenvalues $(\lambda_k^0)_{k\in\N}$ of $K_{22}^0$ with respect to the basis $B_V$. Let $\lambda_{2,i}$ denote the $i$-th eigenvalue of $Q_2$,
	then, due to \Cref{lem:inf-dim-poincare}, we have for all $f\in\fbs{B_V}$
	\[
	\int_V (Q_2D_2f,D_2f)_V\,\mathrm{d}\mu_2 \geq \frac{1}{\lambda_{2,1}}\int_V \left(f-\mu_2(f)\right)^2\,\mathrm{d}\mu_2.
	\]
	So, if there is some $\omega_{22}\in(0,\infty)$ such that $\lambda_k^0\geq \omega_{22}\lambda_{2,k}$ 				for each $k\in\N$, then Assumption \nameref{ass:inf-dim-v-poincare} holds with
	$c_S=\frac{\omega_{22}}{\lambda_{2,1}}$.
	\item Similarly, if $\Phi=\Phi_1+\Phi_2$ is as described in Assumption \nameref{ass:L_Phi-convex-reg-est}   and there is some $				\omega_{12}\in(0,\infty)$ such that $(Cd_k,d_k)\geq \omega_{12}\lambda_{1,k}$
	for all $k\in\N$, then Assumption \nameref{ass:inf-dim-u-poincare} holds with $c_A= \frac{\omega_{12}}{\lambda_{1,1}e^{\norm{\Phi_2}_{osc}}}$. Indeed this follows by approximating $\Phi_1$ with the sequence $(\Phi_{1,n,m})_{n,m\in\N}$ provided by \nameref{ass:L_Phi-convex-reg-est}, then applying \Cref{lem:inf-dim-poincare} and finally using the approximation properties of $(\Phi_{1,n,m})_{n,m\in\N}$ which are due to Item (i) and (ii) from Item Reg($\Phi_1$).
\end{enumerate}
\end{rem}

The main result of this section follows by arguing as in \cite{EG23}.
\begin{thm}\label{thm:inf-dim-hypoc-applied}
Assume that  the assumptions \nameref{ass:qc-negative-type}-\nameref{ass:inf-dim-u-poincare}, 
(with either \nameref{ass:inf-dim-matrix-bounded-ev} or \nameref{ass:inf-dim-matrix-bounded-ev_star}), \nameref{ass:ess-N_Phi}, 	\nameref{ass:L_Phi-convex-reg-est} and \nameref{ass:pot_app} hold true.
Then, the semigroup $\sccs$ on $L^2(W;\mu^{\Phi})$ generated by the closure $(L^{\Phi},D(L^{\Phi}))$ of $(L^{\Phi},\fbs{B_W})$,
is hypocoercive in the sense that for each $\theta_1\in (1,\infty)$,
there is some $\theta_2\in (0,\infty)$ such that
\[
\left\| T_tf- \mu^{\Phi}(f) \right\|_{L^2(\mu^{\Phi})} \leq \theta_1 \mathrm{e}^{-\theta_2 t}\left\| f-\mu^{\Phi}(f)  \right\|_{L^2(\mu^{\Phi})}
\]
for all $f\in {L^2(W;\mu^{\Phi})}$ and all $t\geq 0$.
For $\theta_1\in (1,\infty)$, the constant $\theta_2$ determining the speed of convergence can be explicitly computed in terms of 	$c_S$, $c_A$, $c_{\Phi_2}$ and $
c_1\defeq \frac{1}{2}\big(\sqrt{C_{1}}+\sqrt{C_2}+\sqrt{M_{22}}\big)
$ as
\[
\theta_2=\frac{1}{2}\frac{\theta_1-1}{\theta_1}\frac{\min\{c_S,c_1\}}{\Big(1+c_1+\sqrt{8+\frac{c_{\Phi_2}}{4}}\Big)\Big(1+\frac{1+c_A}				{2c_A}(1+c_1+\sqrt{8+\frac{c_{\Phi_2}}{4}})\Big)+ \frac{1}{2}\frac{c_A}{1+c_A}}\frac{c_A}{1+c_A}.
\]	
\end{thm}
\begin{proof}
First of all note that assumptions \nameref{ass:pot_app} guarantees by  \Cref{theo:ess_diss_phi_Lfour} that the operator $(L^{\Phi},\fbs{B_W})$ is essentially m-dissipative with a nice core. 
In view of the abstract hypocoercivity method formulates in \cite{GS14} and \cite{GS16} and its refined formulation in \cite{EG23} we first have to check the data conditions. This are indeed valid by the same reasoning as in \cite[Proposition 7.8]{EG23}. In order to conclude the statement, we verify the three hypocoercivity conditions, namely the boundedness of the auxiliary operators, the microscopic coercivity and the macroscopic coercivity. The boundedness of the axillary operator holds true with $c_1$
by either invoking \nameref{ass:inf-dim-matrix-bounded-ev} or \nameref{ass:inf-dim-matrix-bounded-ev_star} by the same reasoning as in \cite[Proposition 7.11]{EG23} and with $c_2=\sqrt{8+\frac{c_{\Phi_2}}{4}}$ by \nameref{ass:ess-N_Phi} and \nameref{ass:L_Phi-convex-reg-est} together with \cite[Proposition 7.9]{EG23}.
By assumption \nameref{ass:inf-dim-v-poincare} the microscopic coercivity holds true with $\Lambda_m=c_S$ and by  \nameref{ass:inf-dim-u-poincare} the macroscopic coercivity with $ \Lambda_M=c_A$. Recall that \nameref{ass:ess-N_Phi} is also used to show the microscopic coercivity as explained in \cite[Remark 6.4]{EG23}. Now the claim follows by \cite[Theorem 7.16]{EG23}.
\end{proof}

\section{Degenerate stochastic reaction-diffusion equations}\label{ch:example_Lfour_ess_diss}
In this section, we analyze degenerate second order in time stochastic reaction-diffusion equations with multiplicative noise, whereby different to \cite{EG23} the gradient of the potential might be unbounded.
The general setting is describes as follows.
Let $\mathrm{d}\xi$ be the standard Lebesgue measure on $((0,1),\borel(0,1))$ and define $U\defeq L^2((0,1);\mathrm{d}\xi)$. Moreover, we denote by $W^{1,2}_0(0,1)$ the classical Sobolev space of weakly differentiable functions with zero boundary conditions on $(0,1)$ and by $W^{2,2}(0,1)$ the Sobolev space of two times weakly differentiable functions on $(0,1)$. In the following, we set $W\defeq U\times U$ and let
$(-{\partial_{\xi}^2}, D({\partial_{\xi}^2}))$ be the negative second order derivative with Dirichlet boundary conditions, i.e. 
\[
D({\partial_{\xi}^2})\defeq W^{1,2}_0(0,1)\cap W^{2,2}(0,1)\subseteq U.
\]
It is well known that the inverse of $(-{\partial_{\xi}^2}, D({\partial_{\xi}^2}))$ can be extended to a bounded linear operator on $U$. This extension is denoted by $(-{\partial_{\xi}^2})^{-1}$. Therefore, it is reasonable to define the linear continuous operator
\[
Q=(-{\partial_{\xi}^2})^{-1}:(U,\normx_{U}) \rightarrow (D({\partial_{\xi}^2}),\normx_{U})\subseteq (U,\normx_{U}).
\]
The operator $Q$ is positive and self-adjoint. Further, $B_U=(d_k)_{k\in\N}=(\sqrt{2}\sin(k\pi\cdot))_{k\in\N}$ is an orthonormal basis of $U$ diagonalizing $Q$ with corresponding eigenvalues $(\lambda_k)_{k\in\N}=({(k\pi)^{-2}})_{k\in\N}$.

For parameters $\alpha_1,\alpha_2\in \R$ with \[\alpha_1,\alpha_2>\frac{1}{2},\] we consider two centered non-degenerate infinite dimensional Gaussian measures $\mu_1$ and $\mu_2$ on $(U,\mathcal{B}(U))$, with covariance operators
\[
Q_1\defeq Q^{\alpha_1}\quad\text{and}\quad Q_2\defeq Q^{\alpha_2},
\]
respectively.

Since $(\lambda_k)_{k\in\N}\in \ell^{r}(\N)$ for $r>\frac{1}{2}$, $Q_1$ and $Q_2$ are indeed trace class.
By construction, $B_U$ is a basis of eigenvalues of $Q_1$ and $Q_2$ with corresponding eigenvalues given by
\[
\lambda_{1,k}\defeq \lambda_k^{\alpha_1}\quad\text{and}\quad\lambda_{2,k}\defeq \lambda_k^{\alpha_2},\quad k\in\N,
\]
respectively.

We have to note, that Item App($\Phi 3$) from Assumption \nameref{ass:pot_app}, which is needed to apply the central essential m-dissipativity result from \Cref{subsec:unb_grad}, is not shown in this section. Indeed, as explained in \Cref{subsec:unb_grad} it is considered as a conjecture, whose validity is reasonable by the strategy described in \Cref{lem:sol_finite}. To be consistent with this strategy, we derive stronger regularity results for the potential and coefficients than required in Assumption \nameref{ass:pot_app}.

\subsection{Essential m-dissipativity}

For $\sigma_1\in[0,\infty)$, we choose $K_{12}=Q^{\sigma_1}$ and since $K_{21}=K_{12}^*$, also $K_{21}=Q^{\sigma_1}$. Moreover, we assume that $K_{22}$ is diagonal with respect to $B_U$ and therefore determined by its eigenvalue functions $\lambda_{22,k}:U\to\R$, $k\in\N$. Let $\sigma_2,\sigma_3\in[0,\infty)$.
For each $k\in\N$, choose $\psi_k\in C_c^4(\R^k;[0,\infty))$ and define
\begin{equation}\label{eqdef:K_twotwo_var_eig}
\lambda_{22,k}(v)\defeq \lambda_{k}^{\alpha_2}+\lambda_{k}^{\sigma_2}+ \lambda_{k}^{\sigma_3}\frac{\psi_k(p_k v)}{\|\psi_k\|_{C^4}}.
\end{equation}


One can check that 
\[ 
\lambda_{k}^{\alpha_2},\lambda_{k}^{\sigma_2}\leq \lambda_{22,k}(v)=\lambda_{22,k}(P_k v)\leq \lambda_{k}^{\alpha_2}+ \lambda_{k}^{\sigma_2}+\lambda_{k}^{\sigma_3}
\] for all $k\in\N$ and $v\in U$.
For $i\geq k$ and all $v\in U$, we have $\partx{i}\lambda_{22,k}(v)=0$ and for $1\leq i\leq k$, it holds that

\[
\begin{aligned}
|\partx{i}\lambda_{22,k}(v)|=\left|\lambda_{k}^{\sigma_3} \frac{\partial_i\psi_k(p_k v)}{\|\psi_k\|_{C^4}}\right| 
\leq \lambda_{k}^{\sigma_3}.
\end{aligned}
\]
We simply set $K_{22}(v)d_k\defeq \lambda_{22,k}(v)d_k$, which describes a symmetric positive and bounded linear operator on $U$, as required for \Cref{def:inf-dim-operators}. Extending the arguments from above to higher order derivatives, we see that $v\mapsto K_{22}(v)d_k\in C^4_b(V;{V})$ for all $k\in\N$ and also Assumption \nameref{ass:inf-dim-elliptic} holds true. The compact support property of $\psi_k$ implies that \[v\mapsto K_{22}(P_n(v))Q_2^{-1}P_n(v)\]has bounded derivatives up to order three for all $n\in\N$, which is essential to use the arguments from \Cref{lem:sol_finite}. 	
Actually, to check Item {App}($\Phi 2$) from Assumption \nameref{ass:pot_app} it is enough to have $v\mapsto K_{22}(v)d_k\in C^2_b(V;{V})$ for all $k\in\N$, i.e.~it is enough to assume that $\psi_k\in C_b^2(\R^k,[0,\infty))$ for the definition of $\lambda_{22,k}$.

The class of potentials we consider below, is inspired by the considerations in \cite{dPL06}, where the m-dissipativity of Kolmogorov operators, corresponding to degenerate finite dimensional stochastic Hamiltonian systems with additive noise, were investigated.

\begin{defn}\label{rem:MVT_bound_der_phi}
Fix $\phi\in C^4(\R)$, which is bounded from below by zero. Assume that there are constants $A, \bar{B},R,m_0\in(0,\infty)$ and $m_1\in \N_{\geq 4}$ such that
\begin{equation}\label{eq:phi_bound_below}
	\phi(x)\geq A\abs{x}^{m_0}\quad \text{for all}\quad \abs{x}\geq R
\end{equation}

and 
\[
\abs{\phi^{(4)}(x)}\leq  \bar{B}(1+\abs{x}^{m_1-4})\quad \text{for all}\quad x\in\R.
\]
Using the mean value theorem, there is a constant ${B}\in (0,\infty)$ such that for all $x\in\R$ and $j\in \left\lbrace 0,1,2,3,4\right\rbrace $
\begin{equation}\label{eq:poly_bound_phim}
	\abs{\phi^{(j)}(x)}\leq B(1+\abs{x}^{m_1-j}).
\end{equation}

The potential $\Phi: L^2((0,1);\mathrm{d}\xi)\rightarrow \R $ is defined in terms of $\phi$ via

\[
\Phi(u)=\begin{cases}\int_0^1\phi(u(\xi))\,\mathrm{d}\xi,& u\in L^{m_1}((0,1);\mathrm{d}\xi) \\ \infty, & else \end{cases}.
\]

Let $q\in\N$ be even and $(\alpha_m)_{m\in\N}\subseteq (0,\infty)$ be a monotone sequence converging to zero. For $m\in\N$, we set
\[
\Psi_m\defeq \Psi_{m,q}:\R\to \R,\quad \Psi_m(x)\defeq\frac{x}{1+\alpha_m x^q}\quad \text{and}\quad \phi_m\defeq \Psi_m\circ\phi \in C^4(\R).\] 
\end{defn}

We start investigating $(\phi_m)_{m\in\N}$ by establishing that all derivatives up to order four are polynomial bounded independent of the index $m$. This helps to approximate $\Phi$, as required in Assumption \nameref{ass:pot_app}. 
\begin{lem}\label{lem:pot_approx_q}
There exists a constant $q\in\N$ only dependent on $\phi$ such that $\phi_m^{(j)}$ is bounded for all $j\in \left\lbrace 1,2,3,4\right\rbrace$ and there is a constant $\tilde{B}\in\N$ with
\begin{equation}\label{eq:poly_bound_phimj}
	\abs{\phi_m^{(j)}(x)}\leq \tilde{B}(1+\abs{x}^{j(m_1-1)})\quad\text{for all}\quad j\in \left\lbrace 1,2,3,4\right\rbrace\quad \text{and}\quad m\in\N.
\end{equation}

\end{lem}
\begin{proof}
We calculate for all $m\in\N$
\[
\begin{aligned}
	\phi_m'&=\Psi_m'(\phi)\phi',\quad	\phi_m''=\Psi_m''(\phi)(\phi')^2+\Psi_m'(\phi)\phi'',\\
	\phi_m'''&=\Psi_m'''(\phi)(\phi')^3+3\Psi_m''(\phi)\phi'\phi''+\Psi_m'(\phi)\phi''',\\
	\phi_m''''&=\Psi_m''''(\phi)(\phi')^4+6\Psi_m'''(\phi)(\phi')^2\phi'' \\&\quad+3\Psi_m''(\phi)(\phi'')^2+4\Psi_m''(\phi)\phi'\phi'''+\Psi_m'(\phi)\phi'''',
\end{aligned}
\]
and for all $x\in\R$
\[
\begin{aligned}
	\Psi_m'(x)&=\frac{1-\alpha_m(q-1)x^q}{(1+\alpha_mx^q)^2},\quad \Psi_m''(x)=\frac{\alpha_mqx^{q-1}\left(-\alpha_mx^q+q(\alpha_mx^q-1)-1\right)}{(1+\alpha_mx^q)^3}\\
	\Psi_m'''(x)&=-\frac{ x^{q-2}(q (\alpha_m x^q + 1)^2 - q^3 (\alpha_m x^q (\alpha_m x^q - 4) + 1))}{(\alpha_m x^q + 1)^4},\\
	\Psi_m''''(x)&=\alpha_m q x^{q - 3}\frac{ \alpha_m^3 (q - 1) (q + 1) (q + 2) x^{3q} - \alpha_m^2 (11 q^3 + 6 q^2 + q + 6) x^{2q} }{(\alpha_m x^q + 1)^5},\\
	&\qquad +\alpha_m q x^{q - 3}\frac{ \alpha_m (q - 1) (q (11 q + 5) + 6) x^q - q^3 + 2 q^2 + q - 2}{(\alpha_m x^q + 1)^5}.
\end{aligned}
\]
Recall that we assume that $q\in\N$ is even. We show the claim exemplary for $j=1$. The other cases follow similarly, using the calculations from above.
First, we verify that $\Psi_m'$ is bounded independent of $m$. This follows by

\[
\left| {\frac{1-\alpha_m(q-1)x^q}{(1+\alpha_mx^q)^2}}\right| \leq \frac{1}{1+\alpha_mx^q}+q\frac{\alpha_m x^q}{(1+\alpha_mx^q)^2}\leq (q+1)\frac{1}{1+\alpha_mx^q}\leq (q+1).
\]
Using Inequality \eqref{eq:poly_bound_phim} and the estimate right above, we obtain Inequality \eqref{eq:poly_bound_phimj} for $j=1$.\\
To show that $\phi_m'$ is bounded, we proceed as follows. Let $x\in \R$, then, by means of Inequality \eqref{eq:phi_bound_below}, we can estimate
\[
\abs{\phi'_m(x)}\leq (q+1)\left| \frac{\phi'(x)}{1+\alpha_m\phi(x)^q}\right| \leq (q+1) \begin{cases}
	\frac{\bar{B}(1+\abs{x}^{m_1-1})}{1+A^q\alpha_m\abs{x}^{qm_0}}&\text{for}\quad \abs{x}>R\\
	\sup_{\abs{x}\leq R}\abs{\phi'(x})&\text{for}\quad\abs{x}\leq R.
\end{cases}
\]
Therefore, boundedness of $\phi'_m$ follows for $q>\frac{m_1-1}{m_0}$. 
\end{proof}
For the rest of this section, we assume that $q\in \N$ is as in \Cref{lem:pot_approx_q}.

\begin{defn}\label{def:pot_fin_n}
For $n,m\in\N$, we define $\Phi_n:U\to \R$ and $\Phi_n^m:U\to \R$ by
\[
\begin{aligned}
	\Phi_n(u)\defeq\int_0^1\phi(P_nu(\xi))\,\mathrm{d}\xi\quad\text{and}\quad \Phi_n^m(u)\defeq \int_0^1\phi_m(P_nu(\xi))\,\mathrm{d}\xi.
\end{aligned}
\]
It is evident that $(\Phi_n^m)_{n,m\in\N}$ fulfills {App}($\Phi 1$) from Assumption \nameref{ass:pot_app}.
\end{defn}
\begin{lem}\label{lem:meas_equi}
For all $r\geq 1$, it holds 
\begin{align}\label{ineq:inf_e_hoch_m_n}
	\nonumber&\lim_{n\to\infty}\Phi_n=\Phi\quad\text{in}\quad L^r(U;\mu_1),\; \lim_{n\to\infty}\mu_1(e^{-\Phi_n})= \mu_1(e^{-\Phi})\; \text{and}\\ \nonumber&0<\inf_{n,m\in\N}\mu_1(e^{-\Phi_n^m})\leq 1.
\end{align}Moreover, the measures $\mu_1^{\Phi_n^m}$ is uniformly dominated by $\mu_1$, i.e.~for all non-negative measurable functions $f$ and $n,m\in\N$ it holds
\[
\int_U f\,\mathrm{d}\mu_1^{\Phi_n^m}\leq \frac{1}{\inf_{n,m\in\N}\mu_1(e^{-\Phi_n^m})}\int_U f\,\mathrm{d}\mu_1.
\] 

\end{lem}
\begin{proof}
The first claim follows by \Cref{prop:pot_approx_wit_conv}. The second claim immediately follows by the first and the mean value theorem, since $\Phi,\Phi_n\geq 0$  for all $n\in\N$ and the derivative of $[0,\infty)\ni x\mapsto e^{-x}\in\R$ is bounded by $1$.

By definition, it holds $0\leq \phi_m\leq \phi$ and therefore $0\leq \Phi_n^m\leq \Phi_n$. Hence, $0\leq\mu_1(e^{-\Phi_n})\leq \mu_1(e^{-\Phi_n^m})\leq  1$. 

As $\lim_{n\to\infty}\mu_1(e^{-\Phi_n})= \mu_1(e^{-\Phi})>0$, we know that the sequence $(\mu_1(e^{-\Phi_n}))_{n\in\N}$ is bounded from below by a positive constant and therefore the third statement is shown. Finally the last one follows, noting that $e^{-\Phi^m_n}\leq 1$ for all $m,n\in\N$.
\end{proof}
\begin{lem}\label{lem:part_der_phi_m_n}
For all $m,n\in\N$, it holds $\Phi_n,\Phi_n^m\in C^4(U;\R)$ and for $i,j,k,l\in \lbrace1,...,n\rbrace$ we have
\[
\begin{aligned}
	\partx{i}\Phi_n(u)&=\int_0^1\phi'(P_nu)d_i\,\mathrm{d}\xi,\quad \partx{j}\partx{i}\Phi_n(u)=\int_0^1\phi''(P_nu)d_id_j\,\mathrm{d}\xi \\
	\partx{k}\partx{j}\partx{i}\Phi_n(u)&=\int_0^1\phi'''(P_nu)d_id_jd_k\,\mathrm{d}\xi, \\ \partx{l}\partx{k}\partx{j}\partx{i}\Phi_n(u)&=\int_0^1\phi''''(P_nu)d_id_jd_kd_l\,\mathrm{d}\xi,\\
	\text{and} \quad \partx{i}\Phi^m_n(u)&=\int_0^1\phi_m'(P_nu)d_i\,\mathrm{d}\xi,\quad 									\partx{j}\partx{i}\Phi^m_n(u)=\int_0^1\phi_m''(P_nu)d_id_j\,\mathrm{d}\xi \\
	\partx{k}\partx{j}\partx{i}\Phi^m_n(u)&=\int_0^1\phi_m'''(P_nu)d_id_jd_k\,\mathrm{d}\xi,\\ \partx{l}\partx{k}\partx{j}\partx{i}\Phi^m_n(u)&=\int_0^1\phi_m''''(P_nu)d_id_jd_kd_l\,\mathrm{d}\xi.
\end{aligned}
\]The partial derivatives evaluate to zero if one of the indices exceeds $n$.
Furthermore, we have $D\Phi_n^m\in C_b^3(U;\R)$ and consequently we know that Item {App}($\Phi 2$) from Assumption \nameref{ass:pot_app} is valid.
\end{lem}
\begin{proof}
The calculation of the partial derivatives follows as in \Cref{prop:pot_approx_wit_conv}. The proof of \Cref{prop:pot_approx_wit_conv} also contains the arguments to show $\Phi_n,\Phi_n^m\in C_b^4(U;\R)$. Note that the main ingredients are Inequality \eqref{eq:poly_bound_phimj} and \cite[Lemma 5.1]{dPL14}.
\end{proof}
To verify Item {App}($\Phi 2$) from  Assumption \nameref{ass:pot_app} it is enough that $\Phi_n^m\in C^3(U;\R)$ such that $D\Phi_n^m$ has bounded derivatives up to the second order is enough. However, the stronger regularity statement from \ref{lem:part_der_phi_m_n} shows that we are consistent with the strategy described in \Cref{lem:sol_finite}.

We are now able to verify that there are constants $\alpha,\beta,\gamma\in [0,\infty)$ such that Item {App}($\Phi 4$) from Assumption \nameref{ass:pot_app} is valid.

\begin{prop}\label{prop:ver_app_ass}Suppose that $\sigma_3\geq \frac{\min\lbrace\sigma_2,\alpha_2\rbrace}{2}$, $2\sigma_1-\min\lbrace\sigma_2,\alpha_2\rbrace\geq \frac{1}{2}$ and
\[
\begin{aligned}
	&\alpha >\frac{1}{2\alpha_1}+\frac{1}{2}\quad \text{and}\quad \alpha\geq  2\left(\frac{\frac{\min\lbrace\sigma_2,\alpha_2\rbrace}{2}-\sigma_1}{\alpha_1}+1\right)\\
	&\beta >\frac{1}{2\alpha_2}+\frac{\min\lbrace\sigma_2,\alpha_2\rbrace}{\alpha_2} \quad \text{and}\quad \beta\geq \max\left\lbrace  2\left(1-\frac{\min\lbrace\sigma_2,\alpha_2\rbrace}{2\alpha_2}\right),2\left(1-\frac{\sigma_3}{2\alpha_2}\right) \right\rbrace\\
	& \gamma > \frac{1}{4\alpha_2}+\frac{1}{2}\quad \text{and}\quad(4\gamma-\beta)>\frac{1}{2\alpha_2}+\frac{\min\lbrace\sigma_2,\alpha_2\rbrace}{\alpha_2}.
\end{aligned}
\]
Then, $\Phi_n^m$ fulfills {App}($\Phi 4$) from Assumption \nameref{ass:pot_app}.
\end{prop}
\begin{proof}
Before we start verifying the inequalities from {App}($\Phi 4$), we derive an useful integral estimate for products of the partial derivatives of $\Phi_n^m$. So let $i,j,k\in\N$ be given and recall the constants $\tilde{B}$ and $m_1$ from Inequality \eqref{eq:poly_bound_phimj}. Using \Cref{lem:part_der_phi_m_n} and Inequality \eqref{eq:poly_bound_phimj}, we estimate
\begin{align*}
	&\quad	\int_U\abs[\big]{\partx{j}\partxsq{i}\Phi^m_n   \partx{j}\partxsq{k}\Phi^m_n}\;\mathrm{d}\mu_1^{\Phi_n^m}\\
	&\leq\left( \int_U \left(\partx{j}\partxsq{i}\Phi^m_n\right)^2 \;\mathrm{d}\mu_1^{\Phi_n^m}\right)^{\frac{1}{2}}\left(\int_U \left(\partx{j}\partxsq{k}\Phi^m_n\right)^2\;\mathrm{d}\mu_1^{\Phi_n^m}\right)^{\frac{1}{2}}\\
	&\leq  \int_U\left(  \sqrt{2}^3\tilde{B} \int_0^1\left(1+\abs{P_nu(\xi)}^{3(m_1-1)}\right) \,\mathrm{d}\xi \right)^2\;\mathrm{d}\mu_1^{\Phi_n^m}\\
	&\leq  8\tilde{B}^2 \int_U   \int_0^1\left(1+\abs{P_nu(\xi)}^{3(m_1-1)}\right)^2\,\mathrm{d}\xi \;\mathrm{d}\mu_1^{\Phi_n^m}.
\end{align*}In a similar way, one can show
\begin{align*}\int_U(\partx{i}\Phi_n^m)^4\;\mathrm{d}\mu_1^{\Phi_n^m}&\leq 16\tilde{B}^4 \int_U   \int_0^1\left(1+\abs{P_nu(\xi)}^{(m_1-1)}\right)^4\,\mathrm{d}\xi \;\mathrm{d}\mu_1^{\Phi_n^m},\\
	\int_U(\partx{j} \partx{i}\Phi_n^m)^2\;\mathrm{d}\mu_1^{\Phi_n^m}&\leq 4\tilde{B}^2 \int_U   \int_0^1\left(1+\abs{P_nu(\xi)}^{2(m_1-1)}\right)^2\,\mathrm{d}\xi \;\mathrm{d}\mu_1^{\Phi_n^m}\quad\text{and} \\
	\int_U \left(\partx{j}\partx{i}\Phi_n^m\right)^4\,\mathrm{d}\mu_1^{\Phi_n^m}&\leq 16\tilde{B}^4\int_U \int_0^1\left(1+\abs{P_nu(\xi)}^{2(m_1-1)}\right)^4\,\mathrm{d}\xi \;\mathrm{d}\mu_1^{\Phi_n^m}.
\end{align*}
Using the generalized Hölder inequality ($\frac{1}{4}+\frac{1}{4}+\frac{1}{4}+\frac{1}{4}=1$) and the estimates above, we estimate
\begin{align*}
	&\quad	\int_U \abs[\big]{\partx{j}\partx{i}\Phi_n^m\,\partx{i}\Phi_n^m \partx{j}\partx{k}\Phi_n^m\, \partx{k}\Phi_n^m}\;\mathrm{d}\mu_1^{\Phi_n^m}\\
	&\leq 16\tilde{B}^4\left(\int_U \int_0^1\left(1+\abs{P_nu(\xi)}^{2(m_1-1)}\right)^4\,\mathrm{d}\xi \;\mathrm{d}\mu_1^{\Phi_n^m}\right)^{\frac{1}{2}}\\
	&\quad \times 
	\left(\int_U \int_0^1\left(1+\abs{P_nu(\xi)}^{(m_1-1)}\right)^4\,\mathrm{d}\xi \;\mathrm{d}\mu_1^{\Phi_n^m}\right)^{\frac{1}{2}}.
\end{align*}
Combing the estimates we just derived, the measure dominance from \Cref{lem:meas_equi} and the results from \cite[Lemma 5.1]{dPL14}, we know that there is a constant $C\in (0,\infty)$, independent of $i,j,k$, such that for all $m,n\in\N$
\begin{equation}
	\begin{aligned}\label{ineq:one_plus_pn}
		\int_U  &\abs[\big]{\partx{j}\partxsq{i}\Phi^m_n   \partx{j}\partxsq{k}\Phi^m_n} +(\partx{j} \partx{i}\Phi_n^m)^2\\ \qquad &+  \left(\partx{j}\partx{i}\Phi_n^m\right)^4+ \abs[\big]{\partx{j}\partx{i}\Phi_n^m\,\partx{i}\Phi_n^m \partx{j}\partx{k}\Phi_n^m\, \partx{k}\Phi_n^m}\;\mathrm{d}\mu_1^{\Phi_n^m}\leq C.
	\end{aligned}
\end{equation}

Recalling Inequality \eqref{eqdef:K_twotwo_var_eig}, we are able to estimate, by means of Inequality \eqref{ineq:one_plus_pn},
\[
\begin{aligned}
	&\quad\int_W\norm[\Big]{\sum_{i=1}^{\infty}\lambda_{1,i}^{\alpha}K_{22}^{-\frac{1}{2}}K_{12}D\partxsq{i}\Phi^m_n}_U^2\;\mathrm{d}\mu^{\Phi_n^m}\\
	&=\sum_{j=1}^{\infty}\int_W\left(\sum_{i=1}^{\infty}\lambda_i^{\alpha\alpha_1}\lambda_{22,j}^{-\frac{1}{2}}\lambda_j^{\sigma_1}\left(D\partxsq{i}\Phi^m_n,d_j\right)_V\right)^2\;\mathrm{d}\mu^{\Phi_n^m}\\
	&=\sum_{j=1}^{\infty}\lambda_j^{2\sigma_1}\int_V\lambda_{22,j}^{-1}\;\mathrm{d}\mu_2\sum_{i,k=1}^{\infty} \lambda_i^{\alpha\alpha_1}\lambda_k^{\alpha\alpha_1}\int_U\partx{j}\partxsq{i}\Phi^m_n   \partx{j}\partxsq{k}\Phi^m_n\;\mathrm{d}\mu_1^{\Phi_n^m}\\
	&\leq C\sum_{j=1}^{\infty}\lambda_j^{2\sigma_1}\int_V\lambda_{22,j}^{-1}\;\mathrm{d}\mu_2\left(\sum_{i=1}^{\infty} \lambda_i^{\alpha\alpha_1}\right)^2\\
	&\leq {C}\left(\sum_{i=1}^{\infty} \lambda_i^{\alpha\alpha_1}\right)^2\sum_{j=1}^{\infty} \lambda_j^{2\sigma_1-\min\lbrace\sigma_2,\alpha_2\rbrace}\eqdef \kappa_1.
\end{aligned}
\] 
Note that $\kappa_1<\infty$, as we assume $2\sigma_1-\min\lbrace\sigma_2,\alpha_2\rbrace>\frac12$ and $\alpha\alpha_1>\frac12$ is implied by $\alpha >\frac{1}{2\alpha_1}+\frac{1}{2}$.
Similar arguments yield
\[
\begin{aligned}
	&\quad	\sum_{i=1}^{\infty}\int_W\norm[\big]{\lambda_{1,i}^{\frac{\alpha}{2}}K_{22}^{-\frac{1}{2}} K_{12}D \partx{i}\Phi_n^m}_V^2\;\mathrm{d}\mu^{\Phi_n^m}\\&= \sum_{i,j=1}^{\infty}\lambda_j^{2\sigma_1}\lambda_{1,i}^{{\alpha}}\int_V \lambda_{22,j}^{-1}\;\mathrm{d}\mu_2\int_U(\partx{j} \partx{i}\Phi_n^m)^2\;\mathrm{d}\mu_1^{\Phi_n^m}\\
	&\leq C\sum_{i=1}^{\infty}\lambda_{i}^{{\alpha}\alpha_1}\sum_{j=1}^{\infty}\lambda_j^{2\sigma_1-\min\lbrace\sigma_2,\alpha_2\rbrace}\eqdef\kappa_2<\infty.
\end{aligned}
\]
Furthermore, we can derive, using the generalized Hölder inequality and \Cref{lem:Gaussianmoments}
\[
\begin{aligned}
	&\qquad\int_W  \norm[\Big]{\sum_{i=1}^{\infty} K_{22}^{-\frac{1}{2}} K_{12}D\partx{i}\Phi_n^m\, (u,Q_1^{\alpha-1}d_i)_U}_V^2 \;\mathrm{d}\mu^{\Phi_n^m}\\
	&\leq  \sum_{j=1}^{\infty} \lambda_j^{2\sigma_1}\int_V\lambda_{22,j}^{-1}\;\mathrm{d}\mu_2\sum_{i,k=1}^{\infty}\lambda_i^{\alpha_1(\alpha-1)}\lambda_k^{\alpha_1(\alpha-1)}\\&\quad\int_U \partx{j}\partx{i}\Phi_n^m\, (u,d_i)_U\partx{j}\partx{k}\Phi_n^m\, (u,d_k)_U\mathrm{d}\mu_1^{\Phi_n^m}\\
	&\leq  C\sum_{j=1}^{\infty} \lambda_j^{2\sigma_1-\min\lbrace\sigma_2,\alpha_2\rbrace}\sum_{i,k=1}^{\infty}\lambda_i^{\alpha_1(\alpha-1)}\lambda_k^{\alpha_1(\alpha-1)}\\&\quad\left(\int_U(u,d_i)_U^4\,\mathrm{d}\mu_1^{\Phi_n^m}\right)^{\frac{1}{4}}\left(\int_U(u,d_k)_U^4\,\mathrm{d}\mu_1^{\Phi_n^m}\right)^{\frac{1}{4}}\\
	&\leq \sqrt{3} \tilde{c}C\sum_{j=1}^{\infty} \lambda_j^{2\sigma_1-\min\lbrace\sigma_2,\alpha_2\rbrace}\left(\sum_{i=1}^{\infty}\lambda_i^{\alpha_1(\alpha-\frac{1}{2})}\right)^2\eqdef \kappa_3.
\end{aligned}
\]Above the constants $\tilde{c}$, independent of $m,n\in\N$, exists by uniform dominance of the measures $\mu_1^{\Phi_n^m}$ by $\mu_1$, which is due to \Cref{lem:meas_equi} and the results from \Cref{lem:Gaussianmoments}. Moreover, $\kappa_3<\infty$, since $\alpha >\frac{1}{2\alpha_1}+\frac{1}{2}$ is equivalent to $\alpha_1(\alpha-\frac{1}{2})>\frac12$.
To find $\kappa\in (1,\infty)$ such that the Inequalities \eqref{ineq:lfourone} and \eqref{ineq:lfourtwo} from {App}($\Phi 3$) are valid, we continue to estimate
\[
\begin{aligned}
	&\qquad\int_W  \norm[\Big]{\sum_{i=1}^{\infty}K_{22}^{-\frac{1}{2}} K_{12}D\partx{i}\Phi\,\lambda_{1,i}^{\alpha}\partx{i}\Phi}_V^2\;\mathrm{d}\mu^{\Phi_n^m}\\
	&\leq  \sum_{j=1}^{\infty} \lambda_j^{2\sigma_1}\int_V\lambda_{22,j}^{-1}\;\mathrm{d}\mu_2\sum_{i,k=1}^{\infty}\lambda_i^{\alpha_1\alpha}\lambda_k^{\alpha_1\alpha}\int_U \partx{j}\partx{i}\Phi_n^m\,\partx{i}\Phi_n^m \partx{j}\partx{k}\Phi_n^m\, \partx{k}\Phi_n^m\mathrm{d}\mu_1^{\Phi_n^m}\\
	&\leq C\sum_{j=1}^{\infty} \lambda_j^{2\sigma_1-\min\lbrace\sigma_2,\alpha_2\rbrace}\left(\sum_{i=1}^{\infty}\lambda_i^{\alpha_1\alpha}\right)^2=\kappa_1<\infty.
\end{aligned}
\]
Now let $f\in \fbs{B_U}$ be arbitrary.
To show Inequality \eqref{ineq:lfourthree} from Assumption \nameref{ass:pot_app}, let $k\in\N$ and assume $\alpha\geq 2\left(\frac{\frac{\min\lbrace\sigma_2,\alpha_2\rbrace}{2}-\sigma_1}{\alpha_1}+1\right)$ or equivalently $\alpha_1\left(\frac{\alpha}{2}-1\right)+\sigma_1\geq \frac{\min\lbrace\sigma_2,\alpha_2\rbrace}{2}$ to estimate for all $v\in V$
\[
\lambda_{k}^{\alpha_1(\frac{\alpha}{2}-1)+\sigma_1}\leq \lambda_k^{\frac{\min\lbrace\sigma_2,\alpha_2\rbrace}{2}}\leq \lambda_{22,k}^{\frac{1}{2}}(v).
\]
Using the fact that $(d_k)_{k\in\N}$ is an orthonormal basis, results in
\[
\norm[\big]{Q_1^{\frac{\alpha}{2}-1}K_{21}D_2f}_V\leq \norm[\big]{K_{22}^{\frac{1}{2}} D_2f}_V.
\]

Suppose $\beta\geq \max\left\lbrace  2(1-\frac{\min\lbrace\sigma_2,\alpha_2\rbrace}{2\alpha_2}),2(1-\frac{\sigma_3}{2\alpha_2}) \right\rbrace$, which is equivalent to $\frac{\beta}{2}\geq\frac12$, $ \sigma_2+\alpha_2(\frac{\beta}{2}-1)\geq \frac{\sigma_2}{2}$ and $\sigma_3+\alpha_2(\frac{\beta}{2}-1)\geq \frac{\sigma_3}{2}$. We obtain for all $v\in V$ 
\[
\begin{aligned}
	\lambda_k^{\alpha_2(\frac{\beta}{2}-1)}\lambda_{22,k}(v)&\leq  \lambda_{k}^{\alpha_2\frac{\beta}{2}}+\lambda_{k}^{\sigma_2+\alpha_2(\frac{\beta}{2}-1)}+\lambda_{k}^{\sigma_3+\alpha_2(\frac{\beta}{2}-1)}\frac{\psi_k(p_k v)}{\|\psi_k\|_{C^4}}\\
	&\leq \lambda_{k}^{\frac{\alpha_2}{2}}+\lambda_{k}^{\frac{\sigma_2}{2}}+\lambda_{k}^{\frac{\sigma_3}{2}}\frac{\psi_k(p_k v)}{\|\psi_k\|_{C^4}}\\ &\leq \sqrt{3}\lambda_{22,k}^{\frac{1}{2}}(v).
\end{aligned}
\]
Hence, as above,
\[
\norm[\big]{Q_2^{\frac{\beta}{2}-1}K_{22}D_2f}_V\leq\sqrt{3}\norm[\big]{K_{22}^{\frac{1}{2}} D_2f}_V.
\]
To establish Inequality \eqref{ineq:lfourfour}, we estimate for all $v\in V$
\[
\begin{aligned}
	\sum_{i=1}^{\infty}\norm[\big]{\lambda_{2,i}^{\frac{\beta}{2}}K_{22}^{-\frac{1}{2}}(v)\partx{i}K_{22}(v)d_k}_V^2&=\sum_{i=1}^{\infty}\lambda_{2,i}^{{\beta}}\lambda_{22,k}^{-{1}}(v)(\partx{i}\lambda_{22,k}(v))^2\\
	&\leq\sum_{i=1}^{\infty}\lambda_{i}^{{\alpha_2\beta}-\min\lbrace\sigma_2,\alpha_2\rbrace}\lambda_{k}^{2\sigma_3}\\
	&\leq \kappa_4\lambda_{k}^{2\sigma_3},
\end{aligned}
\]
where $ \kappa_4\defeq \sum_{i=1}^{\infty}\lambda_{i}^{{\alpha_2\beta}-\min\lbrace\sigma_2,\alpha_2\rbrace}<\infty$ as $\beta > \frac{1}{2\alpha_2}+ \frac{\min\lbrace\sigma_2,\alpha_2\rbrace}{\alpha_2}$ by assumption. Since we assume $\sigma_3\geq \frac{\min\lbrace\sigma_2,\alpha_2\rbrace}{2}$, we can derive
\[
\sum_{i=1}^{\infty}\norm[\big]{\lambda_{2,i}^{\frac{\beta}{2}}K_{22}^{-\frac{1}{2}}\partx{i}K_{22}D_2g}_V^2
\leq \kappa_4\norm[\big]{K_{22}^{\frac{1}{2}} D_2g}^2_V.
\]
Inequality \eqref{ineq:lfourfive} is valid for $\alpha_2(\frac{\beta}{2}-1)+\sigma_1\geq \alpha_1\frac{\alpha}{2}$.
In order to verify Inequality \eqref{ineq:lfoursix}, we use that $(4\gamma-\beta)>\frac{1}{2\alpha_2}+\frac{\min\lbrace\sigma_2,\alpha_2\rbrace}{\alpha_2}$ implies 
$
\sum_{i=1}^{\infty}\lambda_{i}^{\alpha_2(4\gamma-\beta)-\min\lbrace\sigma_2,\alpha_2\rbrace}<\infty
$ and the following estimate
\[
\begin{aligned}
	&\quad\left(\sum_{i=1}^{\infty}\left(\int_W\lambda_{2,i}^{4\gamma}(\partx{i}^2g)^2\;\mathrm{d}\mu^{\Phi_n^m}\right)^{\frac{1}{2}}\right)^2\\	&\leq\sum_{i=1}^{\infty}\lambda_{i}^{\alpha_2(4\gamma-\beta)-\min\lbrace\sigma_2,\alpha_2\rbrace}\sum_{i=1}^{\infty}\int_W\lambda_{2,i}^{\beta}\lambda_{i}^{\min\lbrace\sigma_2,\alpha_2\rbrace}(\partx{i}^2g)^2\;\mathrm{d}\mu^{\Phi_n^m}\\
	&\leq \sum_{i=1}^{\infty}\lambda_{i}^{\alpha_2(4\gamma-\beta)-\min\lbrace\sigma_2,\alpha_2\rbrace}\int_W \sum_{i=1}^{\infty}\norm[\big]{\lambda_{2,i}^{\frac{\beta}{2}}K_{22}^{\frac{1}{2}}D_2\partx{i}g}^2_V \;\mathrm{d}\mu^{\Phi_n^m}.
\end{aligned}
\]
Finally, if $\gamma>\frac{1}{4\alpha_2}+\frac{1}{2}$ or equivalently $\alpha_2(2\gamma-1)>\frac{1}{2}$, we obtain Inequality \eqref{ineq:lfourseven} as
\[
\begin{aligned}
	&\qquad	\sum_{i=1}^{\infty}  \lambda_{2,i}^{2\gamma-2}\left(\int_W(v,d_i)_V^4\;\mathrm{d}\mu^{\Phi_n^m}\right)^{\frac{1}{2}}\leq  \frac{\sum_{i=1}^{\infty}  \lambda_{i}^{\alpha_2(2\gamma-2)}}{\inf_{n,m\in\N}\mu_1(e^{-\Phi_n^m})} \left(\int_W(v,d_i)_V^4\;\mathrm{d}\mu\right)^{\frac{1}{2}}\\
	&=\frac{\sqrt{3}\sum_{i=1}^{\infty}  \lambda_{i}^{\alpha_2(2\gamma-1)}}{\inf_{n,m\in\N}\mu_1(e^{-\Phi_n^m})}\eqdef\kappa_5<\infty.
\end{aligned}
\]Hence, we choose $\kappa\in (1,\infty)$, in terms of $\kappa_1,\kappa_2,\kappa_3,\kappa_4,\kappa_5$ such that {App}($\Phi 4$) holds true.
\end{proof}

%
%

\begin{prop}\label{lem:appl_unb_pot_last}
Let $\sigma_1\geq \alpha_2\gamma$ and suppose that $\phi(x)=a_1{x}^2+\psi(x)$, $x\in\R,$ where $a_1\in [0,\infty)$ and $\psi\in C^4(\R;\R_{\geq 0})$ grows less than quadratic and its fourth order derivative is polynomial bounded, i.e.~$\phi$ is as demanded in \Cref{rem:MVT_bound_der_phi}. 
Then, $\Phi$ induced by $\phi$ fulfills Item {App}($\Phi 5$) from Assumption \nameref{ass:pot_app}.
\end{prop}
\begin{proof}

Let $p^*\in (4,\infty)$ and $q^*$ be as in {App}($\Phi 5$).
By means of \Cref{rem:pot_approx_alter}, we know that $D\Phi_n\to D\Phi$ as $n\to\infty$ in $L^{p^*}(U;\mu_1^{\Phi};U)$. Moreover, we have pointwisely $\lim_{m\to\infty}\phi_m'=\phi'$. Using Inequality \eqref{eq:poly_bound_phimj} and \cite[Lemma 5.1]{dPL14}, we conclude, by the theorem of dominated convergence,
\[
\begin{aligned}
	\li{m}\int_U \norm[\big]{D\Phi^m_n-D\Phi_n}_U^{p^*}\;\mathrm{d}\mu_1^{\Phi}=\li{m}\int_U \norm{\phi_m'(P_n(u))-\phi'(P_n(u))}_U^{p^*}\;\mathrm{d}\mu_1^{\Phi}=0.
\end{aligned}
\]We obtain the desired convergence \eqref{eq:conv_phi_m_n_p} from Item {App}($\Phi 4$), as $\sigma_1\geq \alpha_2\gamma$. 

Let $m,n\in\N$ be given, then obviously $0\leq \Phi^m_n$ and inequality \eqref{eq:phig4_one} holds true. So far, we have not used the special structure of $\phi$ described in the assertion. We need this to verify Inequality \eqref{eq:phig4_two}.
As $\psi$ grows less than quadratic, there are constants $a_2\in [0,\infty)$ and $a_3\in [0,2)$ such that for all $x\in\R$
\[\psi(x)\leq 1+a_2\abs{x}^{a_3}.\]Using Youngs inequality, there exists $a_4\in (0,\infty)$ with
\[\psi(x)\leq a_4+\frac{1}{q^*4\lambda_{1,1}}{x}^{2}\quad\text{for all}\quad x\in\R.\]
This implies \eqref{eq:phig4_two} from {App}($\Phi 4$) by
\[
\begin{aligned}
	\Phi^m_n(u)&\leq \Phi_n(u)=\int_0^1 a_1{P_n(u)}^2+\psi(P_n(u))\;\mathrm{d}\xi\\
	&\leq a_1\sum_{i,j=1}^n(u,d_i)_U(u,d_j)_U(d_i,d_j)_U+a_4\\&\quad+\frac{1}{q^*4\lambda_{1,1}}\sum_{i,j=1}^n(u,d_i)_U(u,d_j)_U(d_i,d_j)_U\\
	&= a_1\sum_{i=1}^n(u,d_i)^2_U+a_4+\frac{1}{q^*4\lambda_{1,1}}\sum_{i=1}^n(u,d_i)^2_U\\
	&\leq \Phi(u)+a_4+\frac{1}{q^*4\lambda_{1,1}}\norm{u}^2_U
\end{aligned}
\]and therefore 
\[
(q^*-1)\Phi^m_n(u)\leq q^*a_4+\frac{1}{4\lambda_{1,1}}\norm{u}^2_U+q^*\Phi(u).
\]
This ends the proof.
\end{proof}

\begin{rem}\label{rem_moregen_pot}
The special structure of $\phi$, described in \Cref{lem:appl_unb_pot_last}, is only used to verify Inequality \eqref{eq:phig4_two} from {App}($\Phi 5$). So different situations, in which this inequality is valid, can be imagined, e.g.~by perturbing $\Phi$ with a suitable finitely based function. For fixed $k\in\N$ and a function $\tilde{\phi}$ with the properties stated in \Cref{rem:MVT_bound_der_phi} we can incorporate perturbations with potentials of type $\tilde{\Phi}_k$ defined as in \Cref{def:pot_fin_n}. 
\end{rem}
As a consequence of the above considerations, we obtain the following result.
\begin{cor}\label{Cor:appl_ess_sess_Lfour}Suppose Item App($\Phi 3$) from Assumption \nameref{ass:pot_app} is valid and we are in the setting of \Cref{prop:ver_app_ass} and \Cref{lem:appl_unb_pot_last}. Then
\Cref{theo:ess_diss_phi_Lfour} is applicable and consequently essential m-dissipativity of $(L^{\Phi},\fbs{B_W})$ on $L^2(W;\mu^{\Phi})$ is established. Additionally, the semigroup generated by $(L^{\Phi},D(L^{\Phi}))$ is sub-Markovian and conservative. For each $f\in \fbs{B_W}$ it holds
\[
\begin{aligned}
	L^ {\Phi}f&=\tr\left[K_{22}\circ D_2^2f\right]+ \sum_{j=1}^{\infty} (\partx{j}K_{22}D_2f,d_j)_U- (v,(-{\partial_{\xi}^2})^{\alpha_2}K_{22}D_2f)_U\\
	&\quad-(u,(-{\partial_{\xi}^2})^{\alpha_1-\sigma_1}D_2f)_U+(v,(-{\partial_{\xi}^2})^{\alpha_2-\sigma_1}D_1f)_U\\
	&\quad-(\phi'(u),(-{\partial_{\xi}^2})^{-\sigma_1}D_2f)_U
\end{aligned}
\]
\end{cor}

\begin{rem}
We want to stress that \Cref{Cor:appl_ess_sess_Lfour} allows for situation where the gradient of $\Phi$, which is given by $D\Phi(u)=\phi'(u)=a_1u+\psi'(u)$ for $\mu_1$ almost all $u\in L^2(U;\mu_1)$, is not necessarily Lipschitz nor bounded. Indeed, the growth conditions for $\psi$ described in \Cref{lem:appl_unb_pot_last} includes a large class of non-linearities, that show up in the degenerate infinite dimensional stochastic differential equation \eqref{eq:sde} and \eqref{eq:sde_reac}.
\end{rem}

Assuming that Item App($\Phi 3$) from Assumption \nameref{ass:pot_app} holds true, there are several situations, where \Cref{Cor:appl_ess_sess_Lfour} can be applied. We give an example below.
\begin{exa}\label{lem_ex_comb_one}First choose $\sigma_2$ and $\sigma_3$ such that $\sigma_3\geq \frac{\min\lbrace\sigma_2,\alpha_2\rbrace}{2}$. Since $\alpha$ in \Cref{prop:ver_app_ass} can be chosen arbitrary large without imposing restrictions to the other parameters the existence of a suitable $\alpha$ is trivial. For our fixed set of parameters $\sigma_2$ and $\sigma_3$ we first choose $\beta$ and then $\gamma $ large enough so that the inequalities involving $\beta$ and $\gamma$ from \Cref{prop:ver_app_ass} are fulfilled. Finally, we can choose $\sigma_1$ such that the missing inequality  $2\sigma_1-\min\lbrace\sigma_2,\alpha_2\rbrace\geq \frac{1}{2}$ from \Cref{prop:ver_app_ass} and the inequality $\sigma_1\geq \alpha_2\gamma$ from \Cref{lem:appl_unb_pot_last} is valid. To apply  \Cref{Cor:appl_ess_sess_Lfour}, it is left to choose a potential as described in \Cref{lem:appl_unb_pot_last} or even more general as in \Cref{rem_moregen_pot}.

%


%
%
\end{exa}

\subsection{Hypocoercivity}

To show that the semigroup generated by the closure of $(L^{\Phi},\fbs{B_W})$ is hypocoercive, we strengthen our assumptions. Indeed, by means of \Cref{thm:inf-dim-hypoc-applied}, we need to check \nameref{ass:qc-negative-type}--\nameref{ass:inf-dim-u-poincare} with either \nameref{ass:inf-dim-matrix-bounded-ev} or \nameref{ass:inf-dim-matrix-bounded-ev_star}, as well as Assumption \nameref{ass:L_Phi-convex-reg-est} and \nameref{ass:ess-N_Phi}.

Assume that
\begin{align*}
\sigma_2\geq \alpha_2,\quad \sigma_3\geq \frac{3}{2}\alpha_2
\end{align*}and the potential $\Phi$ is as described in \Cref{lem:appl_unb_pot_last}.
In order to show that the semigroup generated  $(L^{\Phi},D(L^{\Phi}))$ is hypocoercive, we additionally assume that $\phi=\phi_1+\phi_2$, where $\phi_1$ is convex and $\phi_2$ is such that the corresponding $\Phi_2$ fulfils Item (Reg($\Phi_2$)) from \nameref{ass:L_Phi-convex-reg-est} and 
\[
2\sigma_1-\alpha_2\leq \frac{\alpha_1}{2}.
\]

\bigskip

Assumption \nameref{ass:qc-negative-type} is obviously valid. Moreover, $\lambda_{22,k}(v)\geq \lambda_k^{\alpha_2}$ for all $k\in\N$ and $v\in V$, gives us validity of Assumption \nameref{ass:inf-dim-v-poincare}.

We verify Assumption \nameref{ass:L_Phi-convex-reg-est} by means of the Moreau-Yosida approximation, using \Cref{ex:yoshida_one}, \Cref{ex:yoshida_two} and \Cref{rem:pot_approx_alter}, as well as Item (ii) from \Cref{rem:weak_conv_measure}.
%

Consequently, Assumption \nameref{ass:inf-dim-u-poincare} follows by Item (ii) from \Cref{rem:inf-dim-eigenvalue-estimates} and the fact that $2\sigma_1-\alpha_2\leq \frac{\alpha_1}{2}\leq \alpha_1$.


Item (i) and (ii) from assumption
\nameref{ass:inf-dim-matrix-bounded-ev} hold for 
\[
C_1=1\quad\text{and}\quad C_{2}(v)= 	1\quad v\in U,
\]by choosing the natural decomposition for $K_{22}$ into $K_1$ and $K_2$ induced by its definition. At this point, it is important that $\sigma_2\geq \alpha_2$ and $\sigma_3\geq \frac{3	}{2}\alpha_2$.
For Item (iii), note that\[ \alpha_{n}^{22}(v)\leq 2(2+\norm{v}_U)\lambda_{n}^{\sigma_3-\frac{\alpha_2}{2}},\] for all $v\in U$, which describes an $\ell^2$-sequence, since
\[
\sigma_{3}-\frac{\alpha_2}{2}\geq  {\alpha_2}> \frac{1}{2}.
\]
Moreover,
\[
\int_U  \|({\alpha_n^{22}}(v))_{n\in\N}\|^2_{\ell^2}\,\mu_2(\mathrm{d}v)\leq 4\|(\lambda_{n}^{\sigma_{3}-\frac{\alpha_2}{2}})_{n\in\N}\|			^2_{\ell^2}\int_U(2+\norm{v}_U)^2\,\mu_2(\mathrm{d}v)<\infty.
\] 

Lastly, we check Assumption \nameref{ass:ess-N_Phi}. 
%
Let $(B,D(B))$ be the closure of the operator $(-Q_1^{-1}C, \lin{d_1,d_2,\dots})$, then $(B,D(B))$ is self-adjoint with

\begin{equation}\label{B_self}
(Bu,u)_U=(-Q^{-\alpha_1-\alpha_2+2\sigma_1}u,u)_U\leq -{\lambda_1^{-\alpha_1-\alpha_2+2\sigma_1}}\norm{u}_{U}^2\;\text{for all}\; u\in \lin{d_1,d_2,\dots}.
\end{equation}
In the inequality above, we used $-\alpha_1-\alpha_2+2\sigma_1<0$, which is true, as we assume $2\sigma_1-\alpha_2\leq \frac{\alpha_1}{2}$.
By definition of $(B,D(B))$, Inequality \eqref{B_self} also holds for $u\in D(B)$. Moreover, 
$C=(-B)^{-\varepsilon}$ for $\varepsilon\defeq -\frac{2\sigma_1-\alpha_2}{2\sigma_1-\alpha_2-\alpha_1}$ and $\varepsilon\in (0,1)$ using $2\sigma_1-\alpha_2>0$ and $2\sigma_1-\alpha_2\leq \frac{\alpha_1}{2}$. Further, $(-B)^{-(1+\varepsilon)}=Q_1\in\loppt{U}$, $e^{-\Phi}\in L^p(U;\mu_1)$ for all $p\in[1,\infty)$ and finally $\Phi\in W_{C^{\frac12}}^{1,4}(U;\mu_1^{\Phi})$ by \Cref{rem:pot_approx_alter} and \Cref{prop:IBP_PHI_most_gen}. 

\cite[Theorem 3.2]{DP00} is consequently applicable and Assumption  \nameref{ass:ess-N_Phi} follows.

\bigskip

\subsection{The process}
Suppose that \Cref{Cor:appl_ess_sess_Lfour} is applicable. Then, as described in th beginning of the appendix in \Cref{sec:process}, there exists a right process, having infinite lifetime, with enlarged state space providing a solution to the martingale for $(L^{\Phi},D(L^{\Phi}))$ problem with respect to the equilibrium measure.
Next, suppose that
\begin{equation}\label{eq:L4eq}
\sigma_2, \sigma_3>\frac12,	\quad-\frac{\alpha_1}{2}+\sigma_1+\frac{\alpha_2}{2}>\frac12,\quad \frac{\alpha_1}{2}+\sigma_1-\frac{\alpha_2}{2}>\frac12\quad\text{and}\quad 2\sigma_1\geq \alpha_2.
\end{equation}
This implies that there exists a $\mu^{\Phi}$-invariant Hunt process
\[
\mathbf{M}=(\Omega,\mathcal{F},(\mathcal{F}_t)_{t\geq 0}, (X_t,Y_t)_{t\geq 0},(P_w)_{w\in W}),
\]
solving the martingale problem for $(L^{\Phi},D(L^{\Phi}))$ under ${P}_{{\mu}^{\Phi}}$ and with $P_{\mu^{\Phi}}$-a.s.~weakly continuous paths and infinite lifetime.


Indeed, Assumption  \nameref{ass:weak_sol} holds true, as $\lambda_{22,k}(v)\leq \lambda_k^{\alpha_2}+\lambda_k^{\sigma_2}+\lambda_k^{\sigma_3}$ for all $v\in V$. \nameref{ass:inf-dim-diff-bound} is checked by means of \Cref{rem:inf-dim-bounding-rho} and the assumption on the parameters described in \eqref{eq:L4eq}, since we already know that $\lambda_{k}^{\alpha_2}\leq \lambda_{22,k}(v)$ for each $v\in V$.

In summary, \Cref{thm:sto_ana_weak_sol_basis} is applicable and  $\textbf{M}$ is a stochastically and analytically weak solution with weakly continuous paths, in the sense of \Cref{thm:sto_ana_weak_sol_basis}, to the degenerate second order in time stochastic reaction-diffusion equation associated to $L^{\Phi}$. This degenerate second order in time stochastic reaction-diffusion equation corresponds to equation \eqref{eq:sde} and is given by

\begin{equation}
\begin{aligned}\label{eq:sde_reac}
	\mathrm{d}X_t&=(-{\partial_{\xi}^2})^{-\sigma_1+\alpha_2}Y_t\,\mathrm{d}t\\
	\mathrm{d}Y_t&=\sum_{i=1}^\infty \partx{i}K_{22}(Y_t)d_i-K_{22}(Y_t)(-{\partial_{\xi}^2})^{\alpha_2}Y_t-(-					{\partial_{\xi}^2})^{-\sigma_1+\alpha_1} X_t\\
	&\qquad-(-{\partial_{\xi}^2})^{-\sigma_1}\phi'(X_t)\,\mathrm{d}t
	+\sqrt{2K_{22}(Y_t)}\,\mathrm{d}W_t.
\end{aligned}
\end{equation}

\subsection{Summary}\label{sec:sum_reac_unbdgrad}

The table below summarizes the results we established in the previous sections.
It includes the combinations of parameters and conditions on the potential such that $(L^{\Phi},D(L^{\Phi}))$ is m-dissipative on $L^2(W;\mu^{\Phi})$, whereas we use as a standing assumption that the potential is as described in \Cref{lem:appl_unb_pot_last} or even more general in \Cref{rem_moregen_pot}. Moreover, recall that we assume the validity of item App($\Phi 3$) from Assumption \nameref{ass:pot_app} as a conjecture.The table also illustrates, under which conditions there is a $\mu^{\Phi}$-invariant Hunt process $\mathbf{M}$ providing a stochastically and analytically weak solution with $P_{\mu^{\Phi}}$-a.s.~weakly continuous paths and infinite lifetime for the infinite dimensional stochastic differential equation \Cref{eq:sde_reac}. Furthermore, it tells us when the semigroup $\sccs$ generated by $(L^{\Phi},D(L^{\Phi}))$ is hypocoercive.

\renewcommand*{\arraystretch}{1.9}
\begin{table}[h!]\label{tab:table3}
\begin{center}
	\caption{degenerate second order in time stochastic reaction-diffusion equation}
	\begin{tabular}{c|c|c|c}
		\hline
		\multicolumn{4}{l}{\makecell{\textbf{M-dissipativity and right process}\\ \textbf{solving the Martingale problem (enlarged state space)}}} \\
		\hline 
		\multicolumn{4}{c}{$\sigma_2\geq 0$, $\sigma_3\geq \frac{\min\lbrace\sigma_2,\alpha_2\rbrace}{2}$, then $\alpha,\beta,\gamma\geq 0$ according to \Cref{lem_ex_comb_one}}\\				 
		\hline 
		\multicolumn{4}{c}{$2\sigma_1-\min\lbrace\sigma_2,\alpha_2\rbrace\geq \frac{1}{2}$\quad\text{and}\quad$\sigma_1\geq \alpha_2\gamma$}\\
		\hline
		\makecell[c]{\textbf{$\mu^{\Phi}$-invariant Hunt process, weak sol.}\\ \textbf{with weakly cont.~paths, inf.~lifetime}} &	\multicolumn{3}{c}{\textbf{$\sccs$ hypocoercive}}\\
		\hline
		$\pm\frac{\alpha_1}{2}+\sigma_1\mp\frac{\alpha_2}{2}>\frac{1}{2}$
		&	\multicolumn{3}{c}{\makecell[c]{$\phi=\phi_1+\phi_2$ as in \Cref{rem_moregen_pot}, $\phi_1$ convex and \\ cor. $\Phi_2$ fulfils Item (Reg($\Phi_2$)) from \nameref{ass:L_Phi-convex-reg-est}		}	}\\
		\hline
		$\sigma_2,\sigma_3
		>\frac{1}{2}$&\multicolumn{3}{c}{ $2\sigma_1-\alpha_2\leq \frac{\alpha_1}{2}$}\\ 
		\hline
		$2\sigma_1\geq \alpha_2$&\multicolumn{3}{c}{$\sigma_2\geq \alpha_2$, $\sigma_3\geq \frac{3	}{2}\alpha_2$}\\  
		\hline
	\end{tabular}
\end{center}
\end{table}
\newpage
\begin{rem}
By using \Cref{ergodic}, we can combine the results from the table above to verify that $\textbf{M}$ is $L^2$-exponentially ergodic. Such a situation is e.g.~given by taking a suitable potential and assuming
$\alpha_1>1$, $-\alpha_1+4\alpha_2>2$, $\sigma_1=\frac{\alpha_1}{4}+\frac{\alpha_2}{2}$, $\sigma_2\geq\alpha_2$, $\sigma_3\geq\frac{3}{2}\alpha_2$. In this case we  choose $\alpha$ large, $\beta=\frac{1}{2\alpha_2}+1+\frac{\alpha_1-1}{16\alpha_2}$ and $\gamma=\frac{1}{4\alpha_2}+\frac12+\frac{\alpha_1-1}{8\alpha_2}$ to verify the inequalities in \Cref{prop:ver_app_ass} and \Cref{lem:appl_unb_pot_last}.
\end{rem}

\section{Appendix}\label{sec:process}

The results and statements in this chapter are a generalization of the ones already published in \cite{EG21_Pr}, where only additive noise was considered. We also extend the results from \cite{EG23}. Indeed, the results below are valid for situations where the gradient of the potential is not bounded, compare \Cref{ch:example_Lfour_ess_diss}.

Assume that we are in the setting described in \Cref{sec:inf_langevin_general}. Hence, $U$ and $V$ are two real separable Hilbert spaces, $W=U\times V$, $\Phi:U\rightarrow (-\infty,\infty]$ is normalized, bounded from below by zero and there is $\theta\in [0,\infty)$ such that $\Phi\in W_{Q_1^{\theta}}^{1,2}(U;\mu_1)$. Furthermore, $K_{12}\in \mathcal{L}(U;V)$, $K_{21}=K_{12}^*\in \mathcal{L}(V;U)$, $K_{22}(v)\in \lopp{V}$ for all $v\in V$ and $Q_1$ and $Q_2$ are the covariance operators of two infinite dimensional non-degenerate Gaussian measures $\mu_1$ and $\mu_2$, respectively.

We start with the situation where the infinite Kolmogorov operator $(L^{\Phi}D(L^{\Phi}))$, associated to the degenerate infinite dimensional stochastic Hamiltonian system with multiplicative noise, is m-dissipative.  
This holds, if we assume Assumption \nameref{ass:pot_app} to a apply \Cref{theo:ess_diss_phi_Lfour}. 
In this situation the strongly continuous contraction semigroup (resolvent) $\sccs$ ($(R_{\alpha}^{L^{\Phi}})_{\alpha >0}$), generated by $(L^{\Phi}D(L^{\Phi}))$, is sub-Markovian. Moreover, $(L^{\Phi}D(L^{\Phi}))$ is conservative and has $\mu^{\Phi}$ as an invariant measure.

Equip $W$ with the classical strong topology. A direct application of \cite[Theorem~2.2]{BBR06_Right_process} shows the existence of a Lusin topological space $(W_1,\mathcal{T}_1)$ with $W\subseteq W_1$ and $W\in \borel_{\mathcal{T}_1}(W_1)$ and a right process 
\[
\bar{\mathbf{M}}=(\bar{\Omega},\bar{\mathcal{F}},(\bar{\mathcal{F}}_t)_{t\geq 0}, (\bar{X_t},\bar{Y_t})_{t\geq 0},(\bar{P}_w)_{w\in W_1})
\]
with state space $W_1$ such that its resolvent, regarded on $L^2(W_1,\bar{\mu}^{\Phi})$, coincides with the resolvent $(R_{\alpha}^{L^{\Phi}})_{\alpha >0}$. Recall that $\bar{\mu}^{\Phi}$ is the measure on $(W_1,\borel_{\mathcal{T}_1}(W_1))$ extending $\mu^{\Phi}$ by zero on $W_1\setminus W$. Note that $(R_{\alpha}^{L^{\Phi}})_{\alpha >0}$ and $(L^{\Phi},D(L^{\Phi}))$ can also be considered on $L^2(W_1;\bar{\mu}^{\Phi})$, since $\bar{\mu}^{\Phi}(W_1\setminus W)=0$. Remember the equilibrium measure  $\bar{P}_{\bar{\mu}^{\Phi}}\defeq \int_{W_1} \bar{P}_w\,\bar{\mu}^{\Phi}(\mathrm{d}w)$. As $(L^{\Phi},D(L^{\Phi}))$ is conservative, we are able to apply \cite[Lemma 2.1.14]{Con2011} obtain that the process has infinite lifetime $\bar{P}_{\bar{\mu}^{\Phi}}$-a.s..

In the next proposition we establish that $\bar{\mathbf{M}}$ solves the martingale problem for the operator $(L^{\Phi},D(L^{\Phi}))$ considered on $L^2(W_1;\bar{\mu}^{\Phi})$ with respect to $\bar{P}_{\mu^{\Phi}}$. 
%
\begin{prop}\label{rem:inf-dim-mart-prob}
The right process $\bar{\mathbf{M}}$ with state space $W_1$ solves the martingale problem for $(L^{\Phi},D(L^{\Phi}))$ considered on $L^2(W_1;\bar{\mu}^{\Phi})$ with respect to $\bar{P}_{\mu^{\Phi}}$.
\end{prop}

\begin{proof}

We use the exact same arguments as in \cite[Proposition~1.4]{BBR06} to show that $\bar{\mathbf{M}}$ solves the martingale problem for $(L^{\Phi},D(L^{\Phi}))$ with respect to $\bar{P}_{\mu^{\Phi}}$. The assertion in \cite[Proposition~1.4]{BBR06} is stated for $\mu^{\Phi}$-standard right processes but the argumentation works analogously for right processes.

\end{proof}For the following considerations we consider $(L^{\Phi},D(L^{\Phi}))$ on $L^2(W;\mu^{\Phi})$.
Before we continue to construct a more regular process, we provide another core for $(L^{\Phi},D(L^{\Phi}))$. This core is essential to apply \cite[Theorem 1.1]{BBR06}.
\begin{lem}{\cite[Theorem 4.3]{EG21_Pr}}\label{lem:count_Q_Algebra}
There exists a countable $\Q$-algebra $\mathcal{A}\subseteq \mathcal{F}C_c^{\infty}(B_W)$, which is core for $(L^{\Phi},D(L^{\Phi}))$ and also separates the points of $W$.
\end{lem}

For the rest of this chapter, let $\mathcal{T}$ denote the weak topology on $W$. Recall that the Borel sigma algebra, with respect to the strong and weak topology on $W$, coincide and are equal to the sigma algebra generated by $\mathcal{A}$.
Next, we establish that \[F_k\defeq \{w\in W: \|w\|_W\leq k\}\] defines a $\mu^{\Phi}$-nest of $\mathcal{T}$-compact sets. Our strategy is based on \cite[Section 5]{BBR06}, where $(F_k)_{k\in\N}$ was used to construct martingale solutions for generators associated to dissipative stochastic differential equations on Hilbert spaces.
We start with the introduction of a new assumption.

\begin{cond}{\textbf{K5}}\label{ass:inf-dim-diff-bound}\leavevmode 
\begin{enumerate}[(i)]
	\item There exists $\rho\in L^1(W;\mu^{\Phi})$ such that for each $n\in\N$, $\rho_n$ defined via $\rho_n(u,v)\defeq 					\rho(P_n^Uu,P_n^Vv)$ is in $L^1(W;\mu^{\Phi})$. Moreover, the sequence $(\rho_{m^K(n)})_{n\in\N}$ converges to $\rho$ in $L^1(W;						\mu^{\Phi})$ as $n\to\infty$ and 
	\[
	\begin{aligned}
		&\quad (P_{m^K(n)}^V v,Q_2^{-1}K_{22}(v)P_{m^K(n)}^Vv)_V+(D\Phi(u),K_{21}P_{m^K(n)}^V v)_U
		\\&\quad+	(P_{m^K(n)}^U u,Q_1^{-1}K_{21}P_{m^K(n)}^V v)_U- (Q_2^{-1}K_{12}P_{m^K(n)}^U u,P_{m^K(n)}^V v)_V\\&	\leq \rho_{m^K(n)}(u,v)
	\end{aligned}
	\]
	for all $n\in\N$ and $(u,v)\in W$.
	\item There exist $a,b\in\N$ such that for all $n\in\N$ and $v\in V$
	\[
	\sum_{j=1}^n\norm{\party{j}K_{22}(v)e_j}_V\leq a(1+ \norm{v}^b_V).
	\]
\end{enumerate}
\end{cond}
\begin{prop}\label{prop:inf-dim-assoc-process}
Let Assumption \nameref{ass:inf-dim-diff-bound} be valid. Then $(F_k)_{k\in\N}$ is $\mu^{\Phi}$-nest of $\mathcal{T}$-compact sets for $					(L^{\Phi},D(L^{\Phi}))$.
\end{prop}
\begin{proof}
By the Theorem of Banach-Alaoglu $F_k$ is $\mathcal{T}$-compact for all $k\in\N$.
It remains to prove that $({F_k})_{k\in\N}$, is a $\mu^{\Phi}$-nest.
For notional purposes, we write $N(u,v)\defeq \|(u,v)\|_W^2$ and $N_n(u,v)=N(P_n^U u,P_n^V v)$ for $n\in\N$.
We only consider those $n\in\N$ that satisfy $n=m^K(n)$, which provide an increasing sequence.
By an approximation argument with a sequence of smooth cut-off functions, we see that $N_n\in D(L^{\Phi})$, compare also the proof of \cite[Theorem 4.3]{EG21_Pr}, with 
\[
\begin{aligned}
	&\quad \frac12 L^{\Phi}N_n(u,v)\\&=\tr[K_{22}(P_n^V v)]
	+\sum_{j=1}^n (\party{j}K_{22}(v)e_j,P_n^V v)_V
	-(P_n^V v,Q_2^{-1}K_{22}(v)P_n^V v)_V \\
	&\qquad-(D\Phi(u),K_{21}P_n^V v)_U- (P_n^U u,Q_1^{-1}K_{21}P_n^V v)_U
	+(P_n^V v,Q_2^{-1}K_{12}P_n^U u)_V \\
	&\geq \sum_{j=1}^n (\party{j}K_{22}(v)e_j,P_n^V v)_V
	-\rho_n(u,v).
\end{aligned}
\]
Using the second item from Assumption \nameref{ass:inf-dim-diff-bound}, we find $a,b\in\N$ with
\[ 
\begin{aligned}
\Big|\sum_{j=1}^n (\party{j}K_{22}(v)e_j,P_n^V v)_V \big|\leq a(1+ \norm{v}^b_V)\norm{P_n^V v}_V\eqdef  h_n(v).
\end{aligned}
\]
Note that $(h_n)_{n\in\N}$ converge in $L^1(V;\mu_2)$. In summary, this implies, when setting
\[
g_n(u,v)\defeq 2\big( \rho_n(u,v)+h_n(v)+\frac{1}{2} N_n(u,v)\big),
\]
that
\begin{equation}\label{eq:inf-dim-level-set-bound}
(\operatorname{Id}-L^{\Phi})N_n\leq g_n\quad \text{$\mu^{\Phi}$-a.e.}
\end{equation}
for all $n\in\N$ with $n=m^K(n)$.
We proceed as in \cite[Proposition 5.5]{BBR06}.
\end{proof}

\begin{rem}\label{rem:inf-dim-bounding-rho}
To verify the first item of Assumption we argue as in \cite[Remark 5.5]{EG23}
%
and check that 
\[
\sum_{i,j=1}^{\infty} \abs{(Q_1^{-\frac{1}{2}}d_i,K_{21}Q_2^{\frac{1}{2}}e_j)_U}<\infty\quad\text{and}\quad \sum_{i,j=1}^{\infty} 			\abs{(Q_2^{-\frac{1}{2}}e_i,K_{12}Q_1^{\frac{1}{2}}d_j)_V}<\infty.
\]

We secondly verify that $({(D\Phi,K_{21}P_n^V v)_U})_{n\in\N}$ converges in $L^1(W;\mu^{\Phi})$, in the case that $c_{\theta}\defeq \norm{Q_1 ^{\-\theta}K_{21}}_{\mathcal{L}(V;U)}<\infty$. Indeed, for each $m,n\in\N$ with $m\leq n$ and $\tilde{v}\in V$, we estimate 
\[
\begin{aligned}
	&\quad \int_W\abs{(D\Phi,K_{21}(P_n^V-P_m^V)v)_V}\,\mathrm{d}\mu^{\Phi}\\&\leq\norm{Q_1^{\theta}D\Phi}_{L^{2}(\mu_1)}\left(\int_V \norm{Q_1^{-\theta}K_{21}(P_m^V v-P_n^V v)}_V^2\,\mathrm{d}\mu_2\right)^{\frac12}\\
	&\leq \norm{Q_1^{\theta}D\Phi}_{L^{2}(\mu_1)}c_{\theta}\left(\int_V \norm{(P_m^V v-P_n^V v)}_V^2\,\mathrm{d}\mu_2\right)^{\frac12}\\
	&\leq \norm{Q_1^{\theta}D\Phi}_{L^{2}(\mu_1)}c_{\theta}\left(\sum_{i=m+1}^n\lambda_{2,i}\right)^{\frac12}.
\end{aligned}
\]As $Q_2$ has finite trace, we get that $({(D\Phi,K_{21}P_n^V v)_U})_{n\in\N}$ is a Cauchy sequence and therefore convergent in $L^1(W;\mu^{\Phi})$. In the above argumentation, boundedness of $Q_1^{\theta}D\Phi$ is not involved.

Lastly, suppose $K_{22}$ is diagonal, as described in \Cref{rem:too-diagonal} and that there is a sequence $(c_k)_{k\in\N}\in \ell^1(\N)$ such that $\lambda_{22,k}(v)\leq c_k$ for all $k\in\N$ and $v\in V$. Then, \[W\ni(u,v)\mapsto(P_{m^K(n)}^V v,Q_2^{-1}K_{22}(v)P_{m^K(n)}^Vv)_V\in\R \]is a Cauchy sequence in $L^1(W;\mu^{\Phi})$, using \Cref{lem:Gaussianmoments}. This consideration is especially useful, when combined with Assumption \nameref{ass:inf-dim-matrix-bounded-ev}.

\end{rem}
For the rest of this section, we assume that Assumption \nameref{ass:inf-dim-diff-bound} holds. As a direct consequence of [4, Proposition 1.4] and \cite[Lemma 2.1.8]{Con2011} we obtain the following proposition.
\begin{prop}\label{prop:inf-dim-path-prop}
There exists a $\mu^{\Phi}$-invariant Hunt process \[\mathbf{M}=(\Omega,\mathcal{F},(\mathcal{F}_t)_{t\geq 0}, (X_t,Y_t)_{t\geq 0},(P_w)_{w\in W})\] with $P_{\mu^{\Phi}}$-a.s.~weakly continuous paths and infinite lifetime, which is associated with $\sccs$ ($\sccr$) and solving the martingale problem for $(L^{\Phi},D(L^{\Phi}))$ with respect to $P_{\mu^{\Phi}}$. 
Further, if $f^2\in D(L^{\Phi})$ with $L^{\Phi}f\in L^4(W;\mu^{\Phi})$, then
\[
N_t^{[f],L^{\Phi}}\defeq \left(M_t^{[f],L^{\Phi}}\right)^2 - \int_0^t L^{\Phi}(f^2)(X_s,Y_s)-(2fL^{\Phi}f)(X_s,Y_s)\,\mathrm{d}s,
\quad t\geq 0
\]
describes an $(\mathcal{F}_t)_{t\geq 0}$-martingale.
\end{prop}

In \cite[Chapter 5 Proposition 4.6]{KS}, it is described how to construct a weak solution to a finite dimensional stochastic differential equation starting from a (local) martingale solution.
Even though this result cannot be translated directly into our infinite dimensional setting, the considerations below are inspired by the finite dimensional one. For the definition of quadratic (co)variation and increasing processes we refer to \cite[Section 1.5]{KS}.

\begin{prop}{\cite[Proposition 5.6]{EG23}}\label{prop:inf-dim-component-solutions}
For any $i\in\N$, the real-valued processes $(X_t^i)_{t\geq 0}$ and $(Y_t^i)_{t\geq 0}$ defined by $X_t^i=(X_t,d_i)_U$ and 						$Y_t^i=(Y_t,e_i)_V$ satisfy $P_{\mu^{\Phi}}$-a.s.
\begin{equation}\label{eq:inf-dim-sde-components}
	\begin{aligned}
		X_t^i-X_0^i &= \int_0^t (Y_s,Q_2^{-1}K_{12}d_i)_V\,\mathrm{d}s \qquad\text{ and } \\
		Y_t^i-Y_0^i &= \int_0^t \left(\sum_{k=1}^{\infty}\party{k} K_{22}(Y_s)e_k,e_i\right)_V- (Y_s,Q_2^{-1}K_{22}(Y_s)e_i)_V - (X_s,Q_1^{-1}K_{21}e_i)_U \\
		&\qquad-(D\Phi(X_s),K_{21}e_i)_U\,\mathrm{d}s
		+ M_t^{[g_i],L^{\Phi}}.
	\end{aligned}
\end{equation}
Above, $(M_t^{[g_i],L^{\Phi}})_{t\geq 0}$ is a continuous $(\mathcal{F}_t)_{t\geq 0}$-martingale such that for $i,j\in\N$ the quadratic covariation of $M^{[g_i],L^{\Phi}}$ and $M^{[g_j],L^{\Phi}}$ fulfills \[[M^{[g_i],L^{\Phi}},M^{[g_j],L^{\Phi}}]_t=2\int_0^t(e_i,K_{22}(Y_s)e_j)_V\,\mathrm{d}s\quad \text{for all}\quad t\in [0,\infty).\]
\end{prop}
			\begin{cond}{\textbf{K6}}\label{ass:weak_sol}
				There is a a non-negative function $k_{22}$ in $L^1(V;\mu_2)$ such that for all $v\in V$
				\[
				\tr[K_{22}(v)]\leq k_{22}(v).
				\]
			\end{cond}
			Before we finally state the final theorem in this chapter, it is useful to define $D(Q_2^{-1}K_{12})$, $D(Q_1^{-1}K_{21})$ and $D(Q_1^{-\theta}K_{21})$ according to \Cref{def:OU_with_pot}, as well as \[D(Q_2^{-1}K_{22})\defeq \{v\in V\mid\; K_{22}(\tilde{v})(v)\in D(Q_2^{-1})\;\text{for all}\; \tilde{v}\in V\}.\]
			
			\begin{thm}\cite[Theorem 5.9]{EG23}\label{thm:sto_ana_weak_sol_basis}
				there is a cylindrical 				Wiener process $(W_t)_{t\geq 0}$ on $V$ such that $P_{\mu^{\Phi}}$-a.s., we have for all $j\in\N$,
				%
				%
				$P_{\mu^{\Phi}}$-a.s.~for every element $\nu_1\in D(Q_2^{-1}K_{12})$ and every element $\nu_2\in D(Q_1^{-1}	K_{21})\cap D(Q_2^{-1}K_{22})\cap D(Q_1^{-\theta}K_{21})$,
				\begin{equation}\label{eq:inf-dim-sde}
					\begin{aligned}
						(X_t-X_0,\nu_1)_U &= \int_0^t (Y_s,Q_2^{-1}K_{12}\nu_1)_V\,\mathrm{d}s\quad\text{and} \\
						(Y_t-Y_0,\nu_2)_V &=\int_0^t \left(\sum_{k=1}^\infty \party{k}K_{22}(Y_s)e_k,\nu_2\right)_V
						- (Y_s,Q_2^{-1}K_{22}(Y_s)\nu_2)_V\\&\qquad-(X_s,Q_1^{-1}K_{21}\nu_2)_U-(Q_1^{\theta}D\Phi(X_s),Q_1^{-\theta}K_{21}\nu_2)_V\,\mathrm{d}s
						\\&\qquad+\left(\int_0^t \sqrt{2K_{22}(Y_s)}\,\mathrm{d}W_s,\nu_2\right)_V.
					\end{aligned}
				\end{equation}
			\end{thm}

			We end this section with an $L^2$-exponential ergodicity result for the stochastically and analytically weak solution, provided by \Cref{thm:sto_ana_weak_sol_basis}. A similar result was already established in a manifold setting in \cite[Corollary 5.2]{GROTHAUS202222}.
			\begin{cor}\cite[Corollary 6.12]{EG21_Pr}\label{ergodic} Assume that the assumptions
				\nameref{ass:qc-negative-type}-\nameref{ass:inf-dim-u-poincare}, 
				(with either \nameref{ass:inf-dim-matrix-bounded-ev} or \nameref{ass:inf-dim-matrix-bounded-ev_star}), \nameref{ass:ess-N_Phi}, 	\nameref{ass:L_Phi-convex-reg-est} and \nameref{ass:pot_app} hold true. 
				Let $\theta_1\in (1,\infty)$ and $\theta_2\in (0,\infty)$ be the constants determined by \Cref{thm:inf-dim-hypoc-applied}. If the Assumptions \nameref{ass:inf-dim-diff-bound} and \nameref{ass:weak_sol} hold true, then the process \[\mathbf{M}=(\Omega,\mathcal{F},(\mathcal{F}_t)_{t\geq 0}, (X_t,Y_t)_{t\geq 0},				(P_w)_{w\in W}),\]constructed in \Cref{prop:inf-dim-path-prop}, is a $\mu^{\Phi}$-invariant Hunt process with $P_{\mu^{\Phi}}$-a.s.~weakly continuous paths, infinite lifetime and solving Equation \eqref{eq:inf-dim-sde-components}.
				Moreover, for all $t\in (0,\infty)$ and $g\in L^2(W;\mu^{\Phi})$, it holds
				\begin{align*}
					&\quad \left\Vert\frac{1}{t}\int_0^tg(X_s,Y_s)\mathrm{d}s-\mu^{\Phi}(g)\right\Vert_{L^2(P_{\mu^{\Phi}})}\\&\leq \frac{1}{\sqrt{t}}\sqrt{\frac{2\theta_1}{\theta_2}\left(1-\frac{1}{t\theta_2}(1-e^{-t\theta_2})\right)}\norm{g-\mu^{\Phi}(g)}_{L^2(\mu^{\Phi})}.
				\end{align*}We call a solution $\mathbf{M}$ with this property $L^2$-exponentially ergodic, i.e.~ergodic with a rate that corresponds to exponential convergence of the corresponding semigroup.
			\end{cor}
			
			\begin{rem}
				We can formulate a similar statement as in \Cref{ergodic} in terms of the right process from \Cref{rem:inf-dim-mart-prob}. Indeed for the computations in \Cref{ergodic} we only need the Markov property and that the semigroup of $(L^{\Phi},D(L^{\Phi}))$ is associated with the transition semigroup of the process.
			\end{rem}
			We end this chapter with a remark concerning the optimality of the convergence rate from \Cref{ergodic}.
			\begin{rem}From \Cref{ergodic} above, we follow, that time average converges to space average in $L^2(P_{\mu^{\Phi}})$ with rate $t^{-\frac{1}{2}}$. If the spectrum of $(L^{\Phi},D(L^{\Phi}))$ contains a negative eigenvalue $-\kappa$ with corresponding eigenvector $g$, then this rate is optimal. Indeed, by a similar reasoning as in the calculation above, we then get for all $t\in (0,\infty)$
				\begin{align*}
					\left\Vert\frac{1}{t}\int_{0}^tg(X_s,Y_s)\mathrm{d}s\right\Vert_{L^2(P_{\mu^{\Phi}})}
					=\frac{1}{\sqrt{t}}\sqrt{\frac{2}{\kappa}\left(1-\frac{1}{t\kappa}(1-e^{-t\kappa})\right)}\norm{g}_{L^2(\mu^{\Phi})}.
				\end{align*}Equality above holds, as the application of the Cauchy-Schwarz inequality is not necessary. Moreover, note that $\mu^{\Phi}(g)=\frac{1}{-\kappa}\mu^{\Phi}(L_{\Phi}g)=0$. 
			\end{rem}

			
			\begin{ack}
				The first named author gratefully acknowledges financial
				support in the form of a fellowship from the “Studienstiftung des deutschen Volkes”.
			\end{ack}
			%
			
			
			\bibliographystyle{abbrv}
			\bibliography{sources_ben.bib}
			
			
			
			
			
			
			
			
			
		\end{document}